\documentclass[reqno]{amsart}
\usepackage{amsfonts, color}
\usepackage{xcolor}
\usepackage{geometry}
\usepackage[normalem]{ulem}
\usepackage{relsize}
\usepackage{amstext, amsthm, amssymb, amsmath, mathtools}
\usepackage{thmtools}
\usepackage{mathrsfs}
\usepackage{bbm}
\usepackage{graphicx}
\usepackage{subfloat}
\usepackage{subcaption}
\usepackage{color}
\usepackage{mathrsfs} 
\usepackage[hypertexnames=false]{hyperref}
\usepackage{algorithm}
\usepackage{algorithmic}
\usepackage{booktabs}

\numberwithin{algorithm}{section}
\usepackage[capitalize,nameinlink,noabbrev]{cleveref}
\usepackage[style=ieee, sorting=nyt]{biblatex}
\usepackage[textsize=tiny]{todonotes}
\usepackage{lineno}

\usepackage{tikz}
\usetikzlibrary{positioning, arrows.meta, calc, fit}

\tikzset{
  box/.style = {
    rectangle,
    draw,
    rounded corners,
    minimum height=10mm,
    text width=1.2cm,     % fixed equal width for all boxes
    align=center,
    font=\sffamily
  },
  arrow/.style = {-{Stealth[length=6pt,width=6pt]}, thick}
}

\newcommand{\diff}{\mathrm{d}}
\newcommand{\dd}{\mathrm{d}}

\newcommand{\law}{\mathrm{Law}}
\newcommand{\dv}{\mathrm{div}}

\newcommand{\R}{\mathbb{R}}
\newcommand{\RR}{\mathbb{R}}
\newcommand{\EE}{{\mathbb E}}

\newcommand{\Xalpha}{{X^\alpha}}

\newcommand{\Xf}{{\mathfrak{X}}}

\newcommand{\mTa}{{\mathcal T}^\alpha}
\newcommand{\NN}{{\mathbb N}}

\newcommand{\Normal}{\mathcal{N}}

\newcommand{\mc}[1]{\mathcal{#1}}

\newtheorem{theorem}{Theorem}[section]
\newtheorem{lemma}[theorem]{Lemma}
\newtheorem{assum}[theorem]{Assumption}
\newtheorem{corollary}[theorem]{Corollary}

\newtheorem{remark}[theorem]{Remark}

\usepackage{graphicx}
\usepackage{caption}
\usepackage{calc}
\usepackage{ragged2e}
\usepackage{setspace}
\newcommand{\FundingLogos}{%
  \raisebox{0pt}{\includegraphics[height=1.5cm]{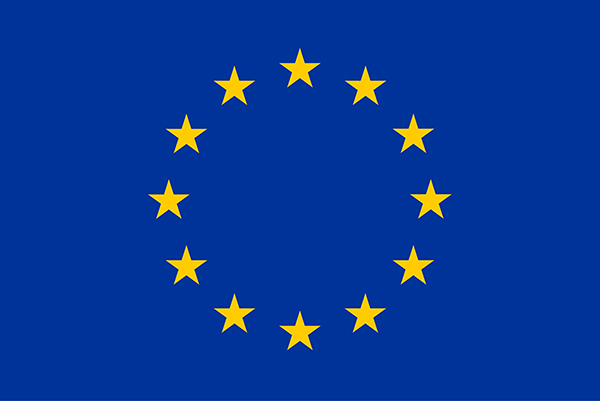}}%
  \hspace{1em}%
  \raisebox{0pt}{\includegraphics[height=1.5cm]{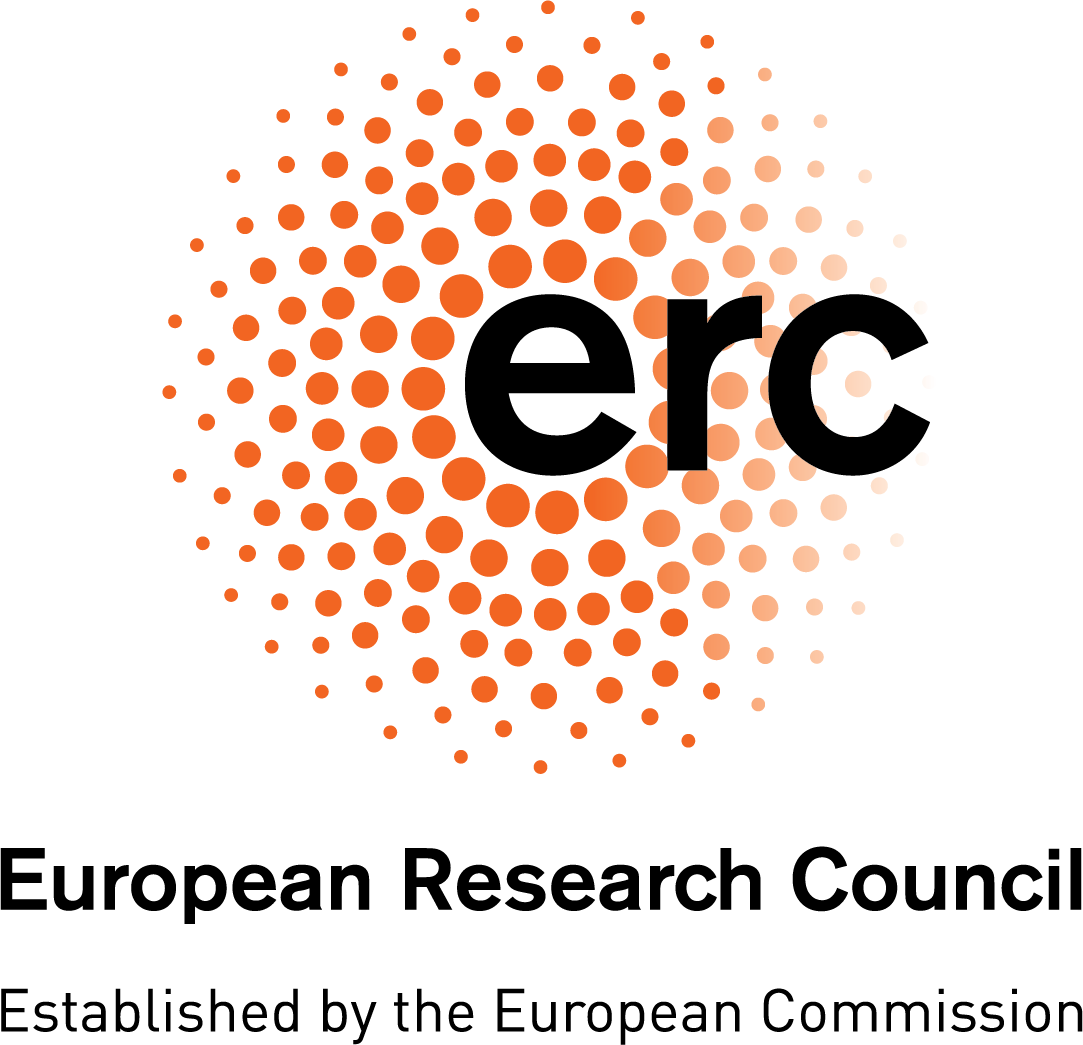}}%
}

\usepackage{tcolorbox}
\tcbset{colback=white,boxrule=0.5pt,arc=2pt,left=4pt,right=4pt,top=4pt,bottom=4pt}

\usepackage[foot]{amsaddr}

\title[From Consensus-Based Optimization to Evolution Strategies]{From Consensus-Based Optimization to Evolution Strategies: Proof of Global Convergence}
\author{Massimo Fornasier$^{1, 2, 3}$}
\email{massimo.fornasier@cit.tum.de}
\author{Hui Huang$^{4}$}
\email{huihuang1@hnu.edu.cn}
\author{Jona Klemenc$^{1,2}$}%\footrecall{tum}\footrecall{mcml}}
\email{jona.klemenc@tum.de}
\author{Greta Malaspina$^{5}$}%\footrecall{tum}\footrecall{mcml}}
\email{greta.malaspina@unifi.it}

\address{$^1$CIT School, Technical University of Munich, Garching bei M\"unchen, Germany.}
\address{$^2$Munich Center for Machine Learning (MCML), Munich, Germany.}
\address{$^3$Munich Data Science Institute (MDSI), Munich, Germany.}
\address{$^4$School of Mathematics, Hunan University, Changsha, China.}
\address{$^5$Department of Industrial Engineering, University of Florence, Florence, Italy.}
\address{\phantom{$^5$}Member of the INDAM Research Group GNCS.}
\date{November 18, 2026}
\begin{document}

\maketitle

 \begin{abstract}

Consensus-based optimization (CBO) is a powerful and versatile zero-order multi-particle method designed to \emph{provably} solve high-dimensional global optimization problems, including those that are genuinely nonconvex or nonsmooth. The method relies on a balance between stochastic exploration and contraction toward a consensus point, which is defined via the Laplace principle as a proxy for the global minimizer.

In this paper, we introduce new CBO variants that address practical and theoretical  limitations of the original formulation of this novel optimization methodology. First, we propose a model called {\bf \(\delta\)-CBO}, which incorporates nonvanishing diffusion to prevent premature collapse to suboptimal states. We also develop a numerically stable implementation, the \textbf{Consensus Freezing scheme}, that remains robust even for arbitrarily large time steps by freezing the consensus point over time intervals. We connect these models through appropriate asymptotic limits. Furthermore, we derive from the Consensus Freezing scheme by suitable time rescaling and asymptotics a  further  algorithm, the \textbf{Consensus Hopping scheme}, which can be interpreted as a form of \((1,\lambda)\)-Evolution Strategy. For all these schemes, we characterize for the first time the invariant measures and establish global convergence results, including exponential convergence rates.
\end{abstract}

{\noindent\small{\textbf{Keywords:} global optimization, 
    consensus-based optimization, MPPI, evolution strategies,
			nonlinear Fokker--Planck equations}}\\
	
	{\noindent\small{\textbf{AMS subject classifications:} 35Q84, 35Q91, 65K10, 35Q93, 37N40, 60H10, 65C35}}

\tableofcontents

\section{Introduction}

{Convex optimization is tractable: cutting-plane and interior-point methods reach
accuracy $\epsilon$ at a cost polynomial in the dimension $d$ and in
$\log(1/\epsilon)$~\cite{nemirovski1983problem,nesterov1994interior,nesterov2018lectures}.
Nonconvex optimization is intrinsically harder: NP-hard problems admit exact
continuous nonconvex reformulations,\footnote{By the Motzkin--Straus
theorem~\cite{motzkin1965maxima}, a graph with adjacency matrix $A$ and clique
number $\omega$ satisfies $\max_{x\in\Delta_d} x^\top A x = 1-1/\omega$, where
$\Delta_d$ is the unit simplex of $\mathbb{R}^d$; maximizing this indefinite
quadratic over a polytope is thus as hard as \textsc{Max-Clique}, which is
NP-hard~\cite{karp1972reducibility}. Even certifying local optimality is NP-hard
in general~\cite{murty1987some}.}
and for black-box $L$-Lipschitz functions on a compact subset of $\mathbb{R}^d$
worst-case accuracy $\epsilon$ requires order $(L/\epsilon)^d$
evaluations~\cite{nemirovski1983problem,nesterov2018lectures,malherbe2017global}.
These are, however, worst-case statements: between convexity and the hardest
instances lies a graded spectrum of structured classes, e.g.,
Polyak--{\L}ojasiewicz~\cite{polyak1963gradient,karimi2016linear},
star-convex~\cite{nesterov2006cubic,hinder2020near}, or smooth
dissipative~\cite{raginsky2017non,xu2018global} functions, whose geometry allows
global minimizers to be reached with explicit, problem-dependent rates.
Yet, in practice, general nonconvex problems are still routinely addressed with gradient- or Hessian-based methods conceived for the
convex setting, or with derivative-free direct and random search, settling for
stationarity, local optimality, or purely non-quantitative 
convergence in probability~\cite{ghadimi2013stochastic,carmon2020lower,conn2009introduction,audet2006mesh,solis1981minimization}.
Given the vast landscape of nonconvex optimization, aiming at genuine global optima with convergence rates for structured problems is thus a legitimate
ambition, not one that should be ruled out a priori by opposing NP-hardness or the curse of dimensionality.
}\\
In this paper, we introduce,  analyze, and connect a class of easy-to-implement, efficient, and inherently scalable (highly parallelizable) methods that can {\it provably} solve some of the most challenging problems in continuous {\it global} optimization, including those that are highly nonconvex and nonsmooth, while offering guaranteed convergence rates.\\
The class of methods we analyze in a unified manner includes Consensus-Based Optimization (CBO){\cite{pinnau2017consensus}}, Model Predictive Path Integral (MPPI) \cite{MPPI}, and paves the way to the analysis of further Evolution Strategies (ES){\cite{AugerHansen2005}, such as CMA-ES. The folklore about the latter methods, MPPI and Evolution Strategies, is that they are empirically extremely successful in global optimization \cite{COCO21}, but they are also inherently local \cite{DiouaneGrattonVicente2015,DiouaneGrattonVicente2015constrained}\footnote{Here, for instance, the authors of \cite{DiouaneGrattonVicente2015,DiouaneGrattonVicente2015constrained} prove convergence to Clarke-stationary points and not to global minimizers.}, as sort of stochastic relaxations of gradient descent \cite{Ollivier2017}.
\\
We prove that actually all these methods are {\it globally} convergent to global minimizers\footnote{Let us stress that, differently from previous literature which may attribute "global convergence" also nonconvex optimization methods converging to steady states, our terminology of global convergence means {\it convergence to global minimizers regardless of the choice of initial conditions.}} {for certain classes of objective functions, with minimal regularity and growth conditions.}
\\
Our findings represent  substantial advances by  delivering definitive theoretical descriptions and provide simple and practically relevant guidance for achieving the most effective implementations.
\\
Traditional optimization methodologies, most notably gradient-based schemes and quasi-Newton methods, depend heavily on local derivative information to steer the search toward minimizers. This intrinsic locality, however, becomes a major obstacle when facing genuine nonconvex or nonsmooth objectives $f$, which represent some of the most difficult classes of problems in modern optimization:
\begin{equation}\label{eq:globopt}
(\arg) \min_{x \in \mathbb R^d} f(x).
\end{equation}
{Nonconvex and nonsmooth optimization arises across a broad range of high-impact applications.  In engineering ---spanning control, robotics, design, and chemical processes--- nonlinear physical models naturally lead to nonconvex formulations. In computer vision and imaging, nonconvex models and nonsmooth regularizers enable tasks such as image denoising (e.g. nonconvex total variation), computer-generated holography, and depth fusion.
In machine learning, nonconvex nonsmooth optimization underpins the training of deep neural networks with nonsmooth activations (e.g., ReLU), sparsity-promoting priors such as the $\ell_1$-norm, and the highly nonconvex minimax problems central to GANs.
Beyond these domains, such optimization problems also appear in mathematical finance (e.g., complex portfolio design), computational biology (e.g., protein folding), and mechanics involving nonsmooth interactions.
}

Metaheuristic approaches, such as Simulated Annealing, Genetic Algorithms, Particle Swarm Optimization, Ant Colony Optimization \cite{GendPotv13,blum2003metaheuristics}, and a wide array of Evolution Strategies \cite{AugerHansen2005}, have achieved impressive performance in practice and are routinely employed and assessed using standardized benchmark environments \cite{COCO21}. Yet, despite their empirical strength, the majority of these techniques offer only limited theoretical insight: robust convergence guarantees and quantitative convergence rates to global minima are typically missing, especially in high-dimensional settings involving  nonconvex objectives. Establishing methods that can {\it provably} bypass the difficulties raised by nonconvexity and locality would open the door to solving problems that currently remain inaccessible to {\it rigorous} mathematical analysis.

A significant milestone for global optimization was reached in 2017 with the introduction of Consensus-Based Optimization (CBO) \cite{pinnau2017consensus}. CBO is a multi-agent, zero-order scheme whose design draws inspiration from the heuristics behind Simulated Annealing and Particle Swarm Optimization.
{As with many other metaheuristic paradigms, the versatility and conceptual simplicity of CBO have been instrumental to its rapid adoption across a wide range of application domains, including machine learning, engineering, and scientific computing. Its uses encompass global optimization under constraints \cite{fornasier2020consensus_hypersurface_wellposedness,fornasier2021consensus,bae2022constrained,beddrich2024,B1-carrillo2021consensus,B1-borghi2021constrained,B1-ha2021emergent}, the minimization of objective functions exhibiting multiple global minima \cite{bungert2022polarized,B1-fornasierSun25,huang2025faithful}, multi-objective optimization tasks \cite{borghi2022consensus,borghi2022adaptive,totzek-multio}, stochastic or noisy objectives \cite{bellavia25,bonandin2024}, high-dimensional learning problems \cite{B1-carrillo2019consensus}, and sampling from probability laws \cite{B1-carrillo2022sampling}.
The method has been continuously enriched through numerous extensions, such as momentum-based variants \cite{CiCP-31-4}, formulations incorporating memory \cite{totzeck2020consensus,riedl2022leveraging,huang_self-interacting_2025}, second-order dynamics \cite{2ndorder,B1-cipriani2021zero,B1-grassi2021mean,grassi2020particle,hui23}, optimal-control-inspired accelerations \cite{huang2024fast}, mirror-descent reinterpretations \cite{bungert2025mirror}, random-batch implementations \cite{ko2022convergence}, truncated-noise perturbations \cite{fornasier2023consensus1}, and models with jump processes \cite{kalise2022consensus}.
CBO ideas have also found application in multi-agent game-theoretic frameworks \cite{chenchene2024}, min–max or saddle-point problems \cite{B1-qiu2022Saddlepoints,borghi2024particle}, hierarchical and bilevel formulations \cite{multilevel,trillos2025cb2o}, and clustered federated learning \cite{JMLR:v25:23-0764}, and reliably computing ground states of nonconvex energies such as the Ginzburg-Landau energy \cite{beddrich2024}. (The literature surrounding CBO continues to expand very rapidly, making it increasingly difficult to provide a comprehensive set of references. We therefore highlight only a curated list of recent developments. For an overview and for comparisons with related multi-agent global optimization schemes, see \cite{totzeck2021trends} and \cite{totzeck2018numerical}.).
}

While CBO has shown remarkable practical effectiveness on many global optimization challenges, CBO  is strikingly supported by {\it rigorous} theory, including convergence rates, for broad classes of nonconvex and nonsmooth objectives \cite{B1-fornasier2021global,B1-fornasier2022globalanisotropic}.

The current implementation of CBO is based on an Euler-Maruyama discretization of its SDE formulation and reads as follows (for $j \in [N]$ particle index and $k \in \mathbb N$ iteration index):
\begin{equation}\tag{$CBO_{N,\Delta t}$}\label{eq:cbo_finite_discrete}
 \Xf_{(k+1)\Delta t}^j =\Xf_{k \Delta t}^j -\lambda \Delta t \big(\Xf_{k \Delta t}^j -  \Xalpha(\widehat\rho^N_k)\big) + \sqrt{\Delta t} \sigma |\Xf_{k \Delta t}^j -  \Xalpha(\widehat\rho^N_t)|  B_{k \Delta t}^j, \quad \Xf_{0}^j \overset{\mbox{i.i.d}}{\sim} \rho_0 \in \mathcal P(\mathbb R^d),
\end{equation}
where $B_{k \Delta t}^j \sim \mathcal N(0,I_d)$, $ \widehat \rho^{N}_k=\frac{1}{N}\sum_{j=1}^N\delta_{\Xf_{k\Delta t}^j}$ and
\begin{equation}\label{XaN}
X^\alpha(\rho)=\frac{\int_{\mathbb R^d}x e^{-\alpha f(x)} d\rho(x)}{\int_{\mathbb R^d}e^{-\alpha f(x)} d \rho(x)},
\end{equation}
is the so-called {\it consensus-point}.\\

The success of CBO stems from a well-calibrated  interaction between exploitation and exploration mechanisms. Exploration is induced by the stochastic movement of a large ensemble of particles, while exploitation is achieved by aggregating their collective information into the consensus point \eqref{XaN} that functions as an instantaneous proxy for the global minimizer. As the number of particles increases, these two mechanisms, diffusive search and consensus-driven contractivity, mutually reinforce one another, yielding to robust convergence.\\

Under minimal assumptions on the objective function $f$, namely local Lipschitz continuity and local polynomial growth around the (unique) global minimizer $x^*$  (by now there are variants of CBO tackling problems with multiple global minimizers \cite{bungert2022polarized,B1-fornasierSun25,huang2025faithful}), for $\alpha>0$ large enough  \cite[Theorem 3.8]{B1-fornasier2021global} ensures convergence in probability with rate
\begin{equation}\label{eq:convprob}
  \mathbb P\Bigg( \big |\frac{1}{N} \sum_{j=1}^N \Xf_{K\Delta t}^j-x^*\big |^2 \leq \varepsilon\Bigg) \geq 1 - \left [ \varepsilon^{-1} (C_{\mathrm{NA}}\Delta t+C_{\mathrm{MFA}} N^{-1}+C_{0}e^{-(2 \lambda - d \sigma^2) \Delta t K}) +\delta \right ],
\end{equation}
where $\varepsilon, \delta$ are arbitrarily small positive accuracies, $2\lambda > d \sigma^2>0$, and $K=K(\varepsilon,\alpha,\rho_0)$ is a determined iteration large enough, beyond which the exponential contractivity $\mathcal O(e^{-(2 \lambda - d \sigma^2) \Delta t K})$ is no longer valid. (Hence, this bound holds at finite time only.)
The proof of this remarkable result combines numerical approximations of SDE \cite{platen1999introduction}, a quantitative mean-field limit \cite[Proposition 3.11]{B1-fornasier2021global}, \cite[Theorem 2.6]{gerber_mean-field_2024}, a quantitative Laplace principle \cite[Proposition 4.5]{B1-fornasier2021global}, and the large time analysis of the nonlinear Fokker-Planck equation 
\begin{equation}\label{eq:FP_CBO}
\partial_t\rho_t(x) = \lambda \dv\left(\rho_t(x)(x-\Xalpha(\rho_t))\right) +\frac{\sigma^2}{2}\Delta( |x-\Xalpha(\rho_t)|^2 \rho_t(x)),
\end{equation}
whose solution $\rho_t$ is the law of the solution of the McKean-Vlasov process
\begin{equation}\tag{CBO}\label{eq:cbo_mf}
    \diff X_t = -\lambda(X_t - \Xalpha(\rho_t)) \diff t + \sigma |X_t - \Xalpha(\rho_t)| \diff B_t,
\end{equation}
obtained from \eqref{eq:cbo_finite_discrete} by taking the limits $\Delta t \to 0$  and $N\to \infty$. See \cite[Theorem 3.7]{B1-fornasier2021global} for more details. \\

It is important to recall that there are approaches to analyze the particle system \eqref{eq:cbo_finite_discrete} directly, without relating it to the continuous-in-time process \eqref{eq:cbo_mf}, see, e.g., \cite{kim2020stochastic,ko2022convergence,gottlich2025exponential}. While these results are very valuable as they establish contractivity of the dynamics (consensus formation) for any {\it finite} $N$, $\Delta t$, and $\alpha$  fixed, differently from \eqref{eq:convprob} they do not provide a quantitative rate of convergence to global minimizers with respect to the number $N$ of particles. (Establishing how many particles one needs to solve a given optimization problem is indeed of utmost practical relevance.) 
This issue has been resolved in a recent modification of the discrete CBO in \cite{byeon2024}, which also provides a quantitative result in probability with respect to the number $N$ of particles. This latter model differs from \eqref{eq:cbo_finite_discrete} as the consensus point is substituted with the hard minimizer $\arg \min_{j=1,\dots,N} f(\Xf_{(k)\Delta t}^j)$. Yet, this valuable result requires that the minimizer $x^*$ is in the support of the initial distribution $\rho_0$ of particles and it is based on the probability that one initial particle is generated randomly in the vicinity of the global minimizer. (The global convergence results in \cite{byeon2024} are indeed the best ones so far that describe a fully discrete CBO dynamics without relating it to continuous-in-time dynamics, which require enough smoothness to $f$ for well-posedness. Indeed, global convergence in \cite{byeon2024} is established for merely continuous or even measurable objectives $f$.) The requirement $x^* \in \operatorname{supp}(\rho_0)$ has been recently removed in \cite{fornasier2025regularity} for the CBO dynamics \eqref{eq:cbo_mf}. Hence, so far,  the only viable approach to establish global convergence rates with respect to $N$ without constraints on the initial distribution remains by quantitative mean-field limit \cite[Proposition 3.11]{B1-fornasier2021global}, \cite[Theorem 2.6]{gerber_mean-field_2024} and the analysis of the continuous-in-time mean-field model. For this reason and because it is also of practical relevance to track the distribution of particles, as we explain in more detail below in Subsection \ref{sec:methnote}, in this paper we focus on {\it mean-field models} and their analysis.
\\

Our reason for insisting on CBO is that its analysis provides a natural framework for explaining the behavior of other relevant global optimizers, such as MPPI and ES. Let us now focus on the challenges of the CBO analysis and later come back to the connection to MPPI and ES.

\subsection{Main challenges}
Formula \eqref{eq:convprob} shows that the expected error decays at a rate of $\mathcal{O}(\Delta t+N^{-1})$ as the particle number $N$ grows and $\Delta t$ vanishes, indicating that accuracy improves substantially with larger ensembles and the time step $\Delta t$ is chosen small. However, the constants in the bound may increase exponentially with certain problem parameters, such as $\alpha$ or the dimension. Consequently, although these theoretical guarantees do ensure asymptotic effectiveness of CBO, they do not fully capture  the practical scenario where  ~\eqref{eq:cbo_finite_discrete}  is restricted to a relatively small number $N$ of particles and one wishes to progress the iterations with larger time steps for faster convergence.
In practice, in such situations, the particle population is often observed to collapse prematurely to a Dirac delta distribution--{\it before} the global minimizer has been reached--as illustrated in \cref{fig:failure_mode_cbo_intro} and previously noted in \cite{huang2025faithful,huang_self-interacting_2025}.
This possible early collapse problem affects also the discrete ``hardmin" CBO model \cite{byeon2024}. Moreover, on the theoretical side, while the estimate \eqref{eq:convprob} provides a bound at {\it finite time},  proving global convergence in  mean-field law over an infinite time horizon has remained elusive, see \cite[Theorem 3.7]{B1-fornasier2021global}, except for well-prepared initial conditions, for which the initial distribution $\rho_0$ is already well-concentrated in the vicinity of the global minimizers $x^*$ \cite{carrillo2018analytical}. 
\begin{figure}[ht]
    \centering
    \setlength{\tabcolsep}{2pt} % or 0pt for no gap
    \begin{tabular}{ccccc}
        \includegraphics[width=0.18\textwidth]{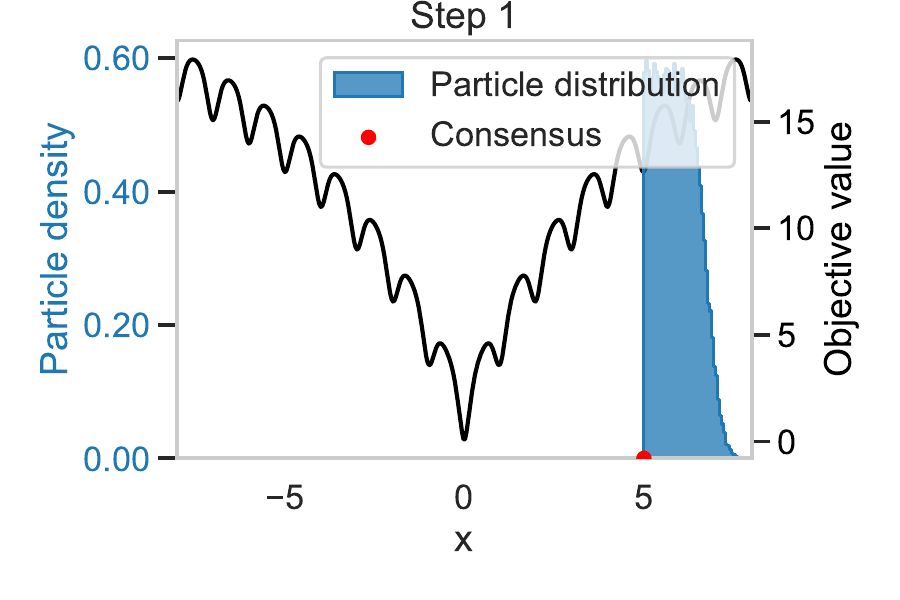} &
±        \includegraphics[width=0.18\textwidth]{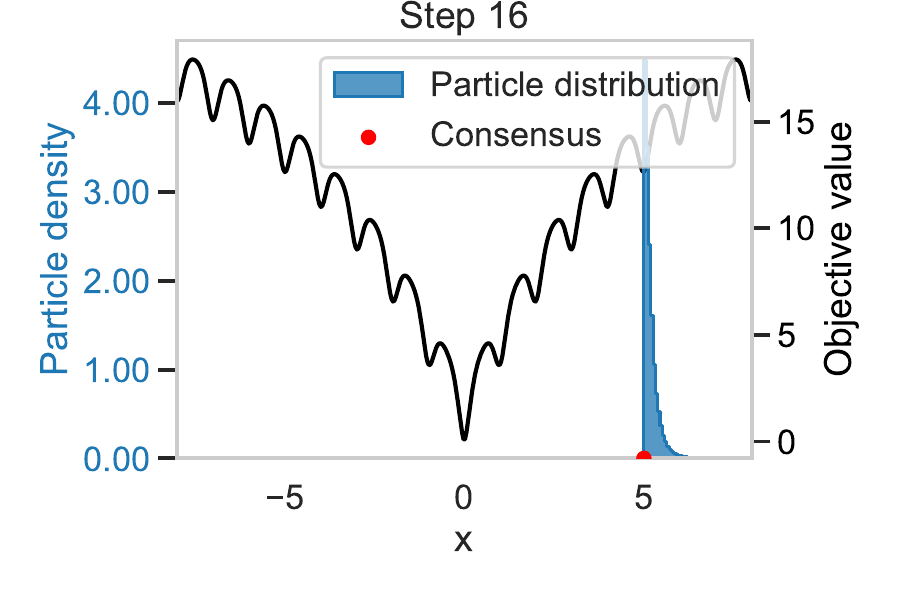} &
        \includegraphics[width=0.18\textwidth]{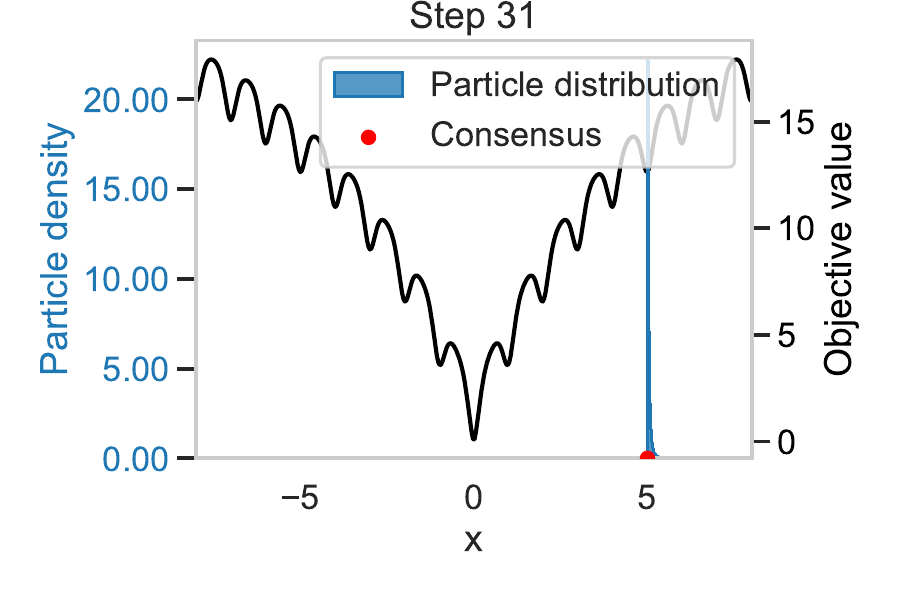} &
        \includegraphics[width=0.18\textwidth]{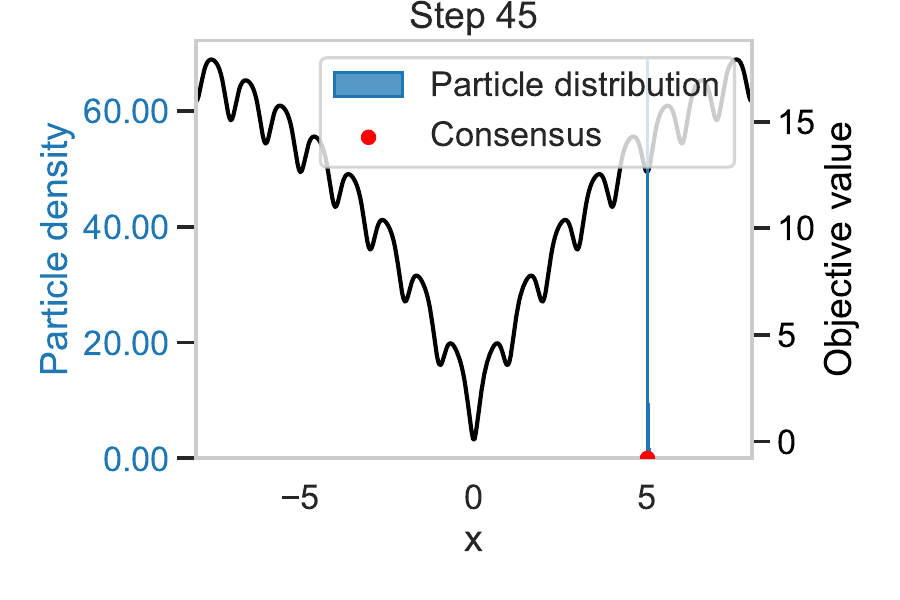} &
        \includegraphics[width=0.18\textwidth]{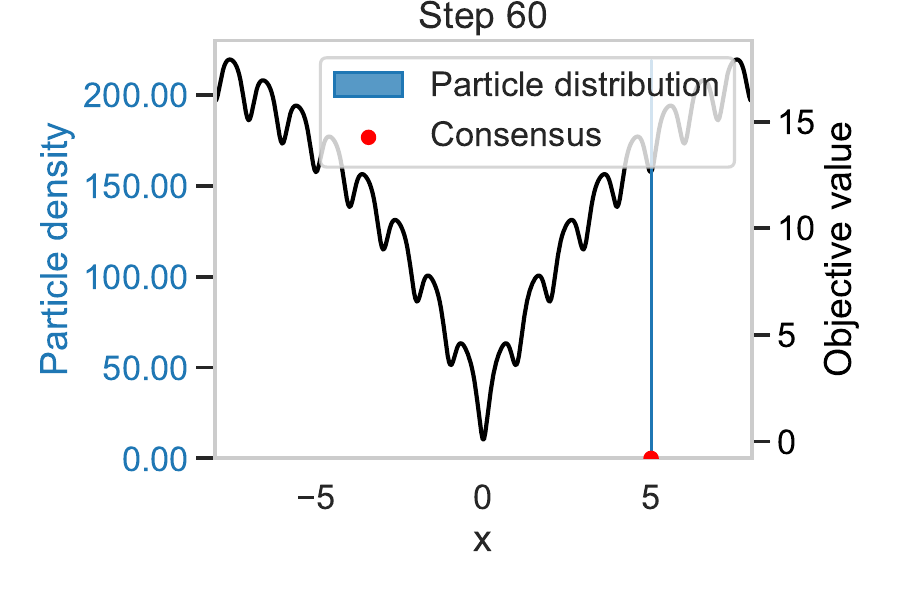} \\[0.1em]
        \includegraphics[width=0.18\textwidth]{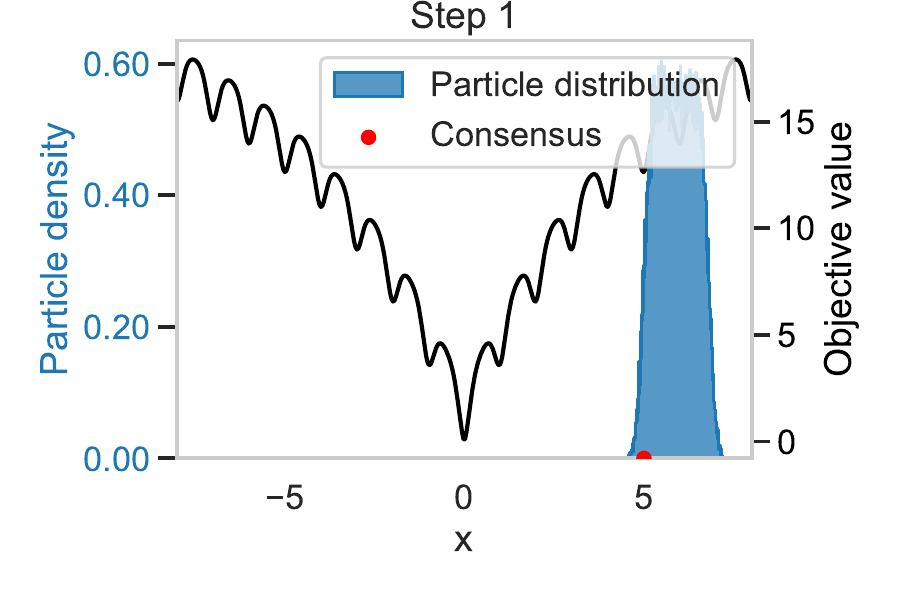} &
        \includegraphics[width=0.18\textwidth]{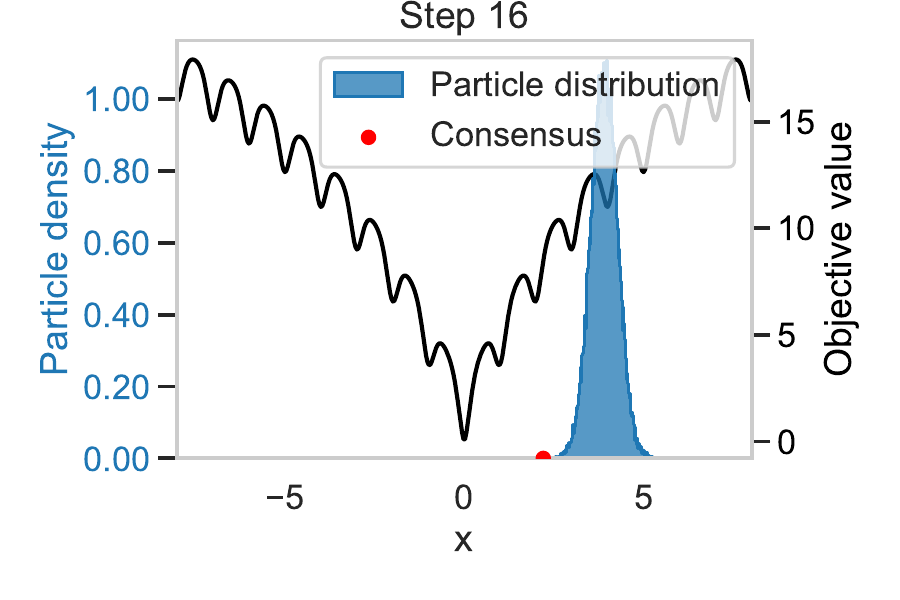} &
        \includegraphics[width=0.18\textwidth]{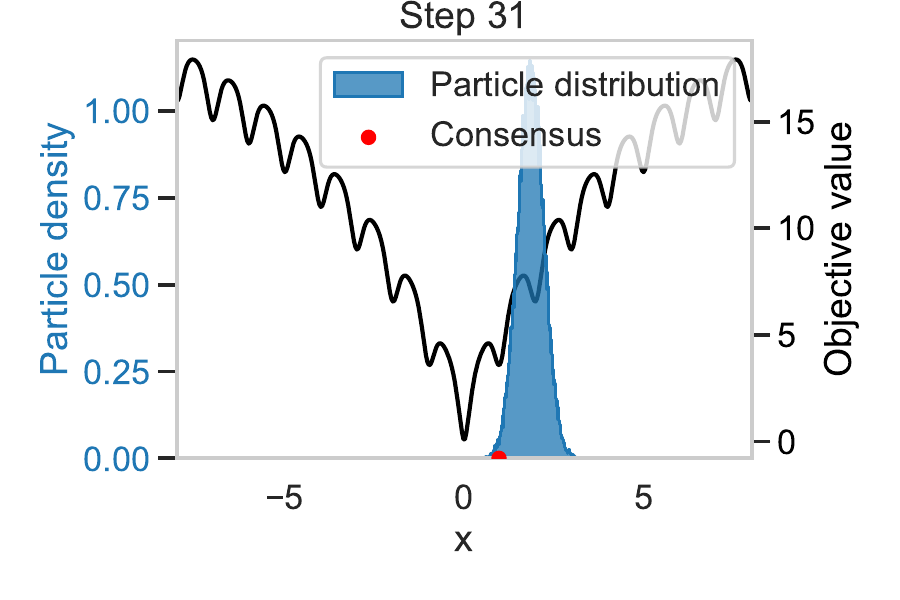} &
        \includegraphics[width=0.18\textwidth]{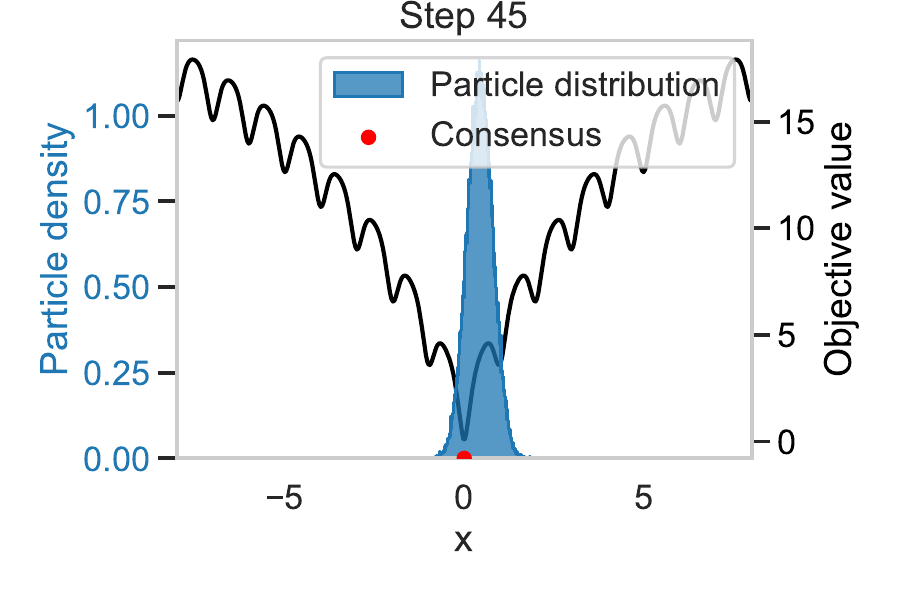} &
        \includegraphics[width=0.18\textwidth]{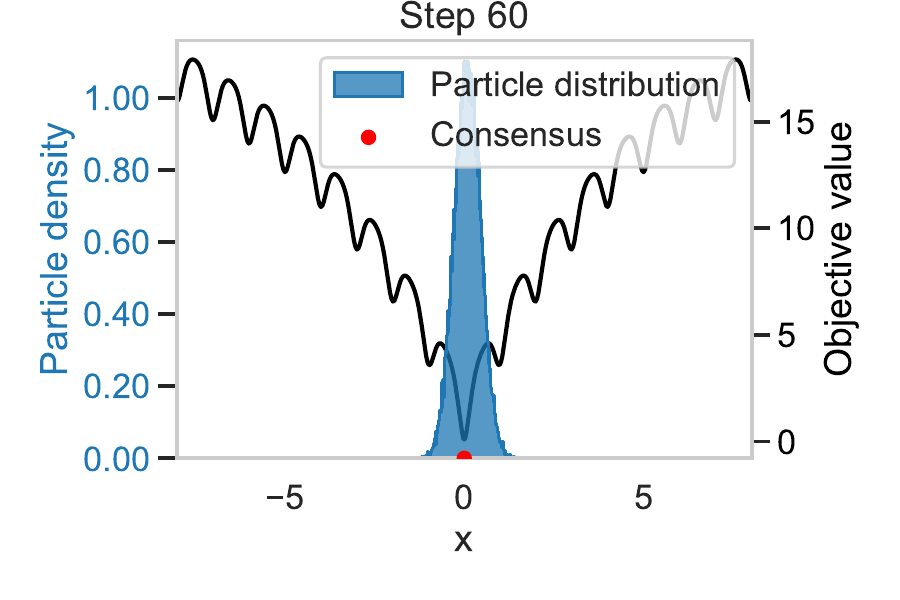}
    \end{tabular}
    \caption{A one-dimensional configuration where CBO collapses prematurely---due to the global minimizer not being contained in the initial distribution---, while $\delta$-CBO succeeds to localize around the global minimizer. In the top row, we show the behavior of standard CBO, where the particles eventually concentrate at $x=5$. In the bottom row, we show the behavior of $\delta$-CBO, where the particles successfully localize around the global minimizer at $x=0$. Let us stress that this one dimensional example is not the typical behavior for CBO as in higher dimension, the minimizer is going to be always in the support of the evolving CBO distribution \cite{fornasier2025regularity}.}
    \label{fig:failure_mode_cbo_intro}
\end{figure}

\subsection{Solution}

In this paper we address  the aforementioned limitations of the model \eqref{eq:cbo_mf}, which are mostly due to two aspects. The first is the way the noise term is allowed to vanish (often too early in the few particles regime); while vanishing noise is certainly appealing as it induces the emergence and concentration to a Dirac delta and is reminiscent of the annealing process in Simulated Annealing, it also makes {\it any} Dirac delta an invariant measure of \eqref{eq:cbo_mf} \cite{huang_self-interacting_2025} and, therefore, a potential attractor. The second is the way \eqref{eq:cbo_mf} gets discretized, namely the Euler-Maruyama scheme \eqref{eq:cbo_finite_discrete}, which is intrinsically affected by instability as $\Delta t$ gets sufficiently large. 
\\

Recently, variants of CBO have been proposed where the noise term does not vanish, allowing for provable global convergence to a {\it unique} invariant measure over an infinite time horizon in the mean-field law \cite{huang2025faithful,huang2025uniform}, see also \cite{bianchi2025}. {The same non-degenerate dynamics also arises as the formal mean-field limit of
a kinetic model of genetic algorithms~\cite{borghi2025kinetic}. We also mention that non-degenerate noise is the central idea  of \cite{B1-carrillo2022sampling}, which is aiming at consensus-based {\it sampling},  differently from the present paper which is focused on pure optimization (through mean/center of mass selection $x^*$ of the limit Gaussian density $\boldsymbol\rho_\infty = \mc{N}(x^*, \frac{\delta^2}{2\lambda} I_d)$, see again this notion below).}
}

We follow this line of work by starting with a study of the $\delta$-CBO scheme, where the particles follow the SDE
\begin{equation}\tag{$\delta$-CBO$_N$}\label{eq:cbo_delta_finite}
    \diff \Xf_t^j = -\lambda \big(\Xf_t^j -  \Xalpha(\widehat\rho^N_t)\big) + \delta \diff B_t^j, \quad j \in [N],
\end{equation}
and of the corresponding mean-field McKean-Vlasov process
\begin{equation}\tag{$\delta$-CBO}\label{eq:cbo_delta_mf}
    \diff X_t = -\lambda(X_t - \Xalpha(\rho^\alpha_t)) \diff t + \delta \diff B_t, \quad \rho^\alpha_t=\operatorname{Law}(X_t).\end{equation}
A first numerical implementation can go again along an Euler-Maruyama discretization
\begin{equation}\tag{$\delta-CBO_{N,\Delta t}$}\label{eq:delta-cbo_finite_discrete}
 \Xf_{(k+1)\Delta t}^j =\Xf_{k \Delta t}^j -\lambda \Delta t \big(\Xf_{k \Delta t}^j -  \Xalpha(\widehat\rho^N_k)\big) + \sqrt{\Delta t} \delta  B_{k \Delta t}^j, \quad \Xf_{0}^j \overset{\mbox{i.i.d}}{\sim} \rho_0 \in \mathcal P(\mathbb R^d).
\end{equation}
 The main difference from \eqref{eq:cbo_mf} lies in the fact that the noise term has a fixed variance. Perhaps counter to the intuitive appeal of vanishing noise, this modification of the model is meant to help preventing premature convergence to suboptimal solutions.
 %\footnote{Admittedly, removing the vanishing noise from the model makes it  easier to analyze, even if it becomes mathematically less intriguing. However, because our goal is to fully understand the mechanisms that make CBO an effective method for challenging nonconvex, nonsmooth optimization problems, it is important to remain focused on the real objective: solving difficult optimization tasks efficiently while maintaining robust theoretical guarantees.}
As the first result of this paper, we prove the global convergence of \eqref{eq:cbo_delta_mf}. In Section \ref{sec:delta_CBO} we mirror the finite time analysis of \cite{B1-fornasier2021global} with significant adaptations (which we highlight below); yet, we go beyond that, and in Section \ref{sec:delta_CBO2} we provide a {\it complete description} of the dynamics also for an infinite time horizon, by proving its convergence to a {\it unique} invariant measure. {More precisely, we shall prove that \eqref{eq:cbo_delta_mf} admits a unique invariant measure $\boldsymbol\rho_\infty^\alpha$, which is a Gaussian satisfying
\begin{equation}
    \boldsymbol\rho_\infty^\alpha = \mathcal{N}\!\left(\Xalpha(\boldsymbol\rho_\infty^\alpha),\, \frac{\delta^2}{2\lambda} I_d\right),
\end{equation}
and that for any initial datum $\rho_0 \in \mathcal{P}_4(\mathbb{R}^d)$, the solution $\rho_t^\alpha$ to \eqref{eq:cbo_delta_mf} converges exponentially fast to $\boldsymbol\rho_\infty^\alpha$ in $2$-Wasserstein distance:
\begin{equation}
    W_2(\rho_t^\alpha, \boldsymbol\rho_\infty^\alpha) \leq C \gamma^t \quad \text{for some } \gamma \in (0,1).
\end{equation}
Furthermore, the mean of the invariant measure $\boldsymbol\rho_\infty^\alpha$ converges to the unique global minimizer $x^*$ as $\alpha \to \infty$, namely
\begin{equation}
    \Xalpha(\boldsymbol\rho_\infty^\alpha) \to x^* \quad \text{as } \alpha \to \infty,
\end{equation}
which in turn implies
\begin{equation}
    W_2(\boldsymbol\rho_\infty^\alpha, \boldsymbol{\rho}_\infty) \to 0 \quad \text{as } \alpha \to \infty,
\end{equation}
where $\boldsymbol{\rho}_\infty = \mathcal{N}\!\left(x^*,\, \frac{\delta^2}{2\lambda} I_d\right)$. In contrast to existing CBO convergence results in the literature, our analysis holds in the  asymptotic regime where both the time horizon $t \to \infty$ and the inverse temperature $\alpha \to \infty$, independently from each other.
}
\\

For implementing the scheme numerically, so far we had only the option of choosing a time discretization $\Delta t$ for the Euler–Maruyama method. At this point, we stress again the emerging tension: to converge with a minimal number of time steps, it would be ideal to set $\Delta t$ large; however, for large $\Delta t$, the discretization errors of the Euler-Maruyama method become large and they accumulate. We recognize these discretization errors in an increased empirical variance compared to what we expect from the continuous-in-time model (\cref{fig:compare_variance_deltat}), eventually leading to a divergence of the method for $\Delta t$ too large (\cref{fig:compare_performance_deltat}). This tension in the time discretization  leads to three problems in practice:
\begin{enumerate}
    \item The regime of large $\Delta t$, where the method would ideally perform fastest, is the one in which the discrete time algorithm \eqref{eq:cbo_finite_discrete} does not mirror the continuous-in-time theoretical dynamics, limiting the usefulness of theoretical bound \eqref{eq:convprob};
    \item $\Delta t$ needs to be carefully chosen, as either too large or too small values inhibit convergence to a global minimizer;
    \item Even for carefully chosen $\Delta t$, the performance might not be optimal, as the particles still make comparably little progress while the empirical variance is already too large.
\end{enumerate}
Unfortunately, all of these three issues cannot be remedied by simply employing an implicit scheme, simply because also in this case we need $\Delta t \approx 0$ in order to approximate the continuous-in-time dynamics \eqref{eq:cbo_delta_finite}-\eqref{eq:cbo_delta_mf} to be able to establish a rate of convergence with respect to $N$.
\begin{figure}[ht]
  \centering
  \begin{subfigure}[t]{0.48\textwidth}
    \includegraphics[width=\linewidth]{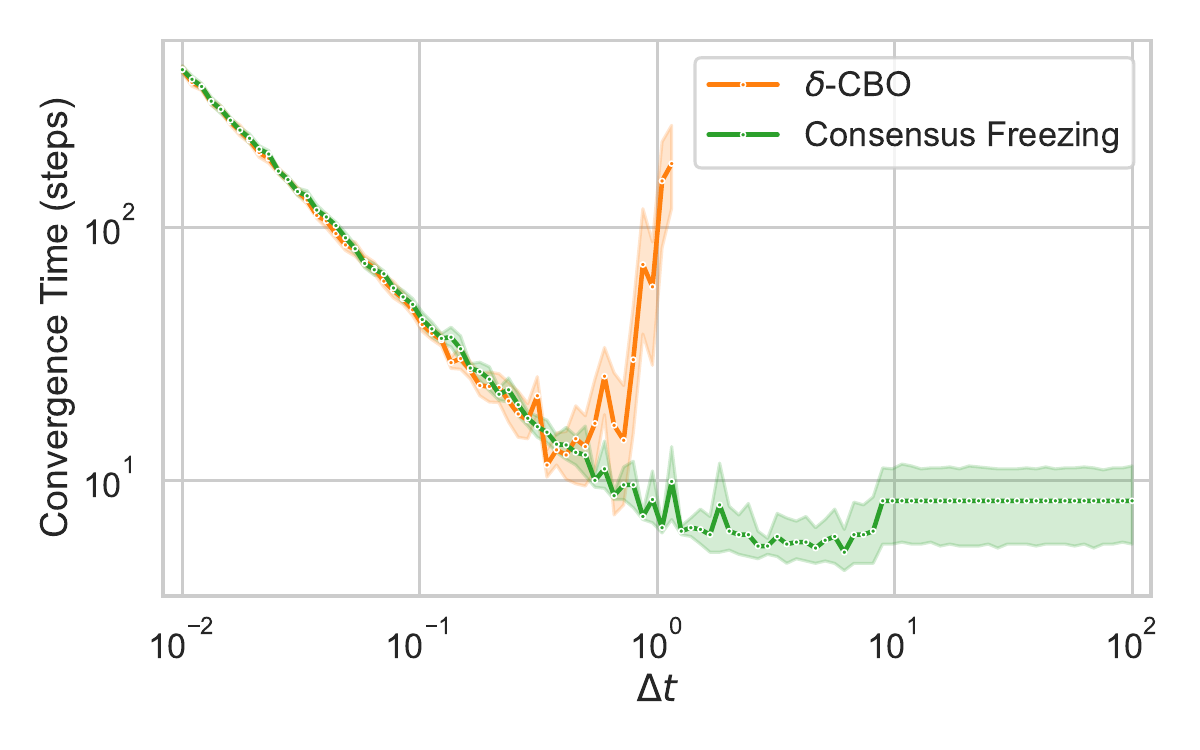}
    \caption{The number of steps required to find the global minimizer for both the $\delta$-CBO scheme and the Consensus Freezing scheme for different values of $\Delta t$. For moderate values of $\Delta t$, the performance of the $\delta$-CBO scheme starts to deteriorate, while the performance of the Consensus Freezing scheme stabilizes.}
    \label{fig:compare_performance_deltat}
  \end{subfigure}
  \hfill
  \begin{subfigure}[t]{0.48\textwidth}
    \includegraphics[width=\linewidth]{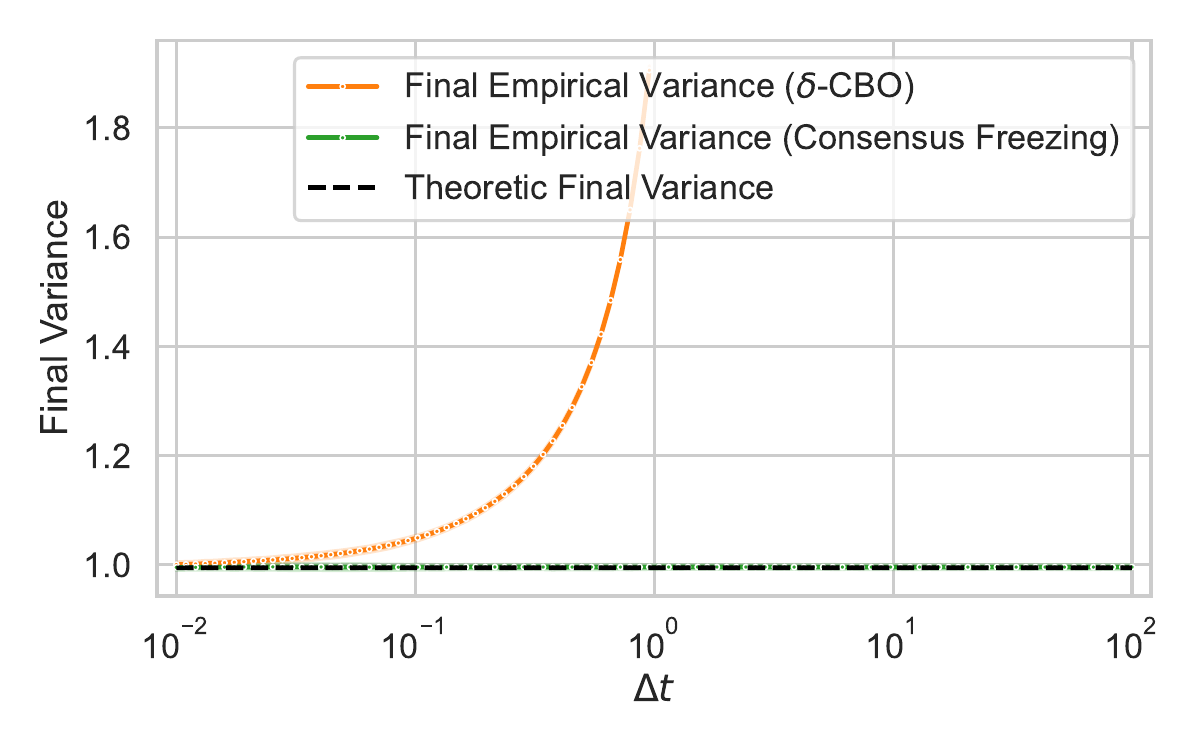}
    \caption{The final empirical variance after 500 steps for both the $\delta$-CBO scheme and the Consensus Freezing scheme, compared with what we would expect from a theoretical analysis. For $\Delta t\geq 0.1$, the empirical variance of the particles starts to grow due to errors stemming from the time discretization. In contrast, the empirical variance of the particles evolving through Consensus Freezing scheme aligns well with what we expect theoretically.}
    \label{fig:compare_variance_deltat}
  \end{subfigure}
  \caption{Comparing the Euler-Maruyama discretization of the $\delta$-CBO scheme with the exact implementation of the Consensus Freezing scheme.}
  \label{fig:two_images}
\end{figure}
To overcome this  challenge, we propose yet another variant of CBO, which we call the \emph{Consensus Freezing scheme}, where, denoting with $\widehat\varrho^{\alpha,s,\Delta t,N}_t = \frac{1}{N}\sum_{j=1}^N\delta_{\Xf^j_t}$, the particle dynamics are described by the SDE
\begin{equation}\tag{CF$_N$}
    \diff \Xf^j_t = -s \cdot \lambda(\Xf^j_t - \Xalpha(\widehat\varrho^{\alpha,s,\Delta t,N}_{t_i})) \diff t + \sqrt{s} \cdot \delta \diff B^j_t, \quad \text{for } t \in [t_i, t_{i+1}),
\end{equation}
and the corresponding mean-field limit
\begin{equation}\tag{CF}\label{eq:flip_scheme}
    \diff X_t = -s\cdot\lambda\left(X_t - \Xalpha(\varrho^{\alpha,s,\Delta t}_{t_i})\right) \diff t + \sqrt{s}\cdot\delta \diff B_t, \quad \text{for } t \in [t_i, t_{i+1}),
\end{equation}
where $\varrho^{\alpha,s,\Delta t}_t \coloneqq \law(X_t)$, and $t_i \coloneqq  i \cdot \Delta t$ for some step size $\Delta t > 0$.
By freezing the consensus point in the interval between time steps $t_i$ and $t_i + \Delta t$, the Consensus Freezing scheme allows for a numerical implementation that does not require any further time discretization. Indeed, comparing the Euler-Maruyama method for the \eqref{eq:cbo_delta_mf} scheme with the exact implementation of the Consensus Freezing scheme, we see that  the two behave very similarly for small values of $\Delta t$, while---in contrast to the pathological behavior of the former---the latter behaves well even for large values of $\Delta t$.  The second result of this paper is about proving that also the  Consensus Freezing scheme is globally convergent and that, for $\Delta t \to 0$, solutions of \eqref{eq:flip_scheme} converge to solutions of \eqref{eq:cbo_delta_mf}.

Strikingly, the Consensus Freezing scheme works best for large values of $\Delta t$ (\cref{fig:compare_performance_deltat}). This last observation motivates one further progression beyond Consensus Freezing: by introducing a fictitious time rescaling parameter  $s$, which allows us to control the speed of the continuous dynamics while leaving $\Delta t$ fixed, and taking the limit $s \to \infty$, we prove that the Consensus Freezing dynamics converge to the discrete-in-time \emph{Consensus Hopping scheme}
\begin{align}
\tag{CH}\label{eq:CH}
\Xf_{k+1}^{CH} =  X_{\alpha}( \mathcal N(\Xf_{k}^{CH}, \sigma_{CH}^2 I_d)), \quad k=0,1,\dots
\end{align}
for the variance $\sigma_{CH}^2=\frac{\delta^2}{2\lambda}$.
This scheme was proposed for the first time under the name of Model Predictive Path Integral (MPPI) in 2017 \cite{MPPI} in the context of optimal control or reinforcement learning in robotics. While this zero-order method has been quickly adopted by the robotics community as one of the methods of choice (besides, e.g., Cross Entropy Method or Covariance Matrix Adaptation Evolution Strategy \cite{sun2026CBOrob}) no proof of global convergence was provided. Moreover the folklore has been so far that the method is local in nature, as a sort of relaxation of gradient descent.
This scheme was re-derived again in connection to CBO for the first time in \cite{riedl_gradient_2023} by taking $\lambda=\frac{1}{\Delta t}$ in the discrete-time formulation \eqref{eq:cbo_finite_discrete}.
Remarkably, the numerical implementation of the Consensus Hopping (aka MPPI) scheme constitutes a simple form of a $(\mu, \lambda)$-Evolution Strategy, with $\mu=1$ and $\lambda=N$, the number of particles. We recall that a $(\mu,\lambda)$-Evolution Strategy \cite{AugerHansen2005} maintains a population of $\mu$ candidate solutions, called parents. In each iteration, the algorithm generates $\lambda$ new candidate solutions, called offspring, by applying a  ``mutation" to the parents. The mutation operator typically adds random perturbations drawn from a normal distribution to the parent solutions. (Compare \eqref{eq:CH}.) Proofs of global convergence for evolution strategies were obtained long ago, for instance in \cite{Rudolph1997,Rudolph1998}, using arguments based on Markov chains. However, most classical proofs of global convergence  do not provide explicit convergence rates. Based on the connection between CBO and Evolution Strategies obtained through the conceptual arch \eqref{eq:cbo_mf}$\Rightarrow$\eqref{eq:flip_scheme} $\Rightarrow$\eqref{eq:CH} (see Figure \ref{fig:CBOworld}), we can now 
also establish exponential convergence rates, which is the last major result of this paper.
Connections between CBO and Evolution Strategies have also been mentioned in \cite{roith25} and analyzed in \cite{borghi2025kinetic}.{ In particular, a complementary bridge between CBO and Evolution Strategies was proposed in~\cite{borghi2025kinetic}, which
derives a kinetic description of Genetic Algorithms (GA) under a
propagation-of-chaos assumption. In a quasi-invariant scaling where one parent
is selected uniformly and the other by Boltzmann selection, their Boltzmann-type
model {\it formally} reduces to CBO with non-degenerate diffusion, i.e., to the
$\delta$-CBO dynamics~\eqref{eq:cbo_delta_mf}  studied here. For Boltzmann selection
\cite{borghi2025kinetic} also proves that {\it the mean} of the kinetic population reaches any prescribed
neighborhood of $x^*$ within a {\it finite time horizon}. In this paper we travel the bridge in the
opposite direction, from CBO towards Evolution Strategies
(Figure~\ref{fig:CBOworld}), and obtain stronger conclusions and rigorous analysis.}
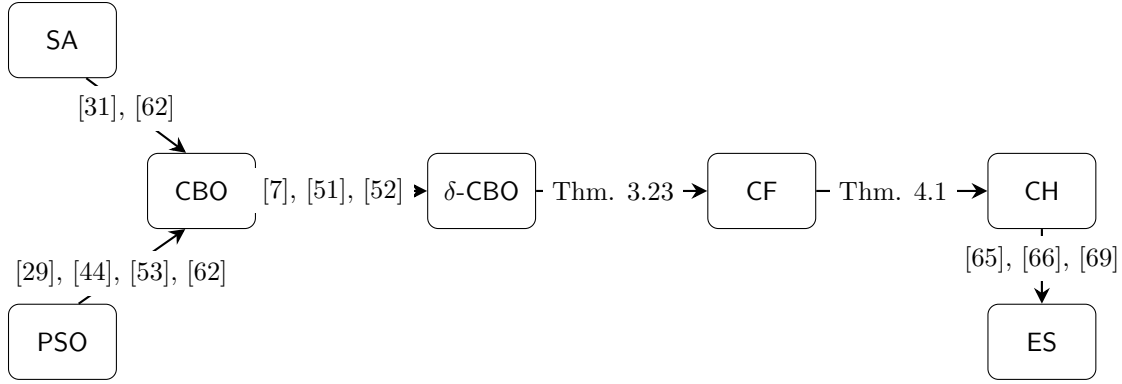
\begin{figure}
\centering
\begin{tikzpicture}[
    node distance=15mm and 30mm,
    box/.append style={text width=1.1cm},                 % slightly narrower boxes (was 1.2cm)
    lbl/.style={font=\footnotesize, inner sep=1.5pt},     % labels sit next to (not on) the arrows
    formalarrow/.style={arrow, dashed},
    group/.style={draw, dashed, rounded corners=6pt, inner xsep=2.5mm, inner ysep=3mm}]
  % main chain (distances measured between box borders)
  \node[box] (cbo) {CBO};
  \node[box, right=20mm of cbo]  (dcbo) {$\delta$-CBO};
  \node[box, right=18mm of dcbo] (cf)   {CF};
  \node[box, right=18mm of cf]   (ch)   {CH};
  % origins of CBO
  \node[box, above left=7mm and 9mm of cbo] (sa)  {SA};
  \node[box, below left=7mm and 9mm of cbo] (pso) {PSO};
  % evolutionary algorithms
  \node[box, below=20mm of ch]   (es) {ES};
  \node[box, below=20mm of dcbo] (ga) {GA};
  \node[group, fit=(ga)(es),
        label={[font=\footnotesize]below:Evolutionary algorithms}] (ea) {};
  % arrows into CBO: labels in the free space to the right of the arrows
  \draw[arrow] (sa.south east) -- (cbo.north west)
    node[lbl, pos=0.5, anchor=south west] {\cite{fornasier2021consensus,pinnau2017consensus}};
  \draw[arrow] (pso.north east) -- (cbo.south west)
    node[lbl, pos=0.5, anchor=north west] {\cite{grassi2020particle,B1-cipriani2021zero,hui23,pinnau2017consensus}};
  % main chain: labels above the arrows
  \draw[arrow] (cbo) -- (dcbo)
    node[lbl, midway, above] {\cite{huang2025faithful,huang2025uniform,bianchi2025}};
  \draw[arrow] (dcbo) -- (cf)
    node[lbl, midway, above] {Thm.~\ref{thm:CFtodCBO}};
  \draw[arrow] (cf) -- (ch)
    node[lbl, midway, above] {Thm.~\ref{thm:FStoHS}};
  \draw[arrow] (ch) -- (es)
    node[lbl, pos=0.4, left] {\cite{sun2026CBOrob,riedl_gradient_2023,roith25}};
  % formal mean-field limit GA -> delta-CBO
  \draw[formalarrow] (ga) -- (dcbo)
    node[lbl, pos=0.5, right] {\cite{borghi2025kinetic}};
\end{tikzpicture}
\caption{Collocation of Consensus-Based Optimization (CBO) within the global optimization landscape. Here SA$=$``Simulated Annealing", PSO$=$``Particle Swarm Optimization", $\delta$-CBO$=$``CBO with fixed variance",  
CF$=$``Consensus Freezing", CH$=$``Consensus Hopping/MPPI", {GA$=$``Genetic Algorithms", ES$=$``Evolution Strategies". The dashed arrow indicates the formal mean-field limit of the kinetic model of GA derived in~\cite{borghi2025kinetic}, which yields the $\delta$-CBO dynamics.}}
\label{fig:CBOworld}
\end{figure}

We summarize our main results as follows
\begin{enumerate}
    \item We prove infinite-time horizon convergence and steady-state solutions for all {dynamics};
    \item We can not only determine the invariant measure and the convergence rate, but in fact we can fully describe the distribution at each time/iteration;
    \item We provide stable numerical implementations for {\it arbitrarily large} time steps of all algorithms and {we prove their convergence in expectation with quantitative rates};
    \item The arch from Consensus-Based Optimization to the Consensus Hopping scheme presented in this manuscript builds a bridge from Particle Swarm Optimization to Evolution Strategies (see Figure \ref{fig:CBOworld}), where each step is motivated by numerical results and supported by {rigorous} theoretical convergence guarantees.
\end{enumerate}

More formally, for each $\alpha > 0$, there exists a measure $\boldsymbol\rho_\infty^\alpha \coloneqq \mathcal{N}\!\left(\Xalpha(\boldsymbol\rho_\infty^\alpha),\, \frac{\delta^2}{2\lambda} I_d\right)$ which is invariant for the $\delta$-CBO dynamics \eqref{eq:cbo_delta_mf}, the Consensus Freezing dynamics \eqref{eq:flip_scheme}, and the Consensus Hopping scheme \eqref{eq:CH}. Furthermore, for any $\rho_0 \in \mathcal{P}_4(\mathbb{R}^d)$ and any $x_0 \in \mathbb{R}^d$, there exists $\alpha_0 > 0$ such that for all $\alpha \geq \alpha_0$, the following exponential convergence estimates hold for the solutions $\rho_t^\alpha$ to \eqref{eq:cbo_delta_mf} with initial datum $\rho_0$, $\varrho_t^{\alpha,s,\Delta t}$ to \eqref{eq:flip_scheme} with initial datum $\rho_0$, and the iterates $x_k^{CH}$ of \eqref{eq:CH} with $x_0^{CH}=x_0$:
\begin{alignat*}{3}
    W_2(\rho_t^\alpha, \boldsymbol\rho_\infty^\alpha) &\leq C_\alpha \gamma^t , \quad &&t \geq 0, \gamma \in (0, 1) , \quad &\text{see \cref{thm:globconv_deltaCBO}},\\
    W_2(\rho_{t_k}^{\alpha,s,\Delta t}, \boldsymbol\rho_\infty^\alpha) &\leq C_\alpha \gamma^k, \quad &&k \in \NN, \gamma \in (0, 1), \quad &\text{see \cref{thm:fsconv}},\\
    W_2(\nu_k^\alpha, \boldsymbol\rho_\infty^\alpha) &\leq C_\alpha \gamma^k, \quad &&k \in \NN, \gamma \in (0, 1), \nu_k^\alpha=\mathcal N(x_k^{CH}, \frac{\delta^2}{2\lambda}I), \quad &\text{see \cref{cor:HS_convergence}}. 
\end{alignat*}
Furthermore, we have that $\boldsymbol{\rho_\infty^\alpha} \to \boldsymbol{\rho_\infty} = \mathcal{N}\!\left(x^*,\, \frac{\delta^2}{2\lambda} I_d\right)$ as $\alpha \to \infty$ in $W_2$ distance (see \cref{thm:invariant}). \\
As for a complete description of the dynamics, we see in \cref{lem:explicit_solution} that both $\rho_t^\alpha$ and $\varrho_t^{\alpha,s,\Delta t}$ are convolutions of Gaussians with shifted and rescaled versions of the initial distribution $\rho_0$; in particular, they stay Gaussian at all times if $\rho_0$ is Gaussian.

{
\subsection{Methodological note}\label{sec:methnote}
Our results concern mainly the evolution of the laws of the particles, i.e., mean-field models. Through quantitative mean-field limits, it is by now standard to show that, for {\it all} dynamics considered in this paper, the mean-field dynamics are approximated by {\it finite particle} ones over (sufficiently large) finite time horizons; quantitative uniform-in-time mean-field limits are recently appearing \cite{huang2025uniform,gerber2025uniform,bayraktar2025uniform}, but, for the sake of simplicity, we do not focus on them here.
For instance, following the arguments of \cite[Theorem 2.6]{gerber_mean-field_2024}, if $f$ satisfies Assumption~\ref{assumptions} and $\rho_0 \in \mathcal{P}_{6}(\mathbb{R}^d)$, then for any $T>0$ the model \eqref{eq:cbo_delta_mf} satisfies \cite[Theorem 4.3]{haug2026quantitative}
\begin{equation}\label{eq:qMFL}
\sup_{j \in [N]} \mathbb{E} \left[ \sup_{t \in [0,T]} |\Xf_t^j - X_t^j|^2 \right] \leq C N^{-1},
\end{equation}
where $\{X_.^j\}_{j=1}^N$ are $N$ i.i.d.\ copies of the solution to \eqref{eq:cbo_delta_mf} and $\{\Xf_.^j\}_{j=1}^N$ is the solution to the particle system \eqref{eq:cbo_delta_finite}. For $\rho_0 \in \mathcal{P}_{4}(\mathbb{R}^d)$, \eqref{eq:qMFL} holds conditionally on events $\Omega_M$ with $\mathbb P(\Omega_M)$ arbitrarily close to one \cite[Theorem 3.3]{haug2026quantitative}, cf.~\cite[Proposition 3.11]{B1-fornasier2021global}, while unconditionally the rate becomes $N^{-1/2}$. Being by now a routine exercise, estimates of the type \eqref{eq:qMFL} have been established for all three schemes \eqref{eq:cbo_delta_mf}, \eqref{eq:flip_scheme}, and \eqref{eq:CH} in the Master thesis \cite{haug2026quantitative}; for the Consensus Hopping scheme see also \cite[Lemma D.3]{riedl_gradient_2023} and \cite[Theorems 3.9 and 4.5]{haug2026quantitative}.
\\
Crucially, the translation from mean-field models to particles is what eventually explains the numerical convergence of the discrete-time algorithms that are actually implemented. Usually, these are Euler--Maruyama discretizations of the particle SDEs, whose time-discretization error adds to the mean-field one, as in \eqref{eq:convprob}, cf.~\cite{platen1999introduction}. In this paper, instead, the time discretization follows naturally from the Consensus Freezing scheme \eqref{eq:flip_scheme} and the Consensus Hopping scheme \eqref{eq:CH}, which we connect to \eqref{eq:cbo_delta_mf} and to each other via quantitative asymptotics (\cref{thm:CFtodCBO,thm:FStoHS}). In Section~\ref{sec:particles} we summarize the quantitative mean-field limit results of \cite{haug2026quantitative} and combine them with our mean-field convergence results to obtain global convergence in probability of the corresponding finite particle schemes, along the lines of \cite[Theorem 3.8]{B1-fornasier2021global} (see also \cite{huang2025uniform}). In particular, the implementation of the Consensus Freezing scheme in Algorithm~\ref{alg:CF} samples the particle system \eqref{eq:flip_scheme_finite} exactly at the times $t_k$ (\cref{lem:exactCF}), hence without any discretization error, in contrast to the Euler--Maruyama discretization of \eqref{eq:cbo_delta_finite} (\cref{fig:two_images}), and its global convergence in probability is established in \cref{thm:alg_CF_conv}.
\\
One more important disclaimer: in view of the isotropic nature of the stochastic perturbations implemented by the schemes \eqref{eq:cbo_delta_mf}, \eqref{eq:flip_scheme}, and \eqref{eq:CH}, for the hardest optimizations the {\it curse of dimensionality} will in general affect the constant $C$ in \eqref{eq:qMFL}, and hence all the particle estimates above. Without anisotropic noise (as introduced, e.g., in \cite{B1-carrillo2019consensus,fornasier2021anisotropic}, see also \cite{bonandin2026exploiting,fiedler2026linearly}), the methods analyzed in this paper may therefore not perform well on all high-dimensional problems. Yet, recent implementations of CBO with adapted noise, e.g., \cite{beddrich2024,fiedler2026linearly}, proved able to solve reliably very hard high-dimensional problems; for instance, CBO has recently enabled significant breakthroughs in real-life applications, e.g., trajectory planning in robotics, for which we refer to \cite{sun2026CBOrob}.}

\subsection{Technical innovations and relevant advances}\label{sec:advances} For the first time our results  connect the theory of CBO to optimal transport theory \cite{villani2009optimal} by making use of a suitable $\log$-Sobolev inequality and Talagrand inequality as in Lemma \ref{lemma:dtH} and Theorem \ref{thm:exponential_decay}. We also provide new versions and refinements of the quantitative Laplace principle \cite[Proposition 4.5]{B1-fornasier2021global} as in Lemma \ref{lemma:laplace2}, Lemma \ref{lem:laplace}, and Lemma \ref{lemma:laplace3}.
We derive new global-in-time bounds on the moments of the densities, see Theorem \ref{thm:global_bound}.
We finally obtain {\it complete descriptions} of the distribution dynamics as in Theorem \ref{lem:explicit_solution} and contractivity to invariant measures for an infinite time horizon as in Theorem \ref{thm:globconv_deltaCBO}. At the core of these advances lies a novel fixed-point argument that hinges on the refined Laplace principle and unveils distinctly the intrinsic global optimization mechanisms driving consensus-based dynamics.

\section{Global Convergence of the $\delta$-CBO Scheme}\label{sec:delta_CBO}
In this section we analyze the \eqref{eq:cbo_delta_mf} scheme. We prove that the law $\rho_t$ of $X_t$ converges, in some sense, to a Gaussian distribution centered at the global minimizer $x^*$, with variance that depends on $\delta$ and $\lambda$.
Here, we follow the blueprint of \cite[Theorem 3.7]{B1-fornasier2021global}, but with significant differences. The first is that in \cite[Theorem 3.7]{B1-fornasier2021global} one shows the exponential decay of the Wasserstein distance of the law to a Dirac delta placed at the global minimizer. (In the original CBO formulation \eqref{eq:cbo_mf} the noise is vanishing, hence the convergence to Dirac deltas.) Instead, here we show an analogous result for \eqref{eq:cbo_delta_mf}, by proving the exponential decay of the relative entropy of the law with respect to $\boldsymbol\rho_\infty=\mathcal N\left(x^*,\frac{\delta^2}{2\lambda}I_d\right)$. Crucial in our analysis is the connection to optimal transport theory, in particular the use of a suitable $\log$-Sobolev inequality and Talagrand inequality \cite[Sections §22.1–§22.2]{villani2009optimal}.
\\

We make the following assumptions on the function $f$.
\begin{assum}\label{assumptions}
	\begin{description}
	    \item[A1] The function $f:\R^d\longrightarrow\R$ is continuous, and there exists a unique $x^*\in\R^d$ such that
        $$f(x^*) = \min_{x\in\R^d}f(x)\eqqcolon \underline f.$$
        Moreover,
        there exists some constant $L_{f}>0$ such that
		\begin{equation*}
			|f(x)-f(y)|\leq L_f(1+|x|+|y|)|x-y|, \quad \forall x,y\in \RR^d\,,
		\end{equation*}
		and there exist constants $c_\ell,c_u>0$  and $M>0$ such that
		\begin{equation*}
		 f(x)-\underline f\leq c_u(|x|^2+1)\quad \forall x\in \RR^d \quad \mbox{ and } \quad c_\ell|x|^2\leq f(x)-\underline f\; \mbox{ for }|x|\geq M\,.
		\end{equation*}
        \item[A2] There exist $f_{\infty},R_0,\eta>0$ and $p\in(0,\infty)$ such that
        \begin{equation*}
            \begin{aligned}
                &(f(x)-\underline f)^\nu\geq\eta |x-x^*| \quad \text{for every }x\in B_{R_0}(x^*)\\
                &f(x)\geq \underline f+f_{\infty}  \quad \text{for every }x\notin B_{R_0}(x^*)
            \end{aligned}
        \end{equation*}
	\end{description}
\end{assum}
For well-posedness of the dynamics and regularity of the law, we only need A1; A1 and A2 are both needed to prove convergence in Section 2 by using quantitative Laplace's principle; For Section 3 again we only need A1.

Let us denote with $\rho^\alpha_t$ the law of $X_t$ at time $t$. We know that $\rho^\alpha_t$ satisfies the following differential equation
\begin{equation}\label{eq:FP_CBO_delta}
\partial_t\rho^\alpha_t(x) = \lambda \dv\left(\rho^\alpha_t(x)(x-\Xalpha(\rho^\alpha_t))\right) +\frac{\delta^2}{2}\Delta \rho^\alpha_t(x).
\end{equation}
The well-posedness of the CBO PDE was first established in \cite[Theorem 3.1, Theorem 3.2]{carrillo2018analytical}, under the assumption that $0 \leq \rho_0 \in \mathcal{P}_4(\mathbb{R}^d)$, ensuring the existence of a measure-valued solution $\rho^\alpha \in \mathcal{C}([0,T]; \mathcal{P}_2(\mathbb{R}^d))$. Higher regularity results have since been obtained in \cite{wang2025mathematical,fornasier2025regularity,fornasier2021consensus}. In particular, for our non-degenerate CBO PDE \eqref{eq:FP_CBO_delta}, one can apply the computations in \cite[Proposition 3]{wang2025mathematical} to conclude that
\begin{equation}
\rho^\alpha \in L^\infty(0,T; H^k(\mathbb{R}^d)) \quad \text{for any } k \geq 2,
\end{equation}
provided the initial data satisfies $\rho_0 \in L^1(\mathbb{R}^d) \cap H^k(\mathbb{R}^d) \cap \mathcal{P}_4(\mathbb{R}^d)$. {The smoothness of solutions justifies the calculus used in the proofs below.}\\

Given a vector $v\in\R^d$ and a matrix $\Sigma\in\R^{d\times d}$, we denote with $\mathcal N(v,\Sigma)$ the Gaussian distribution with mean $v$ and covariance matrix $\Sigma$. Moreover, given two probability distributions $\rho$ and $\nu$ we denote with $\mathcal H$ and $\mathcal I$ the relative entropy and the Fisher information, respectively. That is, 
\begin{equation*}
         \mathcal H(\rho\ |\ \nu) \coloneqq \int_{\R^d} \log\left(\frac{\rho}{\nu}\right)\rho \diff x, \qquad\mathcal I(\rho\ |\ \nu) \coloneqq \int_{\R^d} \left|\nabla \log\left(\frac{\rho}{\nu}\right)\right|^2\rho \diff x
\end{equation*}

Let us denote with $\boldsymbol\rho_\infty $ the Gaussian distribution $\mathcal N\left(x^*,\frac{\delta^2}{2\lambda}I_d\right)$. Namely,
$$\boldsymbol\rho_\infty(x) = \frac{1}{(2\pi\sigma^2)^{d/2}}e^\frac{ |x-x^*|^2}{2\sigma^2},\quad \text{with } \sigma^2 = \frac{\delta^2}{2\lambda}.$$
{The results that follow are motivated and can be informally explained by the relatively simple observation that the equation \eqref{eq:FP_CBO_delta} is {\it nearly} the Wasserstein gradient flow \cite{savare2008gradientflows} of the energy
$$
\mathcal E(\rho)=  \mathcal H(\rho\ |\ \boldsymbol\rho_\infty),
$$
if we do justify the approximation $\Xalpha(\rho^\alpha_t) \approx x^*$ at all times $t \in [0,T_*]$ (for a suitable $T_*$ we will specify below).}

\begin{lemma}\label{lemma:dtH}
Let us assume that {$\rho^\alpha$ % \in \mathcal{C}([0,T]; \mathcal{P}_2(\mathbb{R}^d))$ 
is a %weak 
smooth solution of equation \eqref{eq:FP_CBO_delta}} with fixed $\alpha,\lambda,\delta>0$, for a given $T>0$. Then 
    \begin{equation}
        \frac{\diff}{\diff t}\mathcal H(\rho^\alpha_t\ | \ \boldsymbol\rho_\infty) \leq -\frac{\delta^2}{2}\mathcal I(\rho^\alpha_t\ |\ \boldsymbol\rho_\infty ) + \lambda|\Xalpha(\rho^\alpha_t) - x^*|\mathcal I(\rho^\alpha_t\ |\ \boldsymbol\rho_\infty ) ^{1/2}.
    \end{equation}
\end{lemma}
\begin{proof}
    By definition of $\mathcal H$ and $\boldsymbol\rho_\infty$, we have
    \begin{equation*}
        \mathcal H(\rho^\alpha_t\ | \ \boldsymbol\rho_\infty) = \int_{\mathbb{R}^d}\rho^\alpha_t(x)\log\left(\frac{\rho^\alpha_t(x)}{\boldsymbol\rho_\infty(x)}\right) \diff x = \int_{\mathbb{R}^d}\rho^\alpha_t(x)\left(\log(\rho^\alpha_t(x))+\frac{ |x-x^*|^2}{2\sigma^2}\right)\diff x + \log((2\pi\sigma^2)^{d/2}).
  \end{equation*}
    Taking the derivative with respect to $t$ and applying \eqref{eq:FP_CBO_delta}, we get
        \begin{equation*}\begin{aligned}
        &\frac{\diff}{\diff t}\mathcal H(\rho^\alpha_t\ | \ \boldsymbol\rho_\infty) = \frac{\diff}{\diff t} \int_{\mathbb{R}^d}\rho^\alpha_t(x)\left(\log(\rho^\alpha_t(x))+\frac{ |x-x^*|^2}{2\sigma^2}\right)\diff x\\
        &= \int_{\mathbb{R}^d}\partial_t\rho^\alpha_t(x)\left(\log(\rho^\alpha_t(x))+\frac{ |x-x^*|^2}{2\sigma^2}\right) \diff x + \int_{\mathbb{R}^d}\rho^\alpha_t(x)\frac{\partial_t\rho^\alpha_t(x)}{\rho^\alpha_t(x)}\diff x  \\
        &= \int_{\mathbb{R}^d} \dv\left(\lambda\rho^\alpha_t(x)(x-\Xalpha(\rho^\alpha_t)) +\frac{\delta^2}{2}\nabla \rho^\alpha_t(x)\right)\left(\log(\rho^\alpha_t(x))+\frac{ |x-x^*|^2}{2\sigma^2}\right) \diff x + \frac{\diff}{\diff t}\int\rho^\alpha_t(x) \diff x \\ 
        &=- \int_{\mathbb{R}^d} \left(\lambda\rho^\alpha_t(x)(x-\Xalpha(\rho^\alpha_t)) +\frac{\delta^2}{2}\nabla \rho^\alpha_t(x)\right)^\top\nabla\left(\log(\rho^\alpha_t(x))+\frac{ |x-x^*|^2}{2\sigma^2}\right) \diff x \\
        &=- \int_{\mathbb{R}^d} \left(\lambda\rho^\alpha_t(x)(x-\Xalpha(\rho^\alpha_t)) +\frac{\delta^2}{2}\nabla \rho^\alpha_t(x)\right)^\top\left(\frac{\nabla\rho^\alpha_t(x)}{\rho^\alpha_t(x)}+\frac{(x-x^*)}{\sigma^2}\right) \diff x \\
        &= -\int_{\mathbb{R}^d}\Bigg(\frac{\delta^2}{2}\frac{|\nabla \rho^\alpha_t(x)|^2}{\rho^\alpha_t(x)} + \lambda \nabla\rho^\alpha_t(x)^\top(x-x^*)+ \lambda \nabla\rho^\alpha_t(x)^\top(x-\Xalpha(\rho^\alpha_t)) \\
        &+ \frac{2\lambda^2}{\delta^2}\rho^\alpha_t(x)(x-x^*)^\top(x-\Xalpha(\rho^\alpha_t))\Bigg) \diff x,
    \end{aligned}\end{equation*}
    where in the last equality we used the definition of $\sigma$.
    Adding and subtracting the terms $\lambda\nabla\rho^\alpha_t(x)^\top(x-x^*)$ and $\frac{2\lambda^2}{\delta^2}\rho^\alpha_t(x) |x-x^*|^2$, and defining $q(x) = \frac{\lambda}{2} |x-x^*|^2$, we obtain
    \begin{equation}\label{eq:dtH1}\begin{aligned}
         \frac{\diff}{\diff t}\mathcal H(\rho^\alpha_t\ | \ \boldsymbol\rho_\infty) &= -\int_{\mathbb{R}^d}\frac{\delta^2}{2}\frac{|\nabla \rho^\alpha_t(x)|^2}{\rho^\alpha_t(x)} + 2\nabla\rho^\alpha_t(x)^\top\nabla q(x) + \frac{2}{\delta^2}\rho^\alpha_t(x)|\nabla q(x)|^2\diff x\\
         &+ \lambda \int_{\mathbb{R}^d}(x^*-\Xalpha(\rho^\alpha_t))^\top\left(\nabla\rho^\alpha_t(x)+\frac{\rho^\alpha_t(x)}{\sigma^2}(x-x^*)\right) \diff x  =: T^1_t+T^2_t.
    \end{aligned}\end{equation}

    We now estimate $T^1_t$ and $T^2_t.$
    By definition of $q$ and $\boldsymbol\rho_\infty$ we have
    \begin{equation}\begin{aligned}
        \log(\rho^\alpha_t(x))+\frac{2}{\delta^2}q(x) = \log\left(\rho^\alpha_t(x)e^{\frac{2q(x)}{\delta^2}}\right)  =  \log\left(\frac{\rho^\alpha_t(x)}{\boldsymbol\rho_\infty(x)}\right) - \log((2\pi\sigma^2)^{d/2}),
    \end{aligned}
    \end{equation}
    which implies $$\nabla\left( \log(\rho^\alpha_t(x))+\frac{2}{\delta^2}q(x)\right) = \nabla\log\left(\frac{\rho^\alpha_t(x)}{\boldsymbol\rho_\infty(x)}\right),$$
    and therefore we get
    \begin{equation}\begin{aligned}\label{eq:T1}
        T^1_t &= -\frac{\delta^2}{2}\int_{\mathbb{R}^d}\rho^\alpha_t(x)\left|\nabla\left (\log(\rho^\alpha_t(x))+\frac{2}{\delta^2}q(x)\right)\right|^2  \diff x\\
        &= -\frac{\delta^2}{2}\int_{\mathbb{R}^d}\rho^\alpha_t(x)\left| \nabla\log\left(\frac{\rho^\alpha_t(x)}{\boldsymbol\rho_\infty(x)}\right)\right|^2 \diff x = -\frac{\delta^2}{2}\mathcal I(\rho^\alpha_t\ | \ \boldsymbol\rho_\infty).
    \end{aligned}
    \end{equation}
    To estimate $T^2_t$ notice first that 
    \begin{equation*}\begin{aligned}
    \nabla\rho^\alpha_t(x)+\frac{1}{\sigma^2}\rho^\alpha_t(x)(x-x^*) = \rho^\alpha_t(x)\nabla\left(\log(\rho^\alpha_t(x))+\frac{ |x-x^*|^2}{2\sigma^2}\right) = \rho^\alpha_t(x)\nabla\log\left(\frac{\rho^\alpha_t(x)}{\boldsymbol\rho_\infty(x)}\right).
    \end{aligned}\end{equation*}
    Using this equality in the definition of $T^2_t$ we get
    \begin{equation}\begin{aligned}\label{eq:T2}
        T^2_t &= \lambda\int_{\mathbb{R}^d}\rho^\alpha_t(x)(x^*-\Xalpha(\rho^\alpha_t))^\top\nabla\log\left(\frac{\rho^\alpha_t(x)}{\boldsymbol\rho_\infty(x)}\right) \diff x \\
        &\leq\lambda\int_{\mathbb{R}^d}|\rho^\alpha_t(x)^{1/2}(x^*-\Xalpha(\rho^\alpha_t)) |\ \left|\rho^\alpha_t(x)^{1/2}\nabla\log\left(\frac{\rho^\alpha_t(x)}{\boldsymbol\rho_\infty(x)}\right)\right|\diff x\\
        &\leq\left(\lambda\int_{\mathbb{R}^d}|\rho^\alpha_t(x)^{1/2}(x^*-\Xalpha(\rho^\alpha_t)) |^2 \diff x\right)^{1/2}
        \left(\int_{\mathbb{R}^d}\left|\rho^\alpha_t(x)^{1/2}\nabla\log\left(\frac{\rho^\alpha_t(x)}{\boldsymbol\rho_\infty(x)}\right)\right|^2\diff x\right)^{1/2}\\
        &=\lambda |x^*-\Xalpha(\rho^\alpha_t) |\left(\int_{\mathbb{R}^d}\rho^\alpha_t(x) \diff x\right)^{1/2}\left(\int_{\mathbb{R}^d}\rho^\alpha_t(x)\left|\nabla\log\left(\frac{\rho^\alpha_t(x)}{\boldsymbol\rho_\infty(x)}\right)\right|\diff x\right)^{1/2}\\
        & = \lambda |x^*-\Xalpha(\rho^\alpha_t) | \mathcal I(\rho^\alpha_t\ |\ \boldsymbol\rho_\infty)^{1/2}.
    \end{aligned}
    \end{equation}
    Using \eqref{eq:T1} and \eqref{eq:T2} in \eqref{eq:dtH1} we get the thesis.
\end{proof}

\begin{lemma}\label{lemma:laplace2}
Let {$\rho^\alpha$ % \in \mathcal{C}([0,T]; \mathcal{P}_2(\mathbb{R}^d))$ 
be a %weak 
smooth solution of equation \eqref{eq:FP_CBO_delta}}  with fixed $\alpha,\lambda,\delta>0$, for a given $T>0$. Let Assumption A2 hold and let us assume, without loss of generality, that $\underline f = 0.$ Given $r>0$ let us define $f_r\coloneqq\inf_{x\in B_r(x^*)}f(x).$ Then, for every $r\in(0,R_0]$ and every $0<q<f_{\infty}-f_r$ the following inequality holds
    \begin{equation}
        |\Xalpha(\rho^\alpha_t) - x^*|\leq\frac{(q+ f_r)^\nu}{\eta} + \frac{\delta}{\lambda^{1/2}}\frac{e^{-\alpha q}}{\rho^\alpha_t(B_r(x^*))}\left(\mathcal H(\rho^\alpha_t\ |\ \boldsymbol\rho_\infty)^{1/2} + \frac{1}{\sqrt{2}}\right).
    \end{equation}
\end{lemma}
\begin{proof}
    By the quantitative Laplace principle \cite[Proposition 4.5]{B1-fornasier2021global}, the definition of the Wasserstein distance, and Talagrand inequality \cite{villani2009optimal}, we have
    \begin{equation}\begin{aligned}
        |\Xalpha(\rho^\alpha_t)-x^*|&\leq\frac{(q+ f_r)^\nu}{\eta} + \frac{\delta}{\lambda^{1/2}}\frac{e^{-\alpha q}}{\rho^\alpha_t(B_r(x^*))}\int_{\mathbb{R}^d} |x-x^*|\diff\rho^\alpha_t(x)\\
        &\leq\frac{(q+ f_r)^\nu}{\eta} + \frac{e^{-\alpha q}}{\rho^\alpha_t(B_r(x^*))}W_2(\rho^\alpha_t,\delta_*)\\
        & \leq\frac{(q+ f_r)^\nu}{\eta} + \frac{e^{-\alpha q}}{\rho^\alpha_t(B_r(x^*))}\left(W_2(\rho^\alpha_t,\boldsymbol\rho_\infty)+W_2(\boldsymbol\rho_\infty,\delta_*)\right)\\
        &\leq \frac{(q+ f_r)^\nu}{\eta} + \frac{e^{-\alpha q}}{\rho^\alpha_t(B_r(x^*))}\left(\sqrt{2\sigma^2}\mathcal H(\rho^\alpha_t\ | \ \boldsymbol\rho_\infty)^{1/2}+\sigma\right).
    \end{aligned}\end{equation}
    The thesis follows immediately by definition of $\sigma$.
\end{proof}
The following Lemma provides a lower bound to the probability mass $\rho^\alpha_t(B_r(x^*)).$ The result is analogous to \cite[Proposition 4.6]{B1-fornasier2021global}, therefore the proof is deferred to the Appendix.
\begin{lemma}\label{lemma:ball_measure}
Let {$\rho^\alpha$ % \in \mathcal{C}([0,T]; \mathcal{P}_2(\mathbb{R}^d))$ 
be a %weak 
smooth solution of equation \eqref{eq:FP_CBO_delta}} with fixed $\alpha,\lambda,\delta>0$, for a given $T>0$ and a given initial distribution $\rho_0$. Given any $r>0$ let us denote with $\phi_r:\R^d\longrightarrow[0,1]$ the function 
$$\phi_r(x) = \begin{cases}
    \exp\left(1-\frac{r^2}{r^2- |x-x^*|^2}\right) & \text{ if }  |x-x^*|<r\\
    0 & \text{ otherwise .}
\end{cases}$$
Then, for every $t\geq 0$ 
    \begin{equation}
        \rho^\alpha_t(B_r(x^*))\geq e^{-pt}\int_{\mathbb{R}^d}\phi_r(x)\rho_0(x)\diff x
    \end{equation}

with $$p =\max\left\{2\frac{\lambda c^{1/2}(c^{1/2}r+B)}{r(1-c)^2}+\frac{\delta(1+c^2)d}{r^2(1-c)^4}, \frac{4\lambda^2r(cr^2+B)}{(2c-1)\delta^2}\right\}$$
for any $B,c>0$ such that $\sup_{t\in[0,T]}|X^\alpha(\rho^\alpha_t)-x^*|\leq B$, and $(2c-1)c\geq d(1-c)^2.$
\end{lemma}

The following Theorem proves that under suitable assumptions, the relative entropy $\mathcal H(\rho^\alpha_t|\boldsymbol\rho_\infty)$ becomes arbitrarily small in finite time, and that for finite time intervals the relative entropy decreases exponentially. 
\begin{theorem}\label{thm:exponential_decay}
   Let {$\rho^\alpha$ % \in \mathcal{C}([0,T]; \mathcal{P}_2(\mathbb{R}^d))$ 
be a %weak 
smooth solution of equation \eqref{eq:FP_CBO_delta}} with fixed $\alpha,\lambda,\delta>0$, for a given $T>0$ and a given initial distribution $\rho_0$, and let Assumptions A1-A2 hold. For any $\varepsilon_E > 0$ and $\varepsilon_R \in (0, \lambda)$, for $\alpha$ large enough, there exists 
    $$T_*\leq T_M:= -\frac{\log\left(\frac{\varepsilon_E}{\mathcal H(\rho_0 \mid \boldsymbol\rho_\infty)}\right)}{\lambda - \varepsilon_R}$$
    
    such that $\mathcal H(\rho_{T_*}\mid \boldsymbol\rho_\infty)\leq \varepsilon_E$ and such that, for any $t\in[0,T_*]$, we have
    \begin{equation}\label{eq:exponential_decay}
        \mathcal H(\rho^\alpha_t \mid \boldsymbol\rho_\infty) \leq e^{-(\lambda - \varepsilon_R) t} \mathcal H(\rho_0 \mid \boldsymbol\rho_\infty).
    \end{equation}
\end{theorem}
\begin{proof}
    We set $T_\alpha' \coloneqq \min\{t > 0 \mid \mathcal H(\rho^\alpha_t \mid \boldsymbol\rho_\infty) \leq \varepsilon_{E}\} \in \R \cup \{+\infty\}$, $T_\alpha'' \coloneqq \inf\{t > 0 \mid | \Xalpha(\rho^\alpha_t) - x^* | \geq \mathcal H(\rho_0 \mid \boldsymbol\rho_\infty)^{1/2} + 3\} \in \R \cup \{+\infty\}$ and define $T_\alpha = \min\{T_\alpha', T_\alpha'', T_M\}$.
    % , where
    % \begin{equation*}
    % T_M = -\frac{\log\left(\frac{\varepsilon_E}{\mathcal H(\rho_0 \mid \boldsymbol\rho_\infty)}\right)}{\lambda - \varepsilon_R}.
    % \end{equation*}
    We will continue the proof by showing that for $\alpha$ large enough, we have, for $t \in [0, T_\alpha)$,
    \begin{align}
        \lambda | \Xalpha(\rho^\alpha_t) - x^* | &\leq \frac{\delta^2}{2\lambda} \varepsilon_R \mathcal I(\rho^\alpha_t \mid \boldsymbol\rho_\infty)^\frac{1}{2} \quad \text{and} \label{eq:norm_bound} \\
        | \Xalpha(\rho^\alpha_t) - x^* | &\leq \mathcal H(\rho^\alpha_t \mid \boldsymbol\rho_\infty)^{1/2} + 2  \label{eq:norm_bound_2}
    \end{align}
    \cref{eq:norm_bound} shows that \eqref{eq:exponential_decay} holds for all $t \in [0, T_\alpha]$ through the following calculation:
    \begin{align*}
                \frac{\diff}{\diff t}\mathcal H(\rho^\alpha_t\ | \ \boldsymbol\rho_\infty) &\leq -\frac{\delta^2}{2}\mathcal I(\rho^\alpha_t\ |\ \boldsymbol\rho_\infty ) + \lambda|\Xalpha(\rho^\alpha_t) - x^*|\mathcal I(\rho^\alpha_t\ |\ \boldsymbol\rho_\infty ) ^{1/2} \\
                &\leq -\frac{\delta^2}{2}\mathcal I(\rho^\alpha_t\ |\ \boldsymbol\rho_\infty ) + \frac{\delta^2}{2\lambda} \varepsilon_R \mathcal I(\rho^\alpha_t \mid \boldsymbol\rho_\infty) \leq -\frac{2\lambda}{\delta^2} \left(\frac{\delta^2}{2} - \frac{\delta^2}{2\lambda} \varepsilon_R\right) \mathcal H(\rho^\alpha_t\ |\ \boldsymbol\rho_\infty) \\
                &\leq -(\lambda - \varepsilon_R) \mathcal H(\rho^\alpha_t\ |\ \boldsymbol\rho_\infty),
    \end{align*}
    where we used \cref{lemma:dtH} in the first inequality, \eqref{eq:norm_bound} in the second, and the fact that $\boldsymbol\rho_\infty$ satisfies $\log$-Sobolev inequality \cite[Sections §22.1–§22.2]{villani2009optimal} in the third inequality. In particular, $\mathcal H(\rho^\alpha_t\ | \ \boldsymbol\rho_\infty)$ is monotonically decreasing on $[0, T_\alpha]$, and thus, \cref{eq:norm_bound_2} shows that $T_\alpha'' > T_\alpha$ through continuity of $\mathcal H(\rho^\alpha_t \mid \boldsymbol\rho_\infty)$ and $| \Xalpha(\rho^\alpha_t) - x^* |$ in time, and thus that $\mathcal H(\rho_{T_\alpha} | \boldsymbol\rho_\infty) \leq \varepsilon_E$, finishing the proof.\\
    To show \cref{eq:norm_bound,eq:norm_bound_2}, we first note that thanks to the LSI of $\boldsymbol\rho_\infty$, for all $t \in [0, T_\alpha)$, we have that
    \begin{equation*}
        \mathcal I(\rho^\alpha_t \mid \boldsymbol\rho_\infty) \geq \frac{2\lambda}{\delta^2} \mathcal H(\rho^\alpha_t \mid \boldsymbol\rho_\infty) \geq \frac{2\lambda}{\delta^2} \varepsilon_E \eqqcolon \varepsilon_I.
    \end{equation*}
    Furthermore, for any $r > 0$  $| \Xalpha(\rho^\alpha_t) - x^* |$ is bounded for $t \in [0, T_\alpha)$ by the definition of $T_\alpha''$. Thus, we can apply \cref{lemma:ball_measure} with some $p$ independent of $\alpha$ to get, for all $t \in [0, T_\alpha)$,
    \begin{equation*}
        \rho^\alpha_t(B_r(x^*)) \geq e^{-p t} \int_{\mathbb{R}^d}\phi_r(x) \rho_0(x) \diff x \geq e^{-p T_M} \int_{\mathbb{R}^d}\phi_r(x) \rho_0(x) \diff x \eqqcolon M_r > 0.
    \end{equation*}
    We then first choose $q, r$ such that
    \begin{equation}\label{eq:choice_q_r}
        \frac{(q+ f_r)^\nu}{\eta} \leq \frac{\delta^2}{6\lambda^2} \varepsilon_I^{\frac{1}{2}} \quad \text{and} \quad \frac{(q+ f_r)^\nu}{\eta} \leq 1.
    \end{equation}
    With this choice of $q, r$ and using \cref{lemma:laplace2}, we then have
    \begin{align*}
        |\Xalpha(\rho^\alpha_t) - x^*| &\leq\frac{(q+ f_r)^\nu}{\eta} + \frac{\delta}{\lambda^{1/2}}\frac{e^{-\alpha q}}{\rho^\alpha_t(B_r(x^*))}\left(\mathcal H(\rho^\alpha_t\ |\ \boldsymbol\rho_\infty)^{1/2} + \frac{1}{\sqrt{2}}\right) \\
        &\leq \frac{\delta^2}{6\lambda^2} \mathcal I(\rho^\alpha_t \mid \boldsymbol\rho_\infty)^{\frac{1}{2}} + \underbrace{\frac{\delta}{\lambda^{1/2}}\frac{e^{-\alpha q}}{M_r} \frac{\lambda^{1/2}\sqrt{2}}{\delta}}_{\leq \frac{\delta^2}{6\lambda^2} \text{ for } \alpha \text{ big enough}}\mathcal I(\rho^\alpha_t \mid \boldsymbol\rho_\infty)^{\frac{1}{2}} + \underbrace{\frac{\delta}{\lambda^{1/2}}\frac{e^{-\alpha q}}{M_r}\frac{1}{\sqrt{2}}}_{\leq \frac{\delta^2}{6\lambda^2} \varepsilon_I \text{ for } \alpha \text{ big enough}} \\
        &\leq \frac{\delta^2}{2\lambda^2} \mathcal I(\rho^\alpha_t \mid \boldsymbol\rho_\infty)^{\frac{1}{2}},
    \end{align*}
    which proves \cref{eq:norm_bound}. For \cref{eq:norm_bound_2}, we can repeat the same argument:
    \begin{align*}
        |\Xalpha(\rho^\alpha_t) - x^*| &\leq\frac{(q+ f_r)^\nu}{\eta} + \frac{\delta}{\lambda^{1/2}}\frac{e^{-\alpha q}}{\rho^\alpha_t(B_r(x^*))}\left(\mathcal H(\rho^\alpha_t\ |\ \boldsymbol\rho_\infty)^{1/2} + \frac{1}{\sqrt{2}}\right) \\
        &\leq 1 + \underbrace{\frac{\delta}{\lambda^{1/2}}\frac{e^{-\alpha q}}{M_r} }_{\leq 1 \text{ for } \alpha \text{ large enough}}\mathcal H(\rho^\alpha_t\ |\ \boldsymbol\rho_\infty)^{1/2} + \underbrace{\frac{\delta}{\lambda^{1/2}}\frac{e^{-\alpha q}}{M_r}\frac{1}{\sqrt{2}}}_{\leq 1 \text{ for } \alpha \text{ large enough}} \\
        &\leq \mathcal H(\rho^\alpha_t \mid \boldsymbol\rho_\infty)^{1/2} + 2.
    \end{align*}
\end{proof}
\begin{remark}
    It is important to notice that, while in \cite[Theorem 3.7]{B1-fornasier2021global} the rate of exponential convergence was given by $2 \lambda - d \sigma^2 >0$,  in \eqref{eq:exponential_decay} the  variance of the diffusion $\delta^2/2$ does not play any  role and only the drift parameter $\lambda$ determines the rate of convergence, which, remarkably, does not depend on the dimension $d$ as well. This is due to the fact that by fixing the variance $\delta^2/2$ not only we prevent concentration, but we also avoid the occurrence of dominating and exploding diffusion.
\end{remark}
\begin{remark}
The law $\nu_t = \operatorname{Law}(X_t)$ of the Langevin dynamics $d X_t = - \nabla f(X_t) dt + \sqrt{2 \sigma} dB_t$ fulfills the linear Fokker-Planck equation
$$
\partial_t \nu_t(x) = \dv\left( \nabla f(x) \nu_t(x)\right) +\sigma \Delta \nu_t(x).
$$
Its invariant measure is the Gibbs measure $\nu_\infty \propto e^{-\frac{f}{\sigma}}$, which depends on $f$. Remarkably the attractor for $\alpha\gg1$ of \eqref{eq:FP_CBO_delta} is the Gaussian $\boldsymbol\rho_\infty$, which depends on $x^* =\arg \min f$, but not on the entire function $f$. The meaning of Theorem \ref{thm:exponential_decay}, as also pointed out in \cite{B1-fornasier2021global}, is that \eqref{eq:cbo_delta_finite} performs a canonical convexification of the original nonconvex optimization of $f$ as the number of particles $N$ tends to $\infty$.
\end{remark}
While Theorem \ref{thm:exponential_decay} shows that the dynamics \eqref{eq:cbo_delta_mf} converges globally with exponential rate up to a given prescribed accuracy and until a suitable finite horizon time $T_*$, it says nothing about the behavior of the dynamics for $t> T_*$. This issue has also been noted in connection with the analogous result in \cite[Theorem~3.7]{B1-fornasier2021global}, which leads to \eqref{eq:convprob} and has at times been interpreted as a form of truncated convergence (see, for example, \cite{bianchi2025}). (A statement such as Theorem~\ref{thm:exponential_decay} does not describe the asymptotic behavior as $t\to +\infty$. Yet it provides a characterization over a suitable time horizon $T_*$ at any prescribed accuracy $\varepsilon_E$ and does offer a valid basis for asserting \emph{numerical} convergence, which is what really matters in practice.)
 Therefore, in this paper, we present new results that provide a precise description of the dynamics after $T_*$ and, unlike \cite{bianchi2025}, offer an in-depth analysis of the density behavior. We start first with an analysis of the Consensus Freezing scheme \eqref{eq:flip_scheme} which allows us to build some useful tools and then we come back to an infinite-time horizon convergence result for \eqref{eq:cbo_delta_mf} in Section \ref{sec:delta_CBO2}.

\section{The Consensus Freezing Scheme}\label{sec:CF}
{Discretizing} \cref{eq:cbo_delta_finite} directly using Euler-Maruyama has the drawback that for fixed $\lambda$ and $\delta$, our choice of $\Delta t > 0$ is limited: if $\Delta t$ comes close to $\frac{1}{\lambda}$ or $\frac{1}{\delta}$, the Euler-Maruyama discretization \eqref{eq:delta-cbo_finite_discrete} ceases to be a faithful representation of the continuous dynamics, as the Lipschitz constant of the right-hand side of \eqref{eq:cbo_delta_finite} is $\lambda$ and $\delta$. To circumvent this issue, we propose a variant of \eqref{eq:cbo_delta_finite}, which we call the \emph{Consensus Freezing scheme}:
\begin{equation}\tag{CF}\label{eq:flip_scheme_2}
    \diff X_t = -s\cdot\lambda\left(X_t - \Xalpha(\varrho^{\alpha,s,\Delta t}_{t_i})\right) \diff t + \sqrt{s}\cdot\delta \diff B_t, \quad \text{for } t \in [t_i, t_{i+1}),
\end{equation}
where $\varrho^{\alpha,s,\Delta t}_t \coloneqq \law(X_t)$, and $t_i \coloneqq  i \cdot \Delta t$ for some step size $\Delta t > 0$. Above, we introduced a \emph{speed parameter} $s$, which will become useful in later {asymptotics}. In the corresponding finite-particle system, denoting  $\widehat\varrho^{\alpha,s,\Delta t,N}_t = \frac{1}{N}\sum_{j=1}^N\delta_{\Xf^j_t},$
\begin{equation}\tag{CF$_N$}\label{eq:flip_scheme_finite}
    \diff \Xf^j_t = -s \cdot \lambda(\Xf^j_t - \Xalpha(\widehat\varrho^{\alpha,s,\Delta t,N}_{t_i})) \diff t + \sqrt{s} \cdot \delta \diff B^j_t, \quad \text{for } t \in [t_i, t_{i+1}),
\end{equation}
each particle follows an Ornstein–Uhlenbeck process for $t \in [i\cdot\Delta t, (i+1)\cdot\Delta t)$; we can thus hope to solve \eqref{eq:flip_scheme_finite} \emph{analytically}, {without suffering any drawback if $\Delta t$ is chosen large}.

\subsection{Global Convergence of the Consensus Freezing scheme }
\label{sec:gcCF}
% \begin{assum}\label{assum1}
% 	We assume the following properties for the objective function.
% 	\begin{enumerate}
% 		\item $f:~\RR^d\to \RR$ is bounded from below by $\underline f := \min f$, and there exists some constant $L_{f}>0$ such that
% 		\begin{equation*}
% 			|f(x)-f(y)|\leq L_f(1+|x|+|y|)|x-y|, \quad \forall x,y\in \RR^d\,.
% 		\end{equation*}
% 		\item There exist constants $c_\ell,c_u>0$  and $M>0$ such that
% 		\begin{equation*}
% 		 f(x)-\underline f\leq c_u(|x|^2+1)\quad \forall x\in \RR^d \quad \mbox{ and } \quad c_\ell|x|^2\leq f(x)-\underline f\; \mbox{ for }|x|\geq M\,.
% 		\end{equation*}
% 	\end{enumerate}
% \end{assum}
\begin{lemma}\cite[Lemma 3.2, 3.3]{carrillo2018analytical}\label{lem: xalpha bound}
Suppose that $f(\cdot)$ satisfies Assumption \ref{assumptions}-\textbf{A1}, then it holds that
\begin{equation}\label{lemeq2}
  \frac{\int_{\RR^d}|x|^2e^{-\alpha f(x)}\nu(\dd x)}{\int_{\RR^d}e^{-\alpha f(x)}\nu(\dd x)}  \leq b_1+b_2\int_{\RR^d}|x|^2 \dd\nu \quad \forall \nu\in \mathcal{P}_2(\mathbb{R}^d)\,,
\end{equation}
where 
\begin{equation*}
    b_1=M^2+b_2 \quad \text{ and } \quad b_2 = 2\,\frac{c_u}{c_\ell}\left(1+\frac{1}{\alpha \, M^2 \, c_\ell}\right).
\end{equation*}
Moreover, for any $(\mu,\nu)\in \mc{P}_{4}(\RR^d)\times\mc{P}_{4}(\RR^d)$ with fourth moment bounded by $R>0$, there exists some $L_X>0$ depending on $R$ such that
\begin{equation}\label{eq:stability}
    |\Xalpha(\mu)-\Xalpha(\nu)|\leq L_XW_2(\mu,\nu)\,.
\end{equation}
\end{lemma}

\begin{remark}
    In fact, as it has been already observed in \cite[Remark 1]{carrillo2018analytical}, the constant $b_2$ can be chosen to be independent of $\alpha$ for $\alpha\geq 1$.
\end{remark}

% Without loss of generality, we will assume $s=1$ in this section, which simplifies \cref{eq:flip_scheme_2} to
In the following, we assume $\varrho^{\alpha,s,\Delta t}_0 = \varrho_0=\mc{N}(\mu_0,\Sigma_0)$ for a given $\mu_0\in\R^d$ and we study 
\begin{align}\label{CBO}
\begin{cases}
     &dX_t=-s\lambda (X_t-\Xalpha(\varrho^{\alpha,s,\Delta t}_{t_i}))dt+\sqrt{s}\delta dB_t,\quad  t\in [t_i,t_{i+1}]\,,\\
     &\Xalpha(\varrho^{\alpha,s,\Delta t}_{t_i})=\displaystyle\frac{\int_{\RR^d} x e^{-\alpha f(x)}\varrho^{\alpha,s,\Delta t}_{t_i}(\diff x)}{\int_{\RR^d}e^{-\alpha f(x)}\varrho^{\alpha,s,\Delta t}_{t_i}(\diff x)},\quad \varrho^{\alpha,s,\Delta t}_{t_i}=\mbox{Law}(X_{t_i})
\end{cases}
\end{align}
% where $t_i=i\Delta t$ and $\rho_0=\mc{N}(\mu_0,\Sigma_0)$.
The choice of $\varrho_0$ as a Gaussian distribution will be removed later in this section but it is useful as a technical step.
The above $d$–dimensional Ornstein–Uhlenbeck SDE has a unique strong solution
\begin{equation}
    X_t=\Xalpha(\varrho^{\alpha,s,\Delta t}_{t_i})+e^{-s\lambda(t-t_i)}(X_{t_i}-\Xalpha(\varrho^{\alpha,s,\Delta t}_{t_i}))+\sqrt{s}\delta \int_{t_i}^t e^{-\lambda(t-\tau)}dB_\tau,\quad t\in [t_i,t_{i+1}]\,.
\end{equation}
Then it holds that $\varrho^{\alpha,s,\Delta t}_t=\mc{N}(\mu^{\alpha,s,\Delta t}_t,\Sigma^{\alpha,s,\Delta t}_t)$ with 
\begin{equation}\label{eq:mu_t}
    \mu^{\alpha,s,\Delta t}_t:=\EE[X_t]=\Xalpha(\varrho^{\alpha,s,\Delta t}_{t_i})+e^{-s\lambda(t-t_i)}(\mu^{\alpha,s,\Delta t}_i-\Xalpha(\varrho^{\alpha,s,\Delta t}_{t_i}))
\end{equation}
and 
\begin{equation}
    \Sigma_t^{\alpha,s,\Delta t}:=\mbox{Cov}(X_t)=e^{-2s\lambda(t-t_i)}\Sigma^{\alpha,s,\Delta t}_i+\frac{\delta^2}{2\lambda}(1-e^{-2s\lambda(t-t_i)})I_d\,,
\end{equation}
where $\mu^{\alpha,s,\Delta t}_i=\EE[X_{t_i}]$ and $\Sigma^{\alpha,s,\Delta t}_i=\mbox{Cov}(X_{t_i})$. Thus we have the following iterative formula:
\begin{equation}\label{iterative'}
    \begin{cases}
        \mu^{\alpha,s,\Delta t}_{i+1}=(1-\omega_{\Delta t})\Xalpha(\varrho^{\alpha,s,\Delta t}_{t_i})+\omega_{\Delta t}\mu^{\alpha,s,\Delta t}_i=\omega_{\Delta t}^{i+1}\mu_0+(1-\omega_{\Delta t})\sum_{k=0}^{i}\omega_{\Delta t}^{i-k}\Xalpha(\varrho^{\alpha,s,\Delta t}_{t_k}) \\
        \Sigma_{i+1}^{\alpha,s,\Delta t}=\omega_{\Delta t}^2\Sigma_i+\frac{\delta^2}{2\lambda}(1-\omega_{\Delta t}^2)I_d=\omega_{\Delta t}^{2(i+1)}\Sigma_i^{\alpha,s,\Delta t}+\frac{\delta^2}{2\lambda}(1-\omega_{\Delta t}^{2(i+1)})I_d
    \end{cases}
\end{equation}
where we have introduced the notation $\omega_{\Delta t}=e^{-s\lambda(\Delta t)}$. 

To study invariant measures of the dynamics \eqref{CBO}, let us assume for a moment that there exists a point $\boldsymbol\mu^\alpha_\infty\in\RR^d$ such that $X^\alpha\Big(\mc{N}\big(\boldsymbol\mu^\alpha_\infty,\frac{\delta^2}{2\lambda}I_d\big)\Big) = \boldsymbol\mu^\alpha_\infty$. In this case, the measure $\boldsymbol\rho^\alpha_\infty \coloneqq \mc{N}\big(\boldsymbol\mu^\alpha_\infty,\frac{\delta^2}{2\lambda}I_d\big)$ would be an invariant measure of \eqref{CBO}: indeed, if $\varrho^{\alpha,s,\Delta t}_{t_i}=\boldsymbol\rho^\alpha_\infty$, then $\mu^{\alpha,s,\Delta t}_i=\boldsymbol\mu^\alpha_\infty$ and $\Sigma_i^{\alpha,s,\Delta t}=\frac{\delta^2}{2\lambda}I_d$; it follows from \eqref{iterative'} that 
\begin{equation}
    \mu^{\alpha,s,\Delta t}_{i+1}= \boldsymbol\mu^\alpha_\infty\mbox{ and }\Sigma_{i+1}^{\alpha,s,\Delta t}=\frac{\delta^2}{2\lambda}I_d,
\end{equation}
and so $\varrho^{\alpha,s,\Delta t}_{i+1}=\boldsymbol\rho^\alpha_\infty$ and $\boldsymbol\rho^\alpha_\infty$ is invariant to \eqref{CBO}. Hence, any fixed point of the mapping
\begin{equation}\label{eq:T_alpha}
    \begin{split}
        \mc{T}^\alpha\colon\RR^d&\to \RR^d \\
        \mu &\mapsto X^\alpha\Big(\mc{N}\big(\mu,\frac{\delta^2}{2\lambda}I_d\big)\Big).
    \end{split}
\end{equation} 
gives rise to an invariant measure of \eqref{CBO}. To prove the existence of such a fixed point, we will show that for $r > 0$ and $\alpha$ sufficiently large, $\mc{T}^\alpha\colon B_r(x^*) \to B_r(x^*)$ is a contraction mapping on the closed ball $B_r(x^*)$. For this purpose, we need the following version of the Laplace principle:
\begin{lemma}\label{lem:laplace}
    Assume $f(\cdot)$ satisfies Assumption \ref{assumptions}-\textbf{A1} with a unique global minimizer $x^*$, and let $\nu$ be a probability measure satisfying
    \begin{equation}\label{eq:nuassum}
        \EE_\nu[|x|^2]\leq c_2\quad \mbox{ and }\quad  \nu (B_s(x^*))\geq C_s>0\quad \mbox{ for any } s>0\,.
    \end{equation}
Define the probability measure
    \begin{equation}\label{eq:eta}
        \eta_\nu^\alpha(A):=\frac{\int_{A}e^{-\alpha f(x)}\nu(\diff x)}{\int_{\RR^d}e^{-\alpha f(y)}\nu(dy)},\quad A\in \mc{B}(\RR^d).
    \end{equation}
    Then for any $\varepsilon>0$, it holds 
    \begin{equation}\label{eq: massconv}
        \eta_\nu^\alpha(\{x \in \RR^d :|x-x^*|\geq \varepsilon\})\longrightarrow 0  \quad \mbox{ as }  \quad \alpha \to \infty\,
    \end{equation}
and
    \begin{equation}\label{eq: meanconv}
       \EE_{\eta_\nu^\alpha}[x]:= \frac{\int_{\RR^d} x e^{-\alpha f(x)}\nu(\diff x)}{\int_{\RR^d}e^{-\alpha f(x)}\nu(\diff x)}\to x^*\mbox{ as }\alpha\to\infty.
    \end{equation}
\end{lemma}
\begin{proof}
The proof is almost identical to \cite[Theorem 3.3, Corollary 3.4]{huang2025faithful}. We will only provide a sketch here. Let $\varepsilon>0$ be fixed, and let us define
\begin{equation*}
    \overline f_\varepsilon:=\sup\{f(x): |x-x^*|\leq \varepsilon\} \mbox{ and } \underline f_\varepsilon:=\inf\{f(x): |x-x^*|\geq \varepsilon\}. 
\end{equation*}
Since $f(\cdot)$ satisfies Assumption \ref{assumptions}-\textbf{A1} we know that for any $\varepsilon>0$, there exists some $\varepsilon'>0$ such that
\begin{equation}\label{diff f eps}
    \overline f_{\varepsilon'} \leq \frac{1}{2} (\underline f_\varepsilon + \underline f)\,.
\end{equation}
It follows from the arguments in \cite[Theorem 3.3]{huang2025faithful} that 
\begin{align}\label{massconverge}
    \eta_\nu^\alpha(\{x \in \RR^d :|x-x^*|\geq \varepsilon\} )
   &\leq \frac{e^{-\alpha (\overline f_{\varepsilon'} - \underline f)}}{\nu(B_{\varepsilon'}(x^*))}\leq \frac{e^{-\alpha (\overline f_{\varepsilon'} - \underline f)}}{C_{\varepsilon'}}\,\; \to 0 \; \mbox{ when } \;\alpha \to \infty.
\end{align}
Next we show its mean converges to $x^*$. Indeed, it is easy to compute
\begin{align}\label{distx*}
    |\EE_{\eta_\nu^\alpha}[x]-x^*|&\leq \int_{|x-x^*|\leq \varepsilon}|x-x^*|\eta_\nu^\alpha(\dd x)+\int_{|x-x^*|\geq \varepsilon}|x-x^*|\eta_\nu^\alpha(\dd x)\notag\\
  &\leq \varepsilon+ \bigg(\eta_\nu^\alpha(\{x:|x-x^*|\geq \varepsilon\}) \bigg)^{\frac{1}{2}} \left(\int_{|x-x^*|\geq \varepsilon}|x-x^*|^2\eta_\nu^\alpha(\dd x) \right)^{\frac{1}{2}}\notag\\
  &\leq \varepsilon + \bigg(\eta_\nu^\alpha(\{x:|x-x^*|\geq \varepsilon\}) \bigg)^{\frac{1}{2}} \bigg( 2 |x^*|^2 + 2\,\EE_{\eta_\nu^\alpha}[|x|^2]\bigg)^{\frac{1}{2}}
\end{align}
 Applying Lemma \ref{lem: xalpha bound}, we have 
\begin{equation}\label{bound second moment eta}
  \EE_{\eta_\nu^\alpha}[|x|^2]\leq b_1+b_2 \EE_\nu[|x|^2]\leq b_1+b_2c_2
\end{equation}
and then it holds
\begin{equation*}
    |\EE_{\eta_\nu^\alpha}[x]-x^*|\leq \varepsilon+\bigg(\eta_\nu^\alpha(\{x:|x-x^*|\geq \varepsilon\}) \bigg)^{\frac{1}{2}} \bigg(2 |x^*|^2 + 2b_1 + 2b_2c_2\bigg)^{\frac{1}{2}}.
\end{equation*}
Thus \eqref{eq: meanconv} follows by using \eqref{massconverge}.
\end{proof}
\begin{remark}\label{rmk:uniform_convergence}
   From the above proof, in taking the limit \(\alpha\to\infty\) in \eqref{eq: massconv} and \eqref{eq: meanconv}, the parameter \(\alpha\) depends only on \(x^*\), \(b_1\), \(b_2\), \(C_s\) and \(c_2\).
\end{remark}

\begin{theorem}\label{thm:invariant}
    For any $r>0$ and 
    $\alpha$ sufficiently large depending on $r$,
    the map $\mc{T}^\alpha$ defined in \cref{eq:T_alpha} restricted to $B_r(x^*)$ is a contraction mapping with Lipschitz constant $L_\alpha$, where $L_\alpha\to 0$ as $\alpha\to\infty$.
    Furthermore, for $\alpha$ sufficiently large, the dynamics \eqref{CBO} has a unique invariant measure $\boldsymbol\rho_\infty^\alpha$ with the property that $\boldsymbol\mu_\infty^\alpha \coloneqq \EE_{\boldsymbol{\rho}_\infty^\alpha}[x] \in B_r(x^*)$, and $\boldsymbol{\rho}_\infty^\alpha$ is of the form
    \begin{equation}\label{eq:invariant}
        \boldsymbol\rho_\infty^\alpha=\mc{N}\left(\boldsymbol\mu_\infty^\alpha,\frac{\delta^2}{2\lambda}I_d\right)=\mc{N}\left(\mc{T}^\alpha(\boldsymbol\mu_\infty^\alpha),\frac{\delta^2}{2\lambda}I_d\right)=\mc{N}\left(\Xalpha(\boldsymbol\rho_\infty^\alpha),\frac{\delta^2}{2\lambda}I_d\right).
    \end{equation}
    Furthermore it holds that
 \begin{equation}
     W_2(\boldsymbol\rho_\infty^\alpha, \boldsymbol{\rho}_\infty)\to 0\quad \mbox{ as }\alpha\to\infty \,.
 \end{equation}
 with $\boldsymbol{\rho}_\infty= \mc{N}\left(x^*,\frac{\delta^2}{2\lambda}I_d\right)$\,.
\end{theorem}
\begin{remark}
    We note here that the invariant measure found in the above theorem is also an invariant measure of \eqref{eq:cbo_delta_mf}. {Incidentally, we also mention that Theorem~\ref{thm:invariant} provides the explicit steady states
of the mean-field limit of genetic algorithms with CBO-type parent selection
derived in~\cite{borghi2025kinetic}, where their characterization was left open.}
\end{remark}
\begin{proof}
\textbf{Step 1:} Let us first verify that for $\alpha$ sufficiently large it holds that  $\mathcal{T}^\alpha (\mu)\in B_r(x^*)$ for any 
$\mu\in B_r(x^*)$. In this case, we let $\nu=\mc{N}(\mu,\frac{\delta^2}{2\lambda}I_d)$ in Lemma \ref{lem:laplace}. Since $\mu\in B_r(x^*)$, one can verify that 
\begin{align}\label{lowerbound}
    \nu(B_s(x^*))=\int_{B_s(x^*)}\frac{1}{(2\pi\sigma^2)^{d/2}}\exp\left(-\frac{|x-\mu|^2}{2\sigma^2}\right) \diff x\geq \frac{|B_s(x^*)|}{(2\pi\sigma^2)^{d/2}} \exp\left(-\frac{|r+s|^2}{2\sigma^2}\right)=C_s>0
\end{align}
holds, with $\sigma^2 = \frac{\delta^2}{2\lambda}$, regardless of the specific $\mu\in B_r(x^*)$.
Moreover, it also holds
\begin{equation}
    \EE_\nu\left[|x|^2\right]\leq  |\mu|^2+\frac{d\delta^2}{2\lambda}\leq (|x^*|+r)^2+\frac{d\delta^2}{2\lambda}=c_2.
\end{equation}
Thus according to Lemma \ref{lem:laplace}  and Remark \ref{rmk:uniform_convergence}
\begin{equation}\label{Taxconv}
    \mathcal{T}^\alpha (\mu)\to x^*  \mbox{ as }\alpha\to\infty,
\end{equation}
 uniformly with respect to $\mu \in B_r(x^*)$.
This implies that for $\alpha$ sufficiently large, one has $\mathcal{T}^\alpha(\mu) \in B_r(x^*)$ for any $\mu\in B_r(x^*)$.

\textbf{Step 2:} Next, we show $\mathcal{T}^\alpha$ is a contraction mapping when $\alpha$ is sufficiently large, and we will use the notations
\begin{equation*}
    D(\mu) \coloneqq \int_{\RR^d}e^{-\alpha f(x)}\mc{N}(\mu,\frac{\delta^2}{2\lambda}I_d)(\diff x), \quad\eta^\alpha_{\mu}(\diff x):=\frac{ e^{-\alpha f(x)}\mc{N}(\mu,\frac{\delta^2}{2\lambda}I_d)(\diff x)}{D(\mu)}, \quad N(\mu) \coloneqq \EE_{\eta^\alpha_{\mu}}[x]D(\mu).
\end{equation*}
Then it holds that
\begin{equation}
    \mc{T}^\alpha(\mu)=\EE_{\eta^\alpha_{\mu}}[x]\mbox{ and }N(\mu)=D(\mu)\mc{T}^\alpha(\mu)\,.
\end{equation}
Notice that
\begin{equation*}
    \mc{N}\left(\mu,\frac{\delta^2}{2\lambda}I_d\right)(\diff x)=\frac{1}{(2\pi\sigma^2)^{d/2}}\exp\left(-\frac{|x-\mu|^2}{2\sigma^2}\right),\quad \sigma^2:=\frac{\delta^2}{2\lambda}\,.
\end{equation*}
Some simple computations lead to
\begin{equation}
    \nabla N(\mu)=\frac{1}{\sigma^2}D(\mu)\EE_{\eta^\alpha_{\mu}}[x(x-\mu)^\top]\quad \mbox{ and }\nabla D(\mu)=\frac{1}{\sigma^2} D(\mu)(\mc{T}^\alpha(\mu)-\mu)\,,
\end{equation}
which implies 
\begin{align}
    \nabla \mathcal{T}^\alpha(\mu)&=\frac{1}{\sigma^2}\EE_{\eta^\alpha_{\mu}}\left[x(x-\mu)^\top\right]-\frac{1}{\sigma^2}\mc{T}^\alpha(\mu)\left(\mc{T}^\alpha(\mu)-\mu\right)^\top\notag\\
    &=\frac{1}{\sigma^2}\EE_{\eta^\alpha_{\mu}}\left[(x-\mc{T}^\alpha(\mu))\left(x-\mc{T}^\alpha(\mu)\right)^\top\right]=:\frac{1}{\sigma^2} \mbox{Cov}(\eta^\alpha_{\mu})
\end{align}
Then we have
\begin{equation}\label{57}
    \|\nabla \mathcal{T}^\alpha(\mu)\|\leq \frac{1}{\sigma^2}\EE_{\eta^\alpha_{\mu}}\left[|x-\mc{T}^\alpha(\mu))|^2\right]\leq \frac{1}{\sigma^2}\left(\EE_{\eta^\alpha_{\mu}}[|x-x^*|^2]+|x^*-\mc{T}^\alpha(\mu)|^2\right)\,,
\end{equation}
and  according to \eqref{Taxconv} it holds $|x^*-\mc{T}^\alpha(\mu)|\to 0$ as $\alpha\to\infty$. Next we show that $\EE_{\eta^\alpha_{\mu}}[|x-x^*|^2]\to 0$.

Indeed, according to Assumption \ref{assumptions}, we know that on the ball $B_1(\mu)$, it holds that
    \begin{equation}
        f(x)\leq c_u((|\mu|+1)^2+1)\leq c_u ((|x^*|+r+1)^2+1)=:S\,.
    \end{equation}
This leads to
\begin{equation*}
    D(\mu)\geq \int_{B_1(\mu)}e^{-\alpha f(x)}\mc{N}\left(\mu,\frac{\delta^2}{2\lambda}I_d\right)(\diff x)\geq e^{-\alpha S} \mc{N}(B_1(\mu))\geq \frac{1}{(2\pi\sigma^2)^{d/2}}\exp\left(-\frac{1}{2\sigma^2}\right)\omega_d e^{-\alpha S}\,,
\end{equation*}
where $\omega_d$ is the volume of the unit ball in $\R^d$. Now we split
\begin{align}
    \EE_{\eta_\mu^\alpha}[|x|^4]&=\frac{\int_{|x|\leq K}|x|^4e^{-\alpha f}\mc{N}(\diff x)}{D(\mu)}+\frac{\int_{|x|\geq K}|x|^4e^{-\alpha f}\mc{N}(\diff x)}{D(\mu)}\notag\\
    &\leq K^4+\frac{\int_{|x|\geq K}|x|^4e^{-\alpha f}\mc{N}(\diff x)}{D(\mu)}\,,
\end{align}
where $K\geq M$ to be determined. Notice from Assumption \ref{assumptions}-\textbf{A1} that for $|x|\geq K\geq M$, we have $e^{-\alpha f(x)}\leq e^{-c_\ell \alpha|x|^2}$. This means 
\begin{equation}
   % \mc{N}(\diff x)=
    \frac{1}{(2\pi\sigma^2)^{d/2}}\exp\left(-\frac{|x-\mu|^2}{2\sigma^2}\right)\leq  \frac{1}{(2\pi\sigma^2)^{d/2}} e^{\frac{\mu^2}{2\sigma^2}}e^{-\frac{|x|^2}{4\sigma^2}}\,.
\end{equation}
Then we have
\begin{equation}
  \frac{\int_{|x|\geq K}|x|^4e^{-\alpha f}\mc{N}(\diff x)}{D(\mu)}\leq \frac{1}{\omega_d}e^{\frac{\mu^2+1}{2\sigma^2}}e^{\alpha S}\int_{|x|\geq K}|x|^4 e^{-\beta_\alpha|x|^2}\diff x
\end{equation}
with $\beta_\alpha:=\alpha c_\ell+\frac{1}{4\sigma^2}$. One can further compute that
\begin{align}
    \int_{|x|\geq K}|x|^4 e^{-\beta_\alpha|x|^2}\diff x \leq C(d)\int_K^\infty r^{d+3}e^{-\beta_\alpha r^2}dr\leq C(d,K)\beta_\alpha^{-1}e^{-\beta_\alpha K^2}\leq C(d,K)\beta_\alpha^{-1}e^{-\alpha c_\ell K^2}
\end{align}
when $K$ is sufficiently large. Therefore, it holds that
\begin{align}
  \frac{\int_{|x|\geq K}|x|^4e^{-\alpha f}\mc{N}(\diff x)}{D(\mu)}
  \leq C(d,\sigma,\mu,K)\beta_\alpha^{-1}e^{-\alpha (c_\ell K^2-S)}\leq C(d,\sigma,\mu,K)e^{-\alpha (c_\ell K^2-S)}
\end{align}
where we have used $\beta_\alpha\geq 1/(4\delta^2)$. Meanwhile, if one chooses sufficiently large $K>\sqrt{\frac{S}{c_\ell}}$, one has the following estimate independent of $\alpha$
\begin{equation}\label{moment}
     \EE_{\eta_\mu^\alpha}[|x|^4]\leq C(d,\sigma,r,x^*,K)\,.
\end{equation}

Moreover, following similar arguments in \eqref{distx*} one can deduce
\begin{align*}
&\EE_{\eta^\alpha_{\mu}}[|x-x^*|^2]\\
    \leq &\varepsilon^2+ \bigg(\eta_{\mu}^\alpha(\{x:|x-x^*|\geq \varepsilon\}) \bigg)^{\frac{1}{2}} \left(\int_{|x-x^*|\geq \varepsilon}|x-x^*|^4\eta_{\mu}^\alpha(\dd x) \right)^{\frac{1}{2}}\\
    \leq&\varepsilon^2 +\bigg(\eta_{\mu}^\alpha(\{x:|x-x^*|\geq \varepsilon\}) \bigg)^{\frac{1}{2}} \left(C\EE_{\eta_\mu^\alpha}\left[|x|^4\right]+C|x^*|^4\right)^{1/2}\\
    \leq & \varepsilon^2 +C\bigg(\eta_{\mu}^\alpha(\{x:|x-x^*|\geq \varepsilon\}) \bigg)^{\frac{1}{2}} \to \varepsilon^2 \mbox{ as }\alpha\to\infty\,,
    \end{align*}
where $C>0$ is independent of $\alpha$ according to \eqref{moment}. Thus, it follows from \eqref{57} that 
\begin{equation}\label{contraction}
     \|\nabla \mathcal{T}^\alpha(\mu)\|\to 0\quad \mbox{ as }\alpha\to\infty,
\end{equation}
so one can take $\alpha$ sufficiently large such that $\|\nabla \mathcal{T}^\alpha(\mu)\|<1$. This implies that $\mathcal{T}^\alpha$ is a contraction.

\textbf{Step 3:}  Collecting the results from above and applying Banach fixed-point theorem, we know that for $\alpha$ sufficiently large $\mathcal{T}^\alpha$ has a unique fixed point $\boldsymbol\mu_\infty^\alpha\in B_r(x^*)$. Thus we have found an invariant measure $\boldsymbol\rho_\infty^\alpha$ for \eqref{CBO} satisfying
$$\boldsymbol\rho_\infty^\alpha=\mc{N}\left(\boldsymbol\mu_\infty^\alpha,\frac{\delta^2}{2\lambda}I_d\right)=\mc{N}\left(\mc{T}^\alpha(\boldsymbol\mu_\infty^\alpha),\frac{\delta^2}{2\lambda}I_d\right)=\mc{N}\left(\Xalpha(\boldsymbol\rho_\infty^\alpha),\frac{\delta^2}{2\lambda}I_d\right)\,.$$
 Moreover, we can apply Lemma \ref{lem:laplace} and obtain
 \begin{equation}
     W_2\left(\boldsymbol\rho_\infty^\alpha, \mc{N}\left(x^*,\frac{\delta^2}{2\lambda}I_d\right)\right)=|\Xalpha(\boldsymbol\rho_\infty^\alpha)- x^*|\to 0\quad \mbox{ as }\alpha\to\infty \,.
 \end{equation}
We note here that to prove $\Xalpha(\boldsymbol\rho_\infty^\alpha)\to x^*$, one shall let $\nu=\boldsymbol\rho_\infty^\alpha=\mc{N}(\boldsymbol\mu_\infty^\alpha,\frac{\delta^2}{2\lambda}I_d)$ in Lemma \ref{lem:laplace}. Since $\boldsymbol\mu_\infty^\alpha\in B_r(x^*)$, the conditions 
\begin{align*}
\boldsymbol\rho_\infty^\alpha(B_s(x^*))=\int_{B_s(x^*)}\frac{1}{(2\pi\sigma^2)^{d/2}}\exp\left(-\frac{|x-\boldsymbol\mu_\infty^\alpha|^2}{2\sigma^2}\right) \diff x\geq \frac{|B_s(x^*)|}{(2\pi\sigma^2)^{d/2}} \exp\left(-\frac{|r+s|^2}{2\sigma^2}\right)=C_s
\end{align*}
and
\begin{equation*}
    \EE_{\boldsymbol\rho_\infty^\alpha}\left[|x|^2\right]\leq (|x^*|+r)^2+\frac{d\delta^2}{2\lambda}=c_2
\end{equation*}
in  Lemma \ref{lem:laplace} hold for $C_s,c_2>0$ independent of $\alpha$.

\textbf{Uniqueness}:
First, let us prove that any invariant measure of the dynamics \eqref{CBO} is a Gaussian with variance $\frac{\delta^2}{2\lambda}I_d$. Indeed, let $\varrho$ be an invariant measure, let $X_t$ be a solution to \cref{eq:flip_scheme_2} with $\law(X_0) = \varrho$ and $\varrho^{\alpha,s,\Delta t}_t \coloneqq \law(X_t)$; then $\varrho^{\alpha,s,\Delta t}_t = \varrho$ for any $t\geq 0$, and thus $X^\alpha(\varrho^{\alpha,s,\Delta t}_t) = X^\alpha(\varrho)$ is constant in time. Thus, \cref{eq:flip_scheme_2} reduces to the Ornstein-Uhlenbeck process with drift towards $X^\alpha(\varrho)$. It is well-known that the unique invariant measure of this process is the Gaussian distribution $\mc{N}\left(X^\alpha(\varrho),\frac{\delta^2}{2\lambda}I_d\right)$, which implies that $\varrho = \mc{N}\left(X^\alpha(\varrho),\frac{\delta^2}{2\lambda}I_d\right)$. \\
From the above, we also see that necessarily, $T_\alpha(\mu) = \mu$ where $\mu$ is the mean of $\varrho$. Since we have already proved that for $\alpha$ sufficiently large, $\mc{T}^\alpha$ has a unique fixed point in $B_r(x^*)$, we conclude that the invariant measure $\boldsymbol\rho_\infty^\alpha$ found above is the unique invariant measure whose mean is contained in $B_r(x^*)$.
\end{proof}

We have now established the existence of the unique invariant measure $\boldsymbol\rho_\infty^\alpha=\mc{N}(\boldsymbol\mu_\infty^\alpha,\frac{\delta^2}{2\lambda}I_d)$ with $\boldsymbol\mu_\infty^\alpha=\mc{T}^\alpha(\boldsymbol\mu_\infty^\alpha)$. Next we prove that the process $X_t$ governed by the CF dynamics \eqref{CBO} with the initial data $\varrho_0=\mc{N}(\mu_0,\frac{\delta^2}{2\lambda}I_d)$ indeed converges to the invariant measure $\boldsymbol\rho_\infty^\alpha$ as $i\to \infty $.
\begin{lemma}\label{lem: stay in ball}
     Let $\varrho_{t_i}^{\alpha,s,\Delta t}=\mc{N}(\mu_i^{\alpha,s,\Delta t}, \Sigma_i^\alpha)$ be the unique solution to the dynamics \eqref{CBO} with the initial data $X_0\sim \varrho_0=\mc{N}(\mu_0,\frac{\delta^2}{2\lambda}I_d)$ and any $\Delta t>0$. Denote by  $r:=|\mu_0-x^*|$, then for $\alpha$ sufficiently large it holds that
     \begin{equation}
         |\mu_i^{\alpha,s,\Delta t}-x^*|\leq r\quad \mbox{ and }\quad \Sigma_i^{\alpha,s,\Delta t}= \frac{\delta^2}{2\lambda}I_d,\quad \forall i\geq 1.
     \end{equation}
\end{lemma}
\begin{proof}
    According to \eqref{iterative'}, it is clear that $\Sigma_i^{\alpha,s,\Delta t} =\frac{\delta^2}{2\lambda}I_d$ for every $i$. Let us compute now
    \begin{align}
        |\mu_1^{\alpha,s,\Delta t} -x^*|&\leq (1-\omega_{\Delta t})|\Xalpha(\rho_0)-x^*|+\omega_{\Delta t}|\mu_0-x^*|\notag\\
        &=(1-\omega_{\Delta t})|\mc{T}^\alpha(\mu_0)-x^*|+\omega_{\Delta t}|\mu_0-x^*|\leq r\,,
    \end{align}
    where we have used the definition of $\mTa$ and the fact that $\mTa$ maps from $B_r(x^*)$ to $B_r(x^*)$ for $\alpha$ sufficiently large. This process may be iterated further to yield
    \begin{equation}
    |\mu_{i+1}^{\alpha,s,\Delta t} -x^*|\leq (1-\omega_{\Delta t})|\mc{T}^\alpha(\mu_i^{\alpha,s,\Delta t})-x^*|+\omega_{\Delta t}|\mu_i^{\alpha,s,\Delta t}-x^*|\leq r\,.
    \end{equation}
\end{proof}
\begin{theorem}
    Let $\varrho_{t_i}^{\alpha,s,\Delta t}=\mc{N}(\mu_i^{\alpha,s,\Delta t}, \Sigma_i^\alpha)$ be the unique solution to the dynamics \eqref{CBO} with the initial data $\rho_0=\mc{N}(\mu_0,\frac{\delta^2}{2\lambda}I_d)$ and any $\Delta t>0$. Let $r=|\mu_0-x^*|$ and assume  $\boldsymbol\rho_\infty^\alpha=\mc{N}(\boldsymbol\mu_\infty^\alpha,\frac{\delta^2}{2\lambda}I_d)$ is the invariant measure found in Theorem \ref{thm:invariant} with $\boldsymbol\mu_\infty^\alpha\in B_r(x^*)$. Then for $\alpha$ sufficiently large it holds that
\begin{equation}
   W_2(\varrho_{t_i}^{\alpha,s,\Delta t},\boldsymbol\rho_\infty^\alpha)\leq \left(\frac{1}{2}(1+\omega_{\Delta t})\right)^i |\mu_0-\boldsymbol\mu_\infty^\alpha|\to 0\quad \mbox{ as }\quad i\to \infty 
\end{equation}
where $0<\omega_{\Delta t}=e^{-s\lambda \Delta t}<1$.
\end{theorem}
\begin{proof}
 Since $s$ and $\Delta t$ are fixed in this proof, we will use the notation $\varrho_{t_i}^{\alpha} = \varrho_{t_i}^{\alpha,s,\Delta t}$ and $\mu_{i}^{\alpha}=\mu_i^{\alpha,s,\Delta t}$.
   We compute the 2-Wasserstein distance between two Gaussian measures
   \begin{align}
&W_2^2(\varrho_{t_i}^{\alpha},\boldsymbol\rho_\infty^\alpha)=W_2^2\left(\mc{N}\left(\mu_i^\alpha,\frac{\delta^2}{2\lambda}I_d\right),\mc{N}\left(\boldsymbol\mu_\infty^\alpha,\frac{\delta^2}{2\lambda}I_d\right)\right)=|\mu_i^\alpha-\boldsymbol\mu_\infty^\alpha|^2+\left\|\frac{\delta^2}{2\lambda}I_d-\frac{\delta^2}{2\lambda}I_d\right\|^2 \notag \\
&=|\mu_i^\alpha-\boldsymbol\mu_\infty^\alpha|^2\,.
   \end{align}
Furthermore, it follows from \eqref{iterative'} that
\begin{align}
    \mu_i^\alpha-\boldsymbol\mu_\infty^\alpha&= (1-\omega_{\Delta t})(\Xalpha(\varrho_{t_{i-1}}^{\alpha})-\boldsymbol\mu_\infty^\alpha)+\omega_{\Delta t}(\mu_{i-1}^\alpha-\boldsymbol\mu_\infty^\alpha)\notag\\
    &=(1-\omega_{\Delta t})(\mTa(\mu_{i-1}^\alpha)-\mTa(\boldsymbol\mu_\infty^\alpha))+\omega_{\Delta t}(\mu_{i-1}^\alpha-\boldsymbol\mu_\infty^\alpha)\,,
\end{align}
where we have used the definition of $\mTa$ and the fact that $\mTa(\boldsymbol\mu_\infty^\alpha)=\boldsymbol\mu_\infty^\alpha$. Then it implies that
\begin{align}
 |\mu_i^\alpha-\boldsymbol\mu_\infty^\alpha|\leq  \left((1-\omega_{\Delta t})L_\alpha+\omega_{\Delta t}\right)|\mu_{i-1}^\alpha-\boldsymbol\mu_\infty^\alpha| 
\end{align}
where we have used that $\mu_{i-1}^\alpha\in B_r(x^*)$ (see in Lemma \ref{lem: stay in ball}), and
$\mTa$ is a contraction mapping from $B_r(x^*)\to B_r(x^*)$ (see in \cref{thm:invariant}) satisfying
\begin{equation}
    |\mTa(\mu_{i-1}^\alpha)-\mTa(\boldsymbol\mu_\infty^\alpha)|\leq L_\alpha |\mu_{i-1}^\alpha-\boldsymbol\mu_\infty^\alpha| \mbox{ with } L_\alpha\to 0 \mbox{ as }\alpha \to\infty.
\end{equation}
Let us define $q_\alpha:=(1-\omega_{\Delta t})L_\alpha+\omega_{\Delta t}\to \omega_{\Delta t}<1$ as $\alpha \to\infty$. This means that for any fixed $\Delta t>0$, there exists some $\alpha_0>0$ such that for all $\alpha\geq \alpha_0$, we have
\begin{equation}
    L_\alpha<\frac{1}{2} \mbox{ and } q_\alpha<\frac{1}{2}(1+\omega_{\Delta t})<1 
\end{equation}
which leads to
\begin{equation}
     |\mu_i^\alpha-\boldsymbol\mu_\infty^\alpha|\leq \left(\frac{1}{2}(1+\omega_{\Delta t})\right)^i|\mu_0-\boldsymbol\mu_\infty^\alpha|\,.
\end{equation}
\end{proof}
To prove the convergence of the Consensus Freezing scheme for an initial measure not necessarily Gaussian, we first require global-in-time bounds for the {higher} moments. We postpone the proof of the following Theorem to the end of \cref{sec:global-in-time-bounds}.
\begin{theorem}\label{thm:global_bound}
   {Let $p \in \NN_{\geq 2}$ and $\rho_0,\varrho_0 \in P_p(\RR^d)$ be arbitrary initial measures.} 
    Let $\rho_t^\alpha, \ \varrho^{\alpha,s,\Delta t}_t$ be the solution to the $\delta$-CBO scheme \eqref{eq:cbo_delta_mf} and Consensus Freezing scheme \eqref{eq:flip_scheme_2}, respectively, with $\rho_0^\alpha = \rho_0$ and $ \varrho^{\alpha,s,\Delta t}_0 = \varrho_0$. Then, there exist $\alpha_0 > 0$ and $S > 0$, such
    that for all $\alpha \geq \alpha_0$, the solutions satisfy
    {
    \begin{equation}\label{eq:second_moment_bound}
        \EE_{\rho_t^\alpha} \big[ |x|^q \big] \leq S, \quad \EE_{\varrho_t^{\alpha,s,\Delta t}} \big[ |x|^q \big] \leq S, \quad \forall t \geq 0, \quad \forall q \in \NN \text{ such that } p \geq q.
    \end{equation}
    }
    and
    \begin{equation}\label{eq:gamma_t_alpha_bound}
        |\Xalpha(\rho_t^\alpha)| \leq S, \quad |\Xalpha(\varrho_t^{\alpha,s,\Delta t})| \leq S, \quad \forall t \geq 0.
    \end{equation}
\end{theorem}
\begin{theorem}\label{thm:fsconv}
    Let $\varrho_{t}^{\alpha,s,\Delta t}$ be the unique solution to the dynamics \eqref{CBO} with the initial data $X_0\sim \varrho_0 {\in \mathcal{P}_4(\RR^d)}$ (not necessarily a Gaussian) and any $\Delta t>0$. {Let $r=2S$}, {where $S$ is a constant fulfilling \cref{eq:second_moment_bound,eq:gamma_t_alpha_bound} such that $x^* \in B_S(0)$ and $\EE_{\rho_0}[x] \in B_S(0)$}, and assume  $\boldsymbol\rho_\infty^\alpha=\mc{N}(\boldsymbol\mu_\infty^\alpha,\frac{\delta^2}{2\lambda}I_d)$ is the invariant measure found in Theorem \ref{thm:invariant} {with $\boldsymbol\mu_\infty^\alpha\in B_r(x^*)$}. Then, for {any $  e^{-s \lambda \Delta t}<\gamma<1$} and for $\alpha>0$ sufficiently large it holds that
    \begin{equation}
       W_2(\varrho_{t}^{\alpha,s,\Delta t},\boldsymbol\rho_\infty^\alpha)\leq C_\alpha \gamma^i\ \quad \mbox{ for } t \in [t_i, t_{i+1}), \ i \in \NN_{0} \,,
    \end{equation}
    for some $C_\alpha>0$. Furthermore, for any $s^\star > 0$, we may choose $C_\alpha$ and $\gamma$ uniformly over all $s\in [s^\star, +\infty)$.
\end{theorem}

    \begin{remark}
        While the above Theorem has quantitative bounds on the convergence rate---in the sense that it can be made arbitrarily close to $e^{-s \lambda \Delta t}$---we do not try to find quantitative bounds on the constant $C_\alpha$ or the required magnitude of $\alpha$. Doing so would require stronger assumptions on $f$ than those we use in this chapter---for instance that there is  growth of $f$ ensured around the global minimum,  as in \cref{assumptions}---and would need the following steps:
        \begin{enumerate}
            \item Use those stronger assumptions to derive a suitable quantitative version of the Laplace principle \cref{lem:laplace}.
            \item Obtaining a quantitative bound on the constant $L_X$ in \cref{lem: xalpha bound}.
            \item Use the quantitative Laplace principle and bound on $L_X$ to chase through the proof of \cref{thm:invariant} to find a quantitative bound on $L_\alpha$.
            \item Chase this quantitative bound on $L_\alpha$ through the proof of \cref{thm:global_bound} to find a quantitative bound on $S$ in \cref{thm:global_bound}.
            \item Chase these quantitative bounds on $L_\alpha$ and $S$ through the proof of \cref{thm:fsconv} to find a quantitative bound on $C_\alpha$ and the required size of $\alpha$.
        \end{enumerate}
        For the sake of {keeping the presentation of this paper as non-technical as possible}, we do not pursue this here.
    \end{remark}
Before we prove \cref{thm:fsconv}, we recall a standard comparison argument for differential inqualities.
\begin{lemma}[Comparison principle, simplified]\label{lem:comparison_principle}
    Let $O \subseteq \RR^d$ be open.
    Let $U(t, v)$ be continuous in $t$ and $U|_O$ be locally Lipschitz continuous in $v$, uniformly in $t$ on bounded sets. Furthermore, let $y: [0, T) \to O$ be a solution to the ODE
    \begin{equation}\label{eq:ode_comparison}
        \frac{\diff}{\diff t} y(t) = U(t, y(t)), \quad y(0) = y_0 \in O,
    \end{equation}
    with $y_0 \in O$.
    If $z: [0, T) \to \RR^d$ is a differentiable function satisfying the differential inequality
    \begin{equation}
        \frac{\diff}{\diff t} z(t) \leq U(t, z(t)), \quad z(0) \leq y_0,
    \end{equation}
    then it holds that $z(t) \leq y(t)$ for all $t \in [0, T)$.
\end{lemma}
\begin{proof}
    As $U|_O$ is locally Lipschitz continuous in $v$, uniformly in $t$ on bounded sets, the function $y(t)$ is the unique solution to the ODE~\eqref{eq:ode_comparison}.
    In particular, it is a maximal solution. The claim then follows from \cite[Theorem 4.1]{hartman1964ordinary}.
\end{proof}
\begin{proof}[Proof of \cref{thm:fsconv}]
    We set $\gamma_t^\alpha \coloneqq \Xalpha(\varrho_{t_i}^{\alpha,s,\Delta t})$ for $t \in [t_i, t_{i+1})$, $c_0 \coloneqq \EE_{\varrho_0}[x]$ and consider the process $X_t$ governed by the CF dynamics \cref{eq:flip_scheme_2} with the initial data $X_0 \sim \varrho_0$, and the process $Z_t$ governed by the SDE
    \begin{equation}\label{eq:reference_process_sde}
        dZ_t = - s \cdot \lambda (Z_t - \gamma_t^\alpha) \diff t + \sqrt{s} \cdot \delta \diff B_t, \qquad Z_0 \sim \mc{N}(c_0, \frac{\delta^2}{2\lambda} I_d),
    \end{equation}
    driven by the same Brownian motion $B_t$ as $X_t$. By considering the difference between \cref{eq:cbo_delta_mf,eq:reference_process_sde}, we obtain
    \begin{equation}\label{eq:exponential_decay_diff}
        X_t - Z_t = e^{-s \lambda t} (X_0 - Z_0).
    \end{equation}
    Next, we define $c_\gamma(t) \coloneqq \EE(Z_t)$ and we will show that
    % Note that $Z_t \sim \mc{N}(c_\gamma(t), \frac{\delta}{2\lambda} I_d) \eqqcolon \eta_t$, where $c_\gamma(t) = \EE(Z_t)$ fulfills the differential equation
    \begin{equation}\label{eq:differential_equation_c_gamma}
        Z_t \sim \mc{N}(c_\gamma(t), \frac{\delta^2}{2\lambda} I_d) \eqqcolon \zeta_t, \quad\frac{\diff}{\diff t} c_\gamma(t) = -s \lambda (c_\gamma(t) - \gamma^\alpha_t), \quad c_\gamma(0) = c_0.
    \end{equation}
    To see this, it suffices to consider another reference process $Z'_t$ fulfilling the SDE
    \begin{equation}\label{eq:reference_process_sde_prime}
        dZ'_t = -s \cdot \lambda Z'_t \diff t + \sqrt{s}\cdot\delta \diff B_t, \qquad Z'_0 = Z_0 - c_0, 
    \end{equation}
    driven by the same Brownian motion $B_t$. On the one hand, we know that $Z'_t \sim \mc{N}(0, \frac{\delta^2}{2\lambda} I_d)$ for all $t \geq 0$; on the other hand, we get---by considering the difference \cref{eq:reference_process_sde,eq:reference_process_sde_prime}---that
    \begin{equation}\label{eq:difference_to_ref_process_differential_equation}
        d(Z_t - Z'_t) = - s \lambda (Z_t - Z'_t - \gamma^\alpha_t) \diff t.
    \end{equation}
    \cref{eq:difference_to_ref_process_differential_equation} equation together with the initial condition in~\eqref{eq:reference_process_sde_prime} implies that $Z_t - Z'_t$ is deterministic and thus $Z_t - Z'_t = \EE[Z_t] - \EE[Z'_t] = c_\gamma(t)$, which together with $Z'_t \sim \mc{N}(0, \frac{\delta^2}{2\lambda} I_d)$ yields the first part of \cref{eq:differential_equation_c_gamma}. The differential equation in \cref{eq:differential_equation_c_gamma} follows by taking expectation in \cref{eq:difference_to_ref_process_differential_equation}. \\
   Next, we use \cref{eq:differential_equation_c_gamma} together with \cref{thm:global_bound} to obtain global-in-time bounds on $|c_\gamma(t) - x^*|$ and $\EE_{\zeta_t}[|x|^4]$. We start by invoking \cref{thm:global_bound}, which yields $S > 0$ such that for all $\alpha$ sufficiently large and all $t \geq 0$, it holds that
    \begin{equation}\label{eq:global_in_time_bounds_in_proof}
        \EE_{\varrho_{t}^{\alpha,s,\Delta t}}[|x|^4] \leq S, \quad |\gamma_t^\alpha| \leq S;
    \end{equation}
    we may further choose $S$ large enough such that $x^* \in B_S(0)$ and $c_\gamma(0) \in B_S(0)$.
    Using \cref{eq:differential_equation_c_gamma}, we obtain the differential bounds
    \begin{equation}\label{eq:differential_c_gamma_bounds}
        \begin{split}
        \frac{\diff}{\diff t} |c_\gamma(t) - x^*|^2 &= 2(c_\gamma(t) - x^*) \frac{\diff}{\diff t} c_\gamma(t) \\
        &= 2 (c_\gamma(t) - x^*) (-s \lambda (c_\gamma(t) - x^*) + s \lambda (\gamma_t^\alpha - x^*)) \\
        &\leq 2 s \lambda \Big( - |c_\gamma(t) - x^*|^2 + |c_\gamma(t) - x^*|\cdot |\gamma_t^\alpha - x^*| \Big) \\
        &\leq 2 s \lambda \Big( - |c_\gamma(t) - x^*|^2 + (|c_\gamma(t) - x^*|) 2S \Big).
        \end{split}
    \end{equation}
    The right hand side is locally Lipschitz on $(0, \infty)$ and hence we can apply \cref{lem:comparison_principle} with a solution $N: [0, \infty] \to \RR$ of the corresponding ODE such that $N(0) \geq |c_\gamma(0) - x^*|^2$ and $N(t) > 0$ for all $t \geq 0$. Such a solution is given by
    \begin{equation*}
        N(t) = ( 2S )^2, \quad \text{which fulfills } \begin{cases} N'(t) = 2 s \lambda \Big( - N(t) + \sqrt{N(t)} \cdot 2S \Big), \\N(0) = (2S)^2 \geq (|c_\gamma(0) - x^*|)^2.
        \end{cases}
    \end{equation*}
    From the above, the comparison principle yields that $|c_\gamma(t) - x^*| \leq 2S \eqqcolon r$ for all $t \geq 0$. Repeating the same argument with $0$ in place of $x^*$, we also obtain $|c_\gamma(t)| \leq 2S$, and thus, we get the fourth-moment-bounds
    \begin{equation*}
        \EE_{\zeta_t}[|x|^4] \leq \EE_{\mc{N}(0, \frac{\delta^2}{2\lambda} I_d)}[(|x| + |c_\gamma(t)|)^4] \leq \sum_{j=0}^4 \binom{4}{j} (2S)^{4-j} \EE_{\mc{N}(0, \frac{\delta^2}{2\lambda} I_d)}[|x|^j] \eqqcolon C(S, \delta, \lambda), \quad \forall t \geq 0.
    \end{equation*}
    By \cref{thm:invariant}, for $\alpha$ sufficiently large, $\mathcal{T}^\alpha$ is a contraction mapping from $B_r(x^*)$ to $B_r(x^*)$ with Lipschitz constant $L_\alpha < 1$; furthermore, by \cref{lem: xalpha bound} and using the joint fourth moments bound $\max\{ C(S, \delta, \lambda), S \}$ on both $\varrho_t^{\alpha,s,\Delta t}$ and $\zeta_t$, we get a Lipschitz constant $L_X$ such that
    \begin{equation}\label{eq:stability_in_proof}
        |\Xalpha(\varrho_{t}^{\alpha,s,\Delta t}) - \Xalpha(\zeta_t)| \leq L_X \cdot W_2(\varrho_{t}^{\alpha,s,\Delta t}, \zeta_t), \quad \forall t \geq 0.
    \end{equation}
    Using the Lipschitz continuity of $\mathcal{T}^\alpha$, \cref{eq:stability_in_proof} and the fact that \cref{eq:exponential_decay_diff} yields
    \begin{equation}\label{73}
        W_2^2(\varrho_{t}^{\alpha,s,\Delta t}, \zeta_t) {\leq} e^{-2s\lambda t} W_2^2(\varrho_0, \zeta_0),
    \end{equation}
    we calculate
    \begin{equation}\label{eq:differential_wasserstein}
        \begin{split}
        |\Xalpha(\varrho_{t}^{\alpha,s,\Delta t}) - \boldsymbol\mu_\infty^\alpha| &\leq |\Xalpha(\varrho_{t}^{\alpha,s,\Delta t}) - \Xalpha(\zeta_t)| + |\Xalpha(\zeta_t) - \boldsymbol\mu_\infty^\alpha| \\
        &\leq L_X W_2(\varrho_{t}^{\alpha,s,\Delta t}, \zeta_t) + L_\alpha |c_\gamma(t) - \boldsymbol\mu_\infty^\alpha| \\
        &\leq L_X e^{-s\lambda t} W_2(\varrho_0, \zeta_0) + L_\alpha |c_\gamma(t) - \boldsymbol\mu_\infty^\alpha|.
        \end{split}
    \end{equation}
    We will now combine \cref{eq:differential_wasserstein} with another majorization argument to obtain the desired result. \\
    First, repeating the calculation in \cref{eq:differential_c_gamma_bounds} but with $\boldsymbol\mu_\infty^\alpha$ in place of $x^*$, we obtain, for all $i \in \NN$, $t \in [t_i, t_{i+1})$ and setting $G_i \coloneqq |\gamma_{t_i}^\alpha - \boldsymbol\mu_\infty^\alpha|$,
    \begin{equation}\label{eq:differential_c_gamma'}
        \begin{split}
            \frac{\diff}{\diff t} |c_\gamma(t) - \boldsymbol\mu_\infty^\alpha|^2 \leq 2 s \lambda \Big( - |c_\gamma(t) - \boldsymbol\mu_\infty^\alpha|^2  + |c_\gamma(t) - \boldsymbol\mu_\infty^\alpha| \cdot G_i \Big).
        \end{split}
    \end{equation}
    As before, the right hand side is locally Lipschitz on $(0, \infty)$ and thus, we can again apply \cref{lem:comparison_principle} with a solution of the corresponding ODE with strictly positive values and higher initial value. Setting $H_i \coloneqq |c_\gamma(t_i) - \boldsymbol\mu_\infty^\alpha|$, for any $\varepsilon > 0$, such a solution is given by
    \begin{equation*}
        \bar{N}_{\varepsilon, i}(\tau) = \bigg((H_i + \varepsilon - G_i)e^{-s\lambda \tau} + G_i\bigg)^2, \quad \text{which fulfills } \begin{cases} \bar{N}_{\varepsilon, i}'(\tau) = 2 s \lambda \Big( - \bar{N}_{\varepsilon, i}(\tau) + \sqrt{\bar{N}_{\varepsilon, i}(\tau)} \cdot G_i \Big), \\ \bar{N}_{\varepsilon, i}(0) = (H_i + \varepsilon)^2;
        \end{cases}
    \end{equation*}
    the added value $\varepsilon > 0$ ensures that $\bar{N}_{\varepsilon, i}(\tau) > 0$ for all $\tau \in [0, \Delta t)$. Letting $\varepsilon \to 0$, setting $D_\alpha \coloneqq L_X \cdot W_2(\varrho_0, \zeta_0)$---note that $L_X$ depends on $\alpha$---, and using the bounds in \cref{eq:differential_wasserstein}, we obtain, for $\tau \in [0, \Delta t)$, and setting $t = t_i + \tau$,
    \begin{equation}\label{eq:79}
        \begin{split}
        |c_\gamma(t) - \boldsymbol\mu_\infty^\alpha| &\leq (H_i - G_i)e^{-s\lambda \tau} + G_i \leq H_i e^{-s\lambda \tau} + (L_\alpha H_i + D_\alpha e^{-s\lambda t_i}) (1 - e^{-s\lambda \tau}) \\
        &= (e^{-s\lambda \tau} - L_\alpha e^{-s \lambda \tau} + L_\alpha) H_i + D_\alpha e^{-s\lambda t_i} (1 - e^{-s\lambda \tau}).
        \end{split}
    \end{equation}
    Letting $\tau \to \Delta t$ and setting $\omega_s \coloneqq e^{-s\lambda\Delta t}$, this yields
    \begin{equation*}
        \begin{split}
            H_{i+1} \leq ( \omega_s - L_\alpha \omega_s + L_\alpha ) H_i + D_\alpha \omega_s^i (1 - \omega_s).
        \end{split}
    \end{equation*}
    Iterating the above formula, we obtain
    \begin{equation}\label{eq:80}
        \begin{split}
            H_{i} &\leq (\omega_s - L_\alpha \omega_s + L_\alpha)^i H_0 + (1 - \omega_s) D_\alpha \sum_{k=0}^{i-1} (\omega_s - L_\alpha \omega_s + L_\alpha)^k \omega_s^{i-1-k} \\
            &\leq (\omega_s - L_\alpha \omega_s + L_\alpha)^i H_0 + (1 - \omega_s) D_\alpha \cdot i \cdot (\omega_s - L_\alpha \omega_s + L_\alpha)^{i-1},
        \end{split}
    \end{equation}
    where we used that $\omega_s \leq \omega_s - L_\alpha \omega_s + L_\alpha$, as $\omega_s < 1$.
    Using \cref{eq:79} again, we can further estimate---for $\alpha$ large enough such that $L_\alpha \leq 1$---, for $t \in [t_i, t_{i+1})$,
    \begin{equation}\label{eq:81}
        \begin{split}
            |c_\gamma(t) - \boldsymbol\mu_\infty^\alpha| &\leq H_i + D_\alpha e^{-s\lambda t_i} (1 - \omega_s) \leq  H_i + (1 - \omega_s) D_\alpha \cdot \omega_s^{i-1} \\
            &\leq (\omega_s - L_\alpha \omega_s + L_\alpha)^i H_0 + {(1 - \omega_s) D_\alpha \cdot (i+1) \cdot (\omega_s - L_\alpha \omega_s + L_\alpha)^{i-1}}\,.
        \end{split}
    \end{equation}
    As $\omega_s - L_\alpha \omega_s + L_\alpha \to \omega_s$ as $L_\alpha \to 0$ and thus as $\alpha \to \infty$, we can find, for every $\gamma \in ( \omega_s, 1)$, some $\alpha_0$ such that $\omega_s - L_\alpha \omega_s + L_\alpha < \gamma$ for any $\alpha \geq \alpha_0$. As furthermore $H_0$  and $D_\alpha$ depend only on the initial data and $\alpha$ and due to the subexponential growth of $(i+1)$, for any such $\alpha \geq \alpha_0$, we can find $\bar{C}_\alpha > 0$ such that
    \begin{equation*}
        W_2(\zeta_{t}, \boldsymbol\rho_\infty^\alpha) = W_2\Big(\mc{N}(c_\gamma(t), \frac{\delta^2}{2\lambda} I_d), \mc{N}(\boldsymbol\mu_\infty^\alpha, \frac{\delta^2}{2\lambda} I_d)\Big) = |c_\gamma(t) - \boldsymbol\mu_\infty^\alpha| = H_i \leq \bar{C}_\alpha \gamma^i, \quad \forall t \in [t_i, t_{i+1}).
    \end{equation*}
    To find such a $\overline{C}_\alpha$, we may first choose $C'$ such that $(i+1) \cdot (\frac{\omega_s - L_\alpha \omega_s + L_\alpha}{\gamma})^{i} \leq C'$ for all $i \geq 0$---possible due to the subexponential growth---and then set $\overline{C_\alpha} \coloneqq \frac{(1-\omega_s)D_\alpha C'}{\omega_s - L_\alpha \omega_s + L_\alpha} + H_0$.
    Using \cref{73} again, we finally obtain
    \begin{equation*}
        W_2(\varrho_{t}^{\alpha,s,\Delta t}, \boldsymbol\rho_\infty^\alpha) \leq W_2(\varrho_{t}^{\alpha,s,\Delta t}, \zeta_{t}) + W_2(\zeta_{t}, \boldsymbol\rho_\infty^\alpha) \leq e^{-s\lambda t} W_2(\varrho_0, \zeta_0) + \bar{C}_\alpha \gamma^i \leq C_\alpha \gamma^i, \quad \forall t \in [t_i, t_{i+1}),
    \end{equation*}
    where we set $C_\alpha \coloneqq W_2(\varrho_0, \zeta_0) + \bar{C}_\alpha$ and recall that $\gamma \in (\omega_s, 1 )$.
    Lastly, note that the right hand side of \cref{eq:80,eq:81} depends on $s$ only via $\omega_s = e^{-s \lambda \Delta t}$; thus, for any $s^\star > 0$, we may choose $C_\alpha$ and $\gamma$ uniformly over all $s \in [s^\star, +\infty)$.
\end{proof}
{
\begin{remark}
    At a first glance, it may seem that results of the form of \cref{thm:fsconv}, modeled on results as in \cite{huang2025faithful}---exponential convergence to an invariant measure $\boldsymbol\rho_\infty^\alpha$ with infinite-time horizon---are weaker than results as in \cref{thm:exponential_decay}, modelled on results as in \cite{B1-fornasier2021global}---exponential convergence to the optimal measure $\boldsymbol\rho_\infty = \mc{N}(x^*, \frac{\delta^2}{2\lambda} I_d)$ with finite-time horizon. However, the convergence $\boldsymbol\rho_\infty^\alpha \to \boldsymbol\rho_\infty$ as $\alpha \to \infty$ (see \cref{thm:invariant}) ensures that, aside from the magnitude of the constants involved in the convergence, results of the latter type can be deduced from results of the former type. We provide an example of such a deduction in the following statement.
\end{remark}
\begin{theorem}
    Let $\varrho_0$ be an arbitrary initial measure, let $s > 0, \gamma > 0, \sigma > 0$ and $\Delta t > 0$ be fixed and let furthermore $e^{-s \lambda \Delta t}<\gamma<1$ and $\varepsilon > 0$. Then, there exists $\alpha_0 > 0$ such that for all $\alpha \geq \alpha_0$, given the solution $\rho_{t_i}^\alpha$ to the dynamics \eqref{CBO} with initial data $\varrho_0$, there exists $C_\alpha > 0$ such that
    \begin{equation}
        W_2(\varrho_{t_i}^{\alpha,s,\Delta t}, \boldsymbol\rho_\infty) \leq C_\alpha \gamma^i \quad \text{for all } i \text{ such that } t_i \leq T^\star,
    \end{equation}
    where $\boldsymbol\rho_\infty = \mc{N}(x^*, \frac{\delta^2}{2\lambda} I_d)$ and $T^\star \coloneqq  \Delta t \cdot \left\lceil\frac{\log(\varepsilon/C_\alpha)}{\log(\gamma) \Delta t}\right\rceil$.
    In particular, we have that
    \begin{equation}
        W_2(\varrho_{T^*}^{\alpha,s,\Delta t}, \boldsymbol\rho_\infty) \leq \varepsilon.
    \end{equation}
\end{theorem}
\begin{proof}
First, we choose $\alpha_0 > 0$ such that for all $\alpha \geq \alpha_0$, it holds that $W_2(\boldsymbol\rho_\infty^\alpha, \boldsymbol\rho_\infty) \leq \varepsilon/2$ (possible due to $\boldsymbol\rho_\infty^\alpha \to \boldsymbol\rho_\infty$ as $\alpha \to \infty$, see \cref{thm:invariant}). Then, we choose $\alpha_0$ larger if necessary such that for all $\alpha \geq \alpha_0$, it holds that $W_2(\varrho_{t_i}^{\alpha,s,\Delta t}, \boldsymbol\rho_\infty^\alpha) \leq C_\alpha \gamma^i$ (possible due to \cref{thm:fsconv}). Setting $T^\star \coloneqq \Delta t \cdot \left\lceil\frac{\log(\varepsilon/(2C_\alpha))}{\log(\gamma)}\right\rceil$, we obtain for all $i$ such that $t_i \leq T^\star$ that
\begin{equation}
    \begin{split}
        W_2(\varrho_{t_i}^{\alpha,s,\Delta t}, \boldsymbol\rho_\infty) &\leq W_2(\varrho_{t_i}^{\alpha,s,\Delta t}, \boldsymbol\rho_\infty^\alpha) + W_2(\boldsymbol\rho_\infty^\alpha, \boldsymbol\rho_\infty) \leq C_\alpha \gamma^i + \varepsilon/2 \\&\leq C_\alpha \gamma^i + C_\alpha \gamma^{T^\star} \leq C_\alpha \gamma^i + C_\alpha \gamma^i = 2 C_\alpha \gamma^i.
    \end{split}
\end{equation}
Renaming $2C_\alpha$ to $C_\alpha$, we obtain the thesis.
\end{proof}
}

Motivated by the latter inifinite-time horizon analysis for the Consensus Freezing scheme \eqref{eq:flip_scheme_2}, we return now to the global convergence of the scheme \eqref{eq:cbo_delta_mf} beyond the finite time horizon $T_*$.

\subsubsection{ Global Convergence of the \texorpdfstring{$\delta$-CBO}{delta-CBO} Scheme, revisited}\label{sec:delta_CBO2}

\begin{theorem}\label{thm:globconv_deltaCBO}    Let $\rho_t^\alpha$ be the unique solution to the dynamics \cref{eq:cbo_delta_mf} with the initial data $X_0\sim \rho_0 \in \mathcal{P}_4(\RR^d)$ (not necessarily a Gaussian). {Let $r=2S$}, {where $S$ is a constant fulfilling \cref{eq:second_moment_bound,eq:gamma_t_alpha_bound}} such that $x^* \in B_S(0)$ and $\EE_{\rho_0}[x] \in B_S(0)$, and assume that $\boldsymbol\rho_\infty^\alpha=\mc{N}(\boldsymbol\mu_\infty^\alpha,\frac{\delta^2}{2\lambda}I_d)$ is the invariant measure found in Theorem \ref{thm:invariant} {with $\boldsymbol\mu_\infty^\alpha\in B_r(x^*)$}.
    Then, for $\alpha$ sufficiently large it holds that
\begin{equation}\label{eq:global_conv_deltaCBO_full}
    W_2(\rho_t^\alpha, \boldsymbol\rho_\infty^\alpha)\leq e^{-\lambda t} C + |\EE_{\rho_0}[x] - \boldsymbol\mu_\infty^\alpha| e^{-\lambda (1 - L_\alpha) t} + \frac{L_X}{L_\alpha} D \big(e^{-\lambda (1 - L_\alpha)t} - e^{-\lambda t}\big),
\end{equation}
where $L_X$ is the constant found in \cref{eq:stability} for $R=S$, $D \coloneqq W_2\Big(\rho_0, \mc{N}\big(\EE_{\rho_0}[x],\frac{\delta^2}{2\lambda}I_d\big)\Big)$, and $L_\alpha$ is a Lipschitz constant of $\mc{T}^\alpha$ on $B_r(x^*)$ guaranteed by \cref{thm:invariant}.
In particular, for any $e^{-\lambda} < \gamma < 1$ and for $\alpha$ sufficiently large, it holds that
\begin{equation}\label{eq:global_conv_deltaCBO}
    W_2(\rho_t^\alpha, \boldsymbol\rho_\infty^\alpha) \leq C_\alpha \gamma^t, \quad \forall t \geq 0,
\end{equation}
for some $C_\alpha>0$.
\end{theorem}

\begin{proof}
    Following the exact same calculations as in the proof of \cref{thm:fsconv} with the slightly modified definition $\gamma_t^\alpha \coloneqq \Xalpha(\rho_t^\alpha)$---i.e., not freezing the consensus point---and setting $s = 1$, we obtain
    as in \cref{eq:differential_c_gamma'}
    \begin{equation}\label{eq:differential_c_gamma}
        \begin{split}
            \frac{\diff}{\diff t} |c_\gamma(t) - \boldsymbol\mu_\infty^\alpha|^2 &= 2 (c_\gamma(t) - \boldsymbol\mu_\infty^\alpha) \frac{\diff}{\diff t} c_\gamma(t) \\
            &= 2 (c_\gamma(t) - \boldsymbol\mu_\infty^\alpha) (-\lambda (c_\gamma(t) - \boldsymbol\mu_\infty^\alpha) + \lambda (\gamma_t^\alpha - \boldsymbol\mu_\infty^\alpha)) \\
            &\leq 2 \lambda \Big( - |c_\gamma(t) - \boldsymbol\mu_\infty^\alpha|^2  + |c_\gamma(t) - \boldsymbol\mu_\infty^\alpha| | \gamma_t^\alpha - \boldsymbol\mu_\infty^\alpha| \Big) \\
            &\leq 2 \lambda \Big( - (1 - L_\alpha) \cdot |c_\gamma(t) - \boldsymbol\mu_\infty^\alpha|^2 + |c_\gamma(t) - \boldsymbol\mu_\infty^\alpha| \cdot L_X W_2(\rho_0, \zeta_0) e^{-\lambda t} \Big).
        \end{split}
    \end{equation}
    As the right hand side is locally Lipschitz away from $0$, we can apply \cref{lem:comparison_principle} to majorize $|c_\gamma(t) - \boldsymbol\mu_\infty^\alpha|^2$ by a function $N(t)$ taking values in $(0, \infty)$ and solving the differential equation
    \begin{equation}
        \frac{\diff}{\diff t} N(t) = - 2\lambda \bigg((1 - L_\alpha) N(t) - L_X W_2(\rho_0, \zeta_0) e^{-\lambda t} N(t)^{1/2}\bigg), \quad N(0) = |c_0 - \boldsymbol\mu_\infty^\alpha|^2;
    \end{equation}
    this function is given by 
    \begin{equation}
        N(t) = \bigg(|c_0 - \boldsymbol\mu_\infty^\alpha| e^{-\lambda (1 - L_\alpha) t} + \frac{L_X}{L_\alpha} W_2(\rho_0, \zeta_0) \big(e^{-\lambda (1 - L_\alpha)t} - e^{-\lambda t}\big)\bigg) ^ 2.
    \end{equation}
    To finish the proof of \cref{eq:global_conv_deltaCBO_full}, we note that
    \begin{align*}
        W_2(\rho_t^\alpha, \boldsymbol\rho_\infty^\alpha) &\leq W_2(\rho_t^\alpha, \zeta_t) + W_2(\zeta_t, \boldsymbol\rho_\infty^\alpha) = W_2(\rho_t^\alpha, \zeta_t) + |c_\gamma(t) - \boldsymbol\mu_\infty^\alpha| \\
        &\leq e^{-\lambda t} W_2(\rho_0, \zeta_0) + N(t)^{1/2} \\
        &= e^{-\lambda t} W_2(\rho_0, \zeta_0) + |c_0 - \boldsymbol\mu_\infty^\alpha| e^{-\lambda (1 - L_\alpha) t} + \frac{L_X}{L_\alpha} W_2(\rho_0, \zeta_0) \big(e^{-\lambda (1 - L_\alpha)t} - e^{-\lambda t}\big).
    \end{align*}
    Lastly, \cref{eq:global_conv_deltaCBO} follows from \cref{eq:global_conv_deltaCBO_full} by noting that $L_\alpha \to 0$ as $\alpha \to \infty$ (see \cref{thm:invariant}) and thus, for any $e^{-\lambda} < \gamma < 1$, we can find $\alpha$ sufficiently large such that $e^{-\lambda (1 - L_\alpha)} \leq \gamma$.
\end{proof}

\begin{remark}
    This result confirms the exponential rate of convergence found in \eqref{eq:exponential_decay}, depending exclusively on $\lambda$.
\end{remark}

\subsubsection{Global-in-time bounds}\label{sec:global-in-time-bounds}
    In this subsection, we work towards the proof of \cref{thm:global_bound}, showing that we can determine global-in-time bounds of {higher moments} of solutions to both the $\delta$-CBO scheme \eqref{eq:cbo_delta_mf} and the Consensus Freezing scheme \eqref{eq:flip_scheme_2}.
    We start by proving that the conclusions of \cref{lem:laplace} hold uniformly over a certain class of measures.
    \begin{lemma}\label{lem:global_xalpha_bound}
       
        {For each $p \in \NN_{\geq 2}$, there exists a function $G_p\colon \RR_{>0} \times \mathcal{P}_p(\RR^d) \times \RR_{\geq 0} \to \RR_{>0}$ such that for all $\rho_0 \in \mathcal{P}_p(\RR^d)$, $\sigma_1 \geq \sigma_0 > 0$ and $R > 0$, we have
        \begin{equation}\label{eq:uniform_second_moment_bound}
            \begin{gathered}
                \EE_{\nu_{\sigma,k,z}}[|x|^p] \leq G_p(R, \rho_0, \sigma_1), \\
                \text{where } \nu_{\sigma,k,z} \coloneqq \mc{N}(0, \sigma^2 I_d) \ast \Big((s_{k,z})_\# \rho_0\Big) \in \mathcal{P}_p(\RR^d) \\
                \text{with } \sigma_1 \geq \sigma, z \in B_R(0) \subseteq R^d, 0 \leq k \leq 1,
            \end{gathered}
        \end{equation}
        where 
        \begin{equation*}
            \begin{split}
                s_{k, z}\colon & \RR^d \to \RR^d, \\
                & y \mapsto k y + z,
            \end{split}
        \end{equation*}
        and where $\ast$ denotes the convolution of measures and $(s_{k,z})_\# \rho_0$ denotes the pushforward of $\rho_0$ under the map $s_{k,z}$.
        } \\
        Furthermore, the convergences noted in \cref{eq: massconv,,eq: meanconv} hold uniformly over the class of measures
        \begin{equation}\label{eq:nu_uniform_class}
            \Big\{ \nu_{\sigma,k,z} \in P_2(\RR^d) \mid \sigma_1 \geq \sigma \geq \sigma_0, z \in B_R(0) \subseteq \RR^d, 0 \leq k \leq 1\Big\}.
        \end{equation}
    \end{lemma}
    \begin{proof}
        {
        We start with \cref{eq:uniform_second_moment_bound} by first calculating
        \begin{equation*}
            \begin{split}
            \EE_{(s_{k,z})_\# \rho_0}[|x|^p] &= \EE_{\rho_0}[|k x + z|^p] \leq \sum_{j=0}^p \binom{p}{j} |z|^{p-j} k^j \EE_{\rho_0}[|x|^j] \leq \sum_{j=0}^p R^{p-j} \EE_{\rho_0}[|x|^j] \eqqcolon H_p(R, \rho_0).
            \end{split}
        \end{equation*}
        Using the above calculation, we estimate
        \begin{equation}\label{eq:nuassum_second_moment}
            \begin{split}
            \EE_{x \sim \nu_{\sigma,k,z}}[|x|^p] &= \EE_{x \sim (s_{k,z})_\# \rho_0}\left[\EE_{y \sim \mc{N}(0, \sigma^2 I_d)}[|x - y|^p]\right] \\
            &\leq \sum_{j=0}^p \binom{p}{j} \EE_{x \sim (s_{k,z})_\# \rho_0}[|x|^{p-j}] \EE_{y \sim \mc{N}(0, \sigma^2 I_d)}[|y|^j]  \\
            &\leq \sum_{j=0}^p \binom{p}{j} H_{p-j}(R, \rho_0) \sigma_1^p\EE_{y \sim \mc{N}(0, I_d)}[|y|^j] \eqqcolon G_p(R, \rho_0, \sigma_1).
            \end{split}
        \end{equation}
        }
        For the second part of the lemma, in light of \cref{rmk:uniform_convergence} it suffices to show that for every $s > 0$, we can find a constant $C_s > 0$ such that
        \begin{equation}\label{eq:nuassum_2}
            \nu_{\sigma,k,z}(B_s(0)) \geq C_s,
        \end{equation}
        uniformly over all $\sigma, k, z$ as in \cref{eq:nu_uniform_class}. To this end, we first pick some $R' > 0$ such that $\rho_0(B_{R'}(0)) \eqqcolon C > 0$. Then, we have that
        \begin{equation*}
            \begin{split}
                \Big((s_{k, z})_\# \rho_0\Big)\big(B_{R'+R}(0)\big) &= \rho_0\Big(s_{k, z}^{-1}\big(B_{R'+R}(0)\big)\Big) = \rho_0\Bigg(B_{\frac{R' + R}{k}}\left(\frac{-z}{k}\right)\Bigg) \\
                &\geq \rho_0\Bigg(B_{R' + \frac{R}{k}}\left(\frac{-z}{k}\right)\Bigg) \geq C > 0.
            \end{split}
        \end{equation*}
        Thus, for all $y \in B_s(0)$, we have
        \begin{equation*}
            \begin{split}
            \frac{\diff \nu_{\sigma,k,z}}{\diff x}(y) &= \int_{\R^d} \frac{\diff\mc{N}(0, \sigma^2 I_d)}{\diff x}(y  - z) \Big((s_{k,z})_\# \rho_0\Big)(\diff z) \geq \int_{B_{R'+R}(0)} \frac{\diff\mc{N}(0, \sigma^2 I_d)}{\diff x}(y  - z) \Big((s_{k,z})_\# \rho_0\Big)(\diff x) \\
            &\geq \frac{C}{(2\pi \sigma^2)^{d/2}} \exp\Big(-\frac{(R' + R + s)^2}{2\sigma^2}\Big) \\
            &\geq  \frac{C}{(2\pi \sigma_1^2)^{d/2}} \exp\Big(-\frac{(R' + R + s)^2}{2\sigma_0^2}\Big) \eqqcolon D_s > 0.
            \end{split}
        \end{equation*}
        Finally, this gives us the desired bound
        \begin{equation}
            \nu_{\sigma,k,z}\big(B_s(0)\big) = \int_{B_s(0)} \frac{\diff \nu_{\sigma,k,z}}{\diff x}(y) \diff y \geq D_s \cdot |B_s(0)|
        \end{equation}
    \end{proof}
    Next, we show that solutions to both the $\delta$-CBO scheme \eqref{eq:cbo_delta_mf} and the Consensus Freezing scheme \eqref{eq:flip_scheme_2} are of the form $\nu_{\sigma,k,z}$ as in \cref{lem:global_xalpha_bound}.
    \begin{theorem}\label{lem:explicit_solution}
        Let $\rho_0,\varrho_0 \in P_2(\RR^d)$ be arbitrary initial measures, and let $\rho_t^\alpha, \ \varrho^{\alpha,s,\Delta t}_t$ be the solution to the $\delta$-CBO scheme \eqref{eq:cbo_delta_mf} and Consensus Freezing scheme \eqref{eq:flip_scheme_2}, respectively, with $\rho_0^\alpha = \rho_0$ and $ \varrho^{\alpha,s,\Delta t}_0 = \varrho_0$.
        Then,
        \begin{equation*}
        \begin{aligned}
            &\rho_t^\alpha = \mc{N}(0, \Sigma_t) \ast \Big((s_{k_t,c^1_\gamma(t)})_\# \rho_0\Big),\\
            &\varrho_t^{\alpha,s,\Delta t} = \mc{N}(0, \Sigma_t) \ast \Big((s_{k_t,c^2_\gamma(t)})_\# \varrho_0\Big)
        \end{aligned}
        \end{equation*}
        where $\Sigma_t = \frac{D}{\lambda} (1 - \exp(-2 \lambda t)) I_d$, with $D \coloneqq \frac{\delta^2}{2}$, 
        $k_t \coloneqq e^{-\lambda t}$ and, for $i=1,2$,$c^i_\gamma(t)$ is the solution to the differential equation
        \begin{equation}\label{eq:differential_equation_c_gamma_t}
            \frac{\diff}{\diff t} c^i_\gamma(t) = -\lambda (c^i_\gamma(t) - \gamma_t^i), \quad c^i_\gamma(0) = 0,
        \end{equation}
        where
        \begin{equation*}
            \gamma_t^i \coloneqq \begin{cases}
                \Xalpha(\rho_t^\alpha) & \text{if }\ i=1\ 
                 \text{($\delta$-CBO scheme)}, 
                \\
                \Xalpha(\varrho_{t_i}^{\alpha,s,\Delta t}) \quad \text{where } t_i \leq t < t_{i+1} & \text{if }\ i=2\ 
                \text{(Consensus Freezing scheme)}.
            \end{cases}
        \end{equation*}
   \end{theorem}
    \begin{proof}
        We let $X_t$ be a solution to the CBO dynamics (Consensus Freezing scheme, respectively) and introduce a reference process $Z_t$ governed by the SDE
        \begin{equation*}
            dZ_t = -\lambda Z_t \diff t + \delta \diff B_t,
        \end{equation*}
        driven by the same Brownian motion $B_t$ as $X_t$. As $Z$ follows the well-known Ornstein-Uhlenbeck process, we have the strong solution
        \begin{equation*}
            Z_t = Z_0 \cdot e^{-\lambda t} + \sigma \int_0^{t}e^{-\lambda (t -s)}\diff B_s, \qquad \text{hence }
            Z_t - Z_0 \cdot e^{-2\lambda t} = \sigma \int_0^{t}e^{-\lambda (t -s)}\diff B_s.
        \end{equation*}
        We see that $Z_t - Z_0 \cdot e^{-2\lambda t}$ is a Gaussian process with covariance
        \begin{equation*}
            cov\big(Z_t - Z_0 \cdot e^{-2\lambda t}\big) = \sigma^2 \int_0^{t} e^{-2\lambda (t - s)} \diff s \cdot I_d = \frac{\sigma^2}{2\lambda} (1 - e^{-2\lambda t}) \cdot I_d = \Sigma_t,
        \end{equation*}
        and hence $Z_t - Z_0 \cdot e^{-\lambda t} \sim \mc{N}(0, \Sigma_t I_d)$. 
        As $Z_0 \cdot e^{-\lambda t} \sim \big((s_{k_t,0})_\# \rho_0\big)$ and $Z_0$ and $Z_t - Z_0$ are independent, we have
        \begin{equation*}
            Z_t \sim \mc{N}(0, \Sigma_t) \ast \Big((s_{k_t,0})_\# \rho_0\Big).
        \end{equation*}
        Now, we consider the difference between $X_t$ and $Z_t$, to see that
        \begin{equation*}
            d(X_t - Z_t) = -\lambda (X_t - Z_t - \gamma_t^i) \diff t,
        \end{equation*}
        from which we obtain
        \begin{equation*}
            X_t - Z_t = c^i_\gamma(t) \qquad \text{i.e., } X_t = s_{1, c^i_\gamma(t)}(Z_t).
        \end{equation*}
        Thus, we have
        \begin{equation}
            \begin{split}
                X_t &\sim (s_{1, c^i_\gamma(t)})_\# \Big(\mc{N}(0, \Sigma_t) \ast \big((s_{k_t,0})_\# \rho_0\big)\Big) = \mc{N}(0, \Sigma_t) \ast \Big (s_{1, c^i_\gamma(t)})_\# (s_{k_t,0})_\# \rho_0\Big) \\
                &= \mc{N}(0, \Sigma_t) \ast \Big((s_{k_t, c^i_\gamma(t)})_\# \rho_0\Big),
            \end{split}
        \end{equation}
        where the first equality comes from the fact that convolution commutes with translation, and the second equality comes from $s_{1, c^i_\gamma(t)} \circ s_{k_t,0} = s_{k_t, c^i_\gamma(t)}$.
    \end{proof}
    Before we can prove the global-in-time bounds, we need to establish that the second moments stay bounded for a short time, uniformly over all $\alpha$ large enough.
    
    \begin{lemma}\label{lem:local_bound}
    Let $\rho_0,\varrho_0 \in P_2(\RR^d)$ be arbitrary initial measures, and let $\rho_t^\alpha, \ \varrho^{\alpha,s,\Delta t}_t$ be the solution to the $\delta$-CBO scheme \eqref{eq:cbo_delta_mf} and Consensus Freezing scheme \eqref{eq:flip_scheme_2}, respectively, with $\rho_0^\alpha = \rho_0$ and $ \varrho^{\alpha,s,\Delta t}_0 = \varrho_0$. Then, for any $T_0 > 0$, there exist $\alpha_0 > 0$, and $S > 0$, such that for all $\alpha \geq \alpha_0$ the solutions satisfy
        \begin{equation}\label{eq:local_second_moment_bound}
        \EE_{\rho_t^\alpha} \left[ |x|^2 \right] \leq S,\quad \EE_{\varrho_t^{\alpha,s,\Delta t}} \left[ |x|^2 \right] \leq S, \quad \forall t \in [0, T_0].
    \end{equation}
    and
    \begin{equation}\label{eq:local_gamma_t_alpha_bound}
        |\Xalpha(\rho_t^\alpha)| \leq S,\quad |\Xalpha(\varrho_t^{\alpha,s,\Delta t})| \leq S, \quad \forall t \in [0, T_0].
    \end{equation}
    \end{lemma}
    \begin{proof}We prove the inequalities for $\rho^\alpha_t.$ For $\varrho^{\alpha,s,\Delta t}_t$ the proof is completely analogous.
    We may choose  any $T_0 > 0$. We pick some arbitrary $R > R' > |x^\ast|$ and set
    \begin{align*}
        T_1^\alpha &\coloneqq \inf \{ t \geq 0 \mid |\Xalpha(\rho^\alpha_t)| > R \},\, T_0^\alpha \coloneqq \min\{T_1^\alpha, T_0\}. \\
        T_2^\alpha &\coloneqq \inf \{t \geq 0 \mid \EE_{\rho^\alpha_t}(|x|^2) > S\}, \text{ where } \\
        S &{\coloneqq G_2(R, \rho_0, \frac{\sigma^2}{2\lambda})},
    \end{align*}
    where $G_2$ is the function found in \cref{eq:uniform_second_moment_bound} for $p=2$.
    From a similar comparison principle as in \cref{eq:differential_c_gamma_bounds}, we see that for any $\alpha > 0$ and any $t \in [0, T_1^\alpha]$, we have that the $c_\gamma^i$ defined in \cref{eq:differential_equation_c_gamma_t} fulfills $c_\gamma^i(t) \leq R$. Hence, by \cref{lem:explicit_solution,lem:global_xalpha_bound}, we directly get that $T_2^\alpha \geq T_1^\alpha$, for any $\alpha > 0$. We thus have, for any $\alpha > 0$, the inequalities $T_2^\alpha \geq T_1^\alpha  \geq T_0^\alpha$ and $T_0 \geq T_0^\alpha$. From \cref{eq:differential_equation_c_gamma_t}, we get that
    $|c_\gamma(t)| \leq R$ for all $t \in [0, T_1^\alpha]$. Furthermore, we get from \cref{lemma:ball_measure} that we can find, for any $s > 0$, a constant $C_s > 0$ such that
    \begin{equation*}
        \rho^\alpha_t(B_s(x^\ast)) \geq C_s > 0, \quad \forall t \in [0, T_0^\alpha],\, \forall \alpha > 0.
    \end{equation*}
    Thus, we see that \cref{eq:nuassum} holds uniformly over $\{\rho^\alpha_t \mid t \in [0, T^\alpha_0], \alpha > 0\}$. We now apply {\cref{lem:laplace,lem:global_xalpha_bound,lem:explicit_solution}}---similarily as in \cref{eq:convergence_of_consensus}---to see that $\Xalpha(\rho^\alpha_t) \to x^\ast$ as $\alpha \to \infty$ uniformly over $ t \in [0, T^\alpha_0]$ in the following sense: For every $\varepsilon > 0$, there exists $\alpha_0 > 0$ such that for all $\alpha \geq \alpha_0$ and for all $t \in [T_0, T'_\alpha]$, it holds that $|\Xalpha(\rho^\alpha_t) - x^\ast| < \varepsilon$. This implies that for $\alpha$ large enough, $\Xalpha(\rho^\alpha_t) < R'$ for all $t \in [0, T^\alpha_0]$, which in turn implies that $T_1^\alpha > T_0^\alpha$ by continuity of $\Xalpha(\rho^\alpha_t)$. Thus, $T_0^\alpha = T_0$ for all $\alpha$ large enough, which finishes the proof.
    \end{proof}
    We are now ready to prove \cref{thm:global_bound}.
    \begin{proof}[Proof of \cref{thm:global_bound}]
        We prove the result for the $\delta$-CBO scheme (that is, for $\rho^\alpha_t$). For the Consensus Freezing scheme the proof is completely analogous.
        By \cref{lem:local_bound}, we may assume that \cref{eq:second_moment_bound,eq:gamma_t_alpha_bound} hold with some $R > |x^\ast|$ for all $t \in [0, T_0]$ with $T_0 > 0$.
        Thus, we have that 
        \begin{equation}
            T'_\alpha \coloneqq \inf \{ t \geq 0 \mid |\Xalpha(\rho^\alpha_t)|  > R + 1 \} > T_0.
        \end{equation}
        We deduce from \cref{eq:differential_equation_c_gamma_t} that
        $|c_\gamma(t)| \leq R + 1$ for all $t \in [0, T'_\alpha]$ and $\alpha$ large enough. \\
        Defining $\sigma_1 \coloneqq \sqrt{\frac{D}{\lambda}} > 0$, $\sigma_0 \coloneqq \sigma_1 \cdot \sqrt{1 - \exp(-\lambda T_0)} > 0$, we can apply \cref{lem:explicit_solution} to see that $\rho_t^\alpha$ is of the form \eqref{eq:nu_uniform_class} for all $t \in [T_0, T'_\alpha)$ and $\alpha > 0$. Thus, we can apply \cref{lem:global_xalpha_bound} and obtain
    \begin{equation}\label{eq:convergence_of_consensus}
            |\Xalpha(\rho^\alpha_t) - x^\ast| \to 0 \text{ as } \alpha \to \infty,
        \end{equation}
        uniformly over $t \in [T_0, T'_\alpha]$, in the following sense: For every $\varepsilon > 0$, there exists $\alpha_0 > 0$ such that for all $\alpha \geq \alpha_0$ and for all $t \in [T_0, T'_\alpha]$, it holds that $|\Xalpha(\rho^\alpha_t) - x^\ast| < \varepsilon$.
        Thus, we can conclude that $|\Xalpha(\rho^\alpha_t)| \leq R$ for all $t \in [0, T'_\alpha]$ for large enough $\alpha$, which implies that $T'_\alpha = \infty$ and thus \cref{eq:gamma_t_alpha_bound} holds with $S = R$. \cref{eq:second_moment_bound} follows from \cref{eq:gamma_t_alpha_bound} by \cref{eq:uniform_second_moment_bound}.
    \end{proof}

\subsubsection{Behavior in presence of multiple minimizers}
In the following we investigate the behavior of the optimization methods presented in this paper when  we relax the assumption that $f(\cdot)$ possesses a unique minimizer and instead adopt the following assumption.

\begin{assum}\label{assum4}
The function $f:\R^d\to\R$ attains its global
minimum $\underline f$, and the set of global
minimizers
\[
\mathcal M
:=
\left\{
x\in\R^d:f(x)=\underline f
\right\}
\]
is nonempty and compact.
\end{assum}

For every $\varepsilon>0$, we define the
$\varepsilon$-neighborhood of $\mathcal M$ by
\[
\mathcal M_\varepsilon
:=
\left\{
x\in\R^d:
\operatorname{dist}(x,\mathcal M)
\leq\varepsilon
\right\}\,,
\]
and denote its complement by $\mathcal{M}_\varepsilon^c = \mathbb{R}^d \setminus \mathcal{M}_\varepsilon$.  
\\

We  analyze specifically how the invariant measure $\boldsymbol\rho_\infty^\alpha$ is centered with respect to $\mathcal M$.
\\

First of all, one can readily extend Lemma \ref{lem:laplace} to the multiple-minimizer case:
\begin{lemma}\label{lem:laplace'}
    Assume $f(\cdot)$ satisfies Assumption \ref{assumptions}-\textbf{A1} and Assumption \ref{assum4}, and let $\nu$ be a probability measure satisfying
    \begin{equation}\label{eq:nuassum'}
        \EE_\nu[|x|^2]\leq c_2\quad \mbox{ and }\quad  \nu (\mathcal{M}_s)\geq C_s>0\quad \mbox{ for any } s>0\,.
    \end{equation}
Define the probability measure
    \begin{equation}\label{eq:eta'}
        \eta_\nu^\alpha(A):=\frac{\int_{A}e^{-\alpha f(x)}\nu(\diff x)}{\int_{\RR^d}e^{-\alpha f(y)}\nu(dy)},\quad A\in \mc{B}(\RR^d).
    \end{equation}
    Then for any $\varepsilon>0$, it holds 
    \begin{equation}\label{eq: massconv'}
        \eta_\nu^\alpha( \mathcal{M}_\varepsilon^c)\longrightarrow 0  \quad \mbox{ as }  \quad \alpha \to \infty\,.
    \end{equation}
\end{lemma}
\begin{proof}
The proof proceeds along the same lines as that of Lemma~\ref{lem:laplace}, with the only modification being the replacement of $B_r(x^*)$ by $\mathcal{M}_r$. Indeed, similar to \eqref{massconverge}, one can obtain
\begin{equation*}
     \eta_\nu^\alpha(\mathcal{M}_\varepsilon^c)
   \leq \frac{e^{-\alpha (\overline f_{\varepsilon'} - \underline f)}}{\nu(\mathcal{M}_{\varepsilon'})}\leq \frac{e^{-\alpha (\overline f_{\varepsilon'} - \underline f)}}{C_{\varepsilon'}}\,\; \to 0 \; \mbox{ when } \;\alpha \to \infty.
\end{equation*}
\end{proof}
Similarly, using the above lemma, one can also adapt the results in Theorem \ref{thm:invariant} to the multiple-minimizer case:
{\begin{theorem}\label{thm:invariant'}
Let $\operatorname{conv}(\mathcal{M})$ denote the convex hull of $\mathcal{M}$, and define
\[
K_r:=
\left\{
x\in\R^d:
\operatorname{dist}
(x,\operatorname{conv}(\mathcal M))\le r
\right\}.
\]
Then for any $r>0$ and $\alpha$
sufficiently large, the map $\mc{T}^\alpha$
defined in \eqref{eq:T_alpha} satisfies
$\mc{T}^\alpha(K_r)\subseteq K_r,$ and it admits at least one fixed point
$\boldsymbol\mu_\infty^\alpha\in K_r$.
 Furthermore, it satisfies
    \begin{equation}\label{eq:invariant'}
\boldsymbol\rho_\infty^\alpha=\mc{N}\left(\boldsymbol\mu_\infty^\alpha,\frac{\delta^2}{2\lambda}I_d\right)=\mc{N}\left(\mc{T}^\alpha(\boldsymbol\mu_\infty^\alpha),\frac{\delta^2}{2\lambda}I_d\right)=\mc{N}\left(\Xalpha(\boldsymbol\rho_\infty^\alpha),\frac{\delta^2}{2\lambda}I_d\right)\,,
    \end{equation}
    and the following holds: 
 \begin{equation}
     \operatorname{dist}(\Xalpha(\boldsymbol\rho_\infty^\alpha),\operatorname{conv}(\mathcal{M}))\to 0\quad \mbox{ as }\alpha\to\infty \,.
 \end{equation}
 In particular, if $\mathcal{M}$ is convex, then $\operatorname{conv}(\mathcal{M})$ can be replaced by $\mathcal{M}$.
\end{theorem}}

\begin{proof}
We follow the proofs in Theorem \ref{thm:invariant} and \cite[Theorem 4.9]{huang2025faithful}.  We first explain why the contraction argument used
in the proof of Theorem~\ref{thm:invariant}
cannot be directly extended to the case of
multiple global minimizers.

Recall that $\nabla\mc{T}^\alpha(\mu)
=
\frac{2\lambda}{\delta^2}
\operatorname{Cov}(\eta_\mu^\alpha),$ 
where $\eta_\mu^\alpha
=
\eta_{\mc{N}(\mu,\frac{\delta^2}{2\lambda}I_d)}^\alpha.$
When the global minimizer $x^*$ is unique,
the Gibbs measure $\eta_\mu^\alpha$
concentrates at $x^*$ as $\alpha\to\infty$.
Consequently, its covariance converges to zero,
which leads to the contraction estimate $\|\nabla\mc{T}^\alpha(\mu)\|\to0.$
 In contrast, when $f$ admits multiple global
minimizers, the limiting Gibbs measure may
distribute its mass among different minimizers,
and its covariance need not vanish.

However, we can still show that $\mc{T}^\alpha$ maps $K_r$ into
itself for sufficiently large $\alpha$.
Arguing similarly to Step 1 in the proof of
Theorem~\ref{thm:invariant}, and using
Lemma~\ref{lem:laplace'}, we obtain
\[
\sup_{\mu\in K_r}
\operatorname{dist}\left(
\mc{T}^\alpha(\mu),
\operatorname{conv}(\mathcal M)
\right)
\longrightarrow 0
\qquad\text{as }\alpha\to\infty.
\]
Here we use the convexity of
$\operatorname{conv}(\mathcal M)$ and the uniform
moment bounds for the Gaussian measures
$\mc{N}(\mu,\frac{\delta^2}{2\lambda}I_d)$,
$\mu\in K_r$.
Consequently, for $\alpha$ sufficiently large,
\[
\mc{T}^\alpha(K_r)\subseteq K_r.
\]
Since $K_r$ is nonempty, compact, and convex,
Brouwer's fixed-point theorem yields a fixed
point $\boldsymbol\mu_\infty^\alpha\in K_r$.

Moreover, we use Lemma \ref{lem:laplace'} and the computation in \eqref{distx*}
\begin{align*}
   \operatorname{dist}(\Xalpha(\boldsymbol\rho_\infty^\alpha), \operatorname{conv}(\mathcal{M}))&\leq \int_{\RR^d}\operatorname{dist}(x,\operatorname{conv}(\mathcal{M}))\eta_{\rho_\infty^\alpha}^\alpha(\dd x) \leq \int_{\RR^d}\operatorname{dist}(x,\mathcal{M})\eta_{\rho_\infty^\alpha}^\alpha(\dd x)\\
   &\leq \int_{\mc{M}_\varepsilon}\operatorname{dist}(x,\mc{M})\eta_{\rho_\infty^\alpha}^\alpha(\dd x)+\int_{\mc{M}_\varepsilon^c}\operatorname{dist}(x,\mc M)\eta_{\rho_\infty^\alpha}^\alpha(\dd x)\notag\\
  &\leq \varepsilon+ C\eta_{\rho_\infty^\alpha}^\alpha(\mathcal{M}_\varepsilon^c)^{1/2}\to 0\,.
\end{align*}
\end{proof}
Unfortunately, except for the case where $\mathcal M$ is convex, it is not possible to claim $ \operatorname{dist}(\Xalpha(\boldsymbol\rho_\infty^\alpha),\mathcal{M})\to 0$, as $\alpha\to\infty$, as the following counterexample shows.
\begin{remark}
Consider, for instance, the symmetric one-dimensional
double-well potential $f(x)=(x^2-1)^2,$
whose set of global minimizers is
$\mathcal M=\{-1,1\}$.
Since $f$ is even, we have $\mc{T}^\alpha(0)=0$
for every $\alpha>0$. Consequently, $\boldsymbol\rho_\infty^\alpha
=
\mc{N}\left(
0,\frac{\delta^2}{2\lambda}
\right)$
is an invariant measure of the mean-field dynamics,
whose mean 0 belongs to
$\operatorname{conv}(\mathcal M)\setminus\mathcal M$.

In contrast, in practical implementations via the finite-particle system, these symmetries are naturally broken. Specifically, for the finite-$N$ particle system \eqref{eq:flip_scheme_finite} in the symmetric double-well potential,  random initialization and independent Brownian fluctuations almost surely break the symmetry around $0$: the probability of the empirical consensus point $\Xalpha(\widehat\varrho^{\alpha,s,\Delta t,N}_{t_i})$ remaining trapped at $0$ is zero, whereas it converges to $-1$ or $1$ with probability $1/2$ each. Thus, in practical particle realizations, the consensus point of the CBO dynamics effectively selects and converges to an actual global minimizer in $\mathcal{M}$.

We nevertheless reiterate that the CBO models studied in this paper are not tailored to the multiple-minimizer setting. CBO variants specifically designed for this purpose have been introduced in \cite{B1-fornasierSun25,bungert2022polarized}.
\end{remark}

\subsection{Convergence of the Consensus Freezing scheme to the $\delta$-CBO Scheme}
In this section, we show that for $\Delta t$ that tends to $0$, the Consensus Freezing scheme \eqref{eq:flip_scheme_2} converges to the $\delta$-CBO scheme \eqref{eq:cbo_delta_mf}.
In this section, we will assume that $s=1$ in the definition of the Consensus Freezing scheme \eqref{eq:flip_scheme_2}.
When the initial distribution is Gaussian, the same results can be obtained using a different argument, which we present with Lemma \ref{lemma:CFtodelta_N} in the Appendix for the sake of completeness.

\begin{theorem}\label{thm:CFtodCBO}
    Let $\rho_0 \in \mc{P}_4(\R^d)$ and $\alpha > 0$ be given. For each $\Delta t > 0$, let $\varrho_{t}^{\alpha,1,\Delta t}$ be the solution to the Consensus Freezing scheme \eqref{eq:flip_scheme_2} with step size $\Delta t$, $s=1$, initial data $\rho_0$, and given parameter $\alpha$; let furthermore $\rho^\alpha_t$ be the solution for the $\delta$-CBO scheme \eqref{eq:cbo_delta_mf} with the same initial data $\rho_0$ (not necessarily Gaussian).
    Then, for every $T > 0$, we have that
    \begin{equation}
        W_2(\varrho_{t}^{\alpha,1,\Delta t}, \rho_t^\alpha) \to 0 \quad \text{as } \Delta t \to 0, \text{ uniformly for } t \in [0,T].
    \end{equation}
\end{theorem}
\begin{proof}
    To ease the notation, since $s=1$ in this proof we will write $\varrho_{t}^{\alpha,\Delta t}=\varrho_{t}^{\alpha,1,\Delta t}$.
    % Let $\rho_0 \in P_2(\R^d)$ and $\alpha > 0$ be given, and let us denote by $\rho_t^{\Delta t}$ and $\rho^\alpha_t$ the solutions to \eqref{eq:flip_scheme_2} with step size $\Delta t$ and \eqref{eq:cbo_delta_mf}, respectively, with initial data $\rho_0$.
    By \cref{thm:global_bound}, we may choose $R > 0$ such that $\varrho_{t}^{\alpha,\Delta t} \in \mc{P}_{4}(\R^d)$ and $\rho^\alpha_t \in \mc{P}_{4}(\R^d)$ with fourth moment bounded by $R$, for all $t \in [0,T]$ and all $\Delta t > 0$.
    Furthermore, by \cref{lem: xalpha bound}, we may also choose $L_X > 0$ such that $\Xalpha$ is $L_X$-Lipschitz on measures in $\mc{P}_{4}(\R^d)$ with fourth moment bounded by $R$.
    Let us first show that we can choose $D$ such that
    \begin{equation}
        W_2(\varrho_{t}^{\alpha,\Delta t}, \varrho_{t + \varepsilon}^{\alpha,\Delta t}) \leq D \sqrt{\varepsilon} \qquad \forall t, \varepsilon, \Delta t > 0.
    \end{equation}
    To this end, we denote by $X_t^{\Delta t}$ the solution to \eqref{eq:flip_scheme_2} with initial data $\rho_0$
    %, set $\gamma_t^{\Delta t} = \Xalpha(\varrho_{t_i}^{\alpha,\Delta t})$, 
    and apply It\^o's formula %to $(\cdot - X_t)^2$ to see that
to     obtain
    \begin{equation*}
        \begin{split}
            (X_{t + \varepsilon}^{\Delta t} - X_t^{\Delta t})^2 &= \int_t^{t + \varepsilon} 2 (X_s^{\Delta t} - X_t^{\Delta t}) \cdot \big(-\lambda (X_s^{\Delta t} -\Xalpha(\varrho_{t_i}^{\alpha,\Delta t}))\big) \diff s \\ &+ \int_t^{t + \varepsilon} 2 (X_s^{\Delta t} - X_t^{\Delta t}) \cdot \delta \diff B_s + \int_t^{t + \varepsilon} \delta^2 \diff s,
        \end{split}
    \end{equation*}
    and hence
    \begin{equation*}
        \begin{split}
            \EE[(X_{t + \varepsilon}^{\Delta t} - X_t^{\Delta t})^2] &\leq \int_t^{t + \varepsilon} 2 \lambda \EE[|X_s^{\Delta t} - X_t^{\Delta t}| \cdot |X_s^{\Delta t} - \Xalpha(\varrho_{t_i}^{\alpha,\Delta t})|] \diff s + \delta^2 \varepsilon \\
            &\leq \int_t^{t + \varepsilon} 2 \lambda (\underbrace{\EE[|X_s^{\Delta t} - X_t^{\Delta t}|^2]}_{\leq 4 R})^{1/2} (\underbrace{\EE[|X_s^{\Delta t} - \Xalpha(\varrho_{t_i}^{\alpha,\Delta t})|^2]}_{\leq 2R + 2RC})^{1/2} \diff s + \delta^2 \varepsilon,
        \end{split}
    \end{equation*}
    which shows the claim. \\
    Next, let us fix $\Delta t$, and let $X_t^{\Delta t}$ and $Y_t$ be solutions to \eqref{eq:flip_scheme_2} and \eqref{eq:cbo_delta_mf}, respectively, with initial data $\rho_0$ and driven by the same Brownian motion $B_t$.
    We then have that
    \begin{equation}
        \frac{\diff}{\diff t} (X_t^{\Delta t} - Y_t) = -\lambda (X_t^{\Delta t} - Y_t) + \lambda (\Xalpha(\varrho_{t_i}^{\alpha,\Delta t}) - \Xalpha(\rho^\alpha_t)),
    \end{equation}
    and thus, setting $X^D_t \coloneqq |X_t^{\Delta t} - Y_t|^2$, we have
    \begin{equation}
        \begin{split}
            \frac{\diff}{\diff t} |X_t^{\Delta t} - Y_t|^2 &= -2 \lambda |X_t^{\Delta t} - Y_t|^2 + 2 \lambda (X_t^{\Delta t} - Y_t) \cdot (\Xalpha(\varrho_{t_i}^{\alpha,\Delta t}) - \Xalpha(\rho^\alpha_t)) \\
            &\leq -2 \lambda |X_t^{\Delta t} - Y_t|^2 + 2 \lambda |X_t^{\Delta t} - Y_t| \cdot \Big(|X^\alpha(\rho^\alpha_t) - X^\alpha(\varrho_{t_i}^{\alpha,\Delta t})| + |X^\alpha(\varrho_{t}^{\alpha,\Delta t}) - X^\alpha(\varrho_{t_i}^{\alpha,\Delta t})|\Big) \\
            &\leq -2 \lambda |X_t^{\Delta t} - Y_t|^2 + 2 \lambda L_X |X_t^{\Delta t} - Y_t| \cdot \Big(\sqrt{\EE[|X_t^{\Delta t} - Y_t|^2]} + D \sqrt{\Delta t}\Big).
        \end{split}
    \end{equation}
    Taking expectation, we arrive at the differential inequality
    \begin{equation}
        \frac{\diff}{\diff t} \EE[|X_t^{\Delta t} - Y_t|^2] \leq 2 \lambda (L_X - 1) \EE[|X_t^{\Delta t} - Y_t|^2] + 2 \lambda L_X D \sqrt{\Delta t} \cdot (\EE[|X_t^{\Delta t} - Y_t|^2])^{1/2}.
    \end{equation}
    We finish the proof by noting that $W^2_2(\varrho_{t}^{\alpha,\Delta t}, \rho^\alpha_t)^2 \leq \EE[|X_t^{\Delta t} - Y_t|^2]$, that the right hand side above is locally Lipschitz away from zero, and by applying a comparison principle with the solution of the ODE
    \begin{equation*}
        \frac{\diff}{\diff t} \psi(t) = 2 \lambda (L_X - 1) \psi(t) + 2 \lambda L_X D \sqrt{\Delta t} \cdot \psi(t)^{1/2}, \quad \psi(0) = 0
    \end{equation*}
    given by
    \begin{equation*}
        \psi(t) = \Big(D \sqrt{\Delta t} \cdot \frac{L_X}{L_X - 1} (e^{\lambda (L_X - 1) t} - 1)\Big)^2,
    \end{equation*}
    which yields the estimate
    \begin{equation}
        W_2(\varrho_{t}^{\alpha,\Delta t}, \rho^\alpha_t) \leq D \sqrt{\Delta t} \cdot \frac{L_X}{L_X - 1} (e^{\lambda (L_X - 1) t} - 1) \to 0 \quad \text{as } \Delta t \to 0,
    \end{equation}
\end{proof}

In the next section, we study a further asymptotics of the Consensus Freezing scheme, for the parameter $s \to \infty$. With this, we will derive one more discrete time global optimization scheme.

\section{The Consensus Hopping Scheme}\label{sec:HS}
Introduced in \cite{riedl_gradient_2023}, the \emph{Consensus Hopping scheme} is defined as follows:
\begin{equation}\tag{CH}\label{eq:consensus-hopping}
    \begin{split}
        x_k^{CH} &= \Xalpha(\nu_k^\alpha), \quad \text{with} \quad \nu_k^\alpha = \Normal(x_{k-1}^{CH}, \tilde{\sigma}^2 I_d), \\
        x_0^{CH} &= x_0.
    \end{split}
\end{equation}
where $\tilde{\sigma}$ is a fixed parameter.

\subsection{Convergence of the Consensus Freezing scheme to the Consensus Hopping Scheme}
In this section, we show that for $s \to \infty$, solutions to \eqref{eq:flip_scheme_2} converge to solutions to \eqref{eq:consensus-hopping} with $\tilde{\sigma} = \frac{\delta}{\sqrt{2\lambda}}$. To start, recall from the discussion around \cref{iterative'} that the solution to the Consensus Freezing scheme with given $s,\Delta t>0$, starting with $\varrho_0 = \mathcal{N}(x_0, \Sigma)$ with $\Sigma \coloneqq \frac{\delta^2}{2\lambda} I_d$, is given by
% \begin{equation*}
%     \begin{split}
%         \varrho^{\alpha,1,\Delta t}_t &= \mathcal{N}(\mu^{\alpha,1,\Delta t}_t, \Sigma), \qquad \text{where} \\
%         \mu^{\alpha,1,\Delta t}_t &= \Xai+e^{-\lambda(t-t_i)}(\mu^{\alpha,1,\Delta t}_{t_i}-\Xalpha(\varrho^{\alpha,1,\Delta t}_{t_i})), \quad t \in (t_i, t_{i+1}].
%     \end{split}
% \end{equation*}
% Thus, for arbitrary $s > 0$ and using the mapping $\mc{T}^\alpha$ introduced in \cref{eq:T_alpha}, we have that a solution to the Consensus Freezing scheme with initial data $\rho_0 = \mathcal{N}(x_0, \Sigma)$ is given by
\begin{equation}\label{eq:flip_scheme_solution_s}
    \begin{split}
        \varrho^{\alpha,s,\Delta t}_t &= \mathcal{N}(\mu^{\alpha,s,\Delta t}_t, \Sigma), \qquad \text{where} \\
        \mu^{\alpha,s,\Delta t}_t &= \mc{T}^\alpha(\mu^{\alpha,s,\Delta t}_{t_i})+e^{-s \cdot \lambda(t-t_i)}(\mu^{\alpha,s,\Delta t}_{t_i}-\mc{T}^\alpha(\mu^{\alpha,s,\Delta t}_{t_i})), \quad t \in (t_i, t_{i+1}].
    \end{split}
\end{equation}
We are now in the position to prove the convergence of the Consensus Freezing scheme to the Consensus Hopping scheme.
\begin{theorem}\label{thm:FStoHS}
    Let $x_0 \in \RR^d$ be an arbitrary initial point, and let $\{x_k^{CH}\}_{k \in \NN}$ be the solution to the Consensus Hopping scheme \eqref{eq:consensus-hopping} with initial data $x_0$ and $\tilde{\sigma} = \frac{\delta}{\sqrt{2\lambda}}$. Furthermore, for $s > 0$, let $\mu_t^{\alpha,s,\Delta t}$ be the solution to the Consensus Freezing scheme \eqref{eq:flip_scheme_2} with initial data $\varrho_0 = \mathcal{N}(x_0, \Sigma)$, where $\Sigma = \frac{\delta^2}{2\lambda} I_d$, and using the prescribed $s$. Then, for large enough $\alpha$, we have that
    \begin{equation}\label{eq:FS_to_HS}
        |\mu^{\alpha,s,\Delta t}_{t_k} - x_k^{CH}| \leq 4 R \cdot e^{-s \lambda \Delta t} \to 0 \quad \text{as } s \to \infty, \quad \forall k \in \NN,
    \end{equation}
    where $R > 0$ is a uniform bound on $\mu_t^s$ given by \cref{thm:global_bound} such that $x^* \in B_R(0)$.
    In particular, we have that
    \begin{equation}\label{eq:W2_FS_to_HS}
        W_2(\varrho_{t_k}^{\alpha,s,\Delta t}, \mathcal{N}(x_k^{CH}, \Sigma)) \leq 4 R \cdot e^{-s \lambda \Delta t} \to 0 \quad \text{as } s \to \infty,
    \end{equation}
    {uniformly with respect to $k\in \mathbb N$}.
\end{theorem}
\begin{proof}
    \cref{eq:W2_FS_to_HS} follows directly from \cref{eq:FS_to_HS} together with the fact that both $\rho_{t_k}^{\alpha,s,\Delta t}$ and $\mathcal{N}(x_k^{CH}, \Sigma)$ are Gaussian with the same covariance matrix $\Sigma = \frac{\delta^2}{2\lambda} I_d$. We will show \cref{eq:FS_to_HS} by induction on $k$. Before, we pick $\alpha > 0$ and $R > 0$ such that the $\mu_{t_k}^{\alpha,s,\Delta t}\in B_R(0)$---which is possible by \cref{thm:global_bound}---, such that $x^* \in B_R(0)$ and such that $\mc{T}^\alpha$ is Lipschitz continuous on $B_{6R}(x^*)$ with Lipschitz constant $L_\alpha \leq \frac{1}{2}$---which is possible by \cref{thm:invariant}. \\ 
    We fix $k \in \mathbb{N}$; by induction assumption, we have that $|\mu^{\alpha,s,\Delta t}_{t_k} - x_k^{CH}| \leq 4 R \cdot e^{-s \lambda \Delta t}$, and thus $x_k^{CH} \in B_{5R}(0) \subseteq B_{6R}(x^*)$. From \cref{eq:flip_scheme_solution_s}, we thus get that 
    \begin{equation*}
        \begin{split}
    |\mu_{t_{k+1}}^{\alpha,s,\Delta t}- x_{k+1}^{CH}| &= |\mc{T}^\alpha(\mu_{t_k}^{\alpha,s,\Delta t}) + e^{-s \cdot \lambda(t_{k+1}-t_k)}(\mu_{t_k}^{\alpha,s,\Delta t} - \mc{T}^\alpha(\mu_{t_k}^{\alpha,s,\Delta t})) - \mc{T}^\alpha(x_k^{CH})| \\
    &\leq |\mc{T}^\alpha(\mu_{t_k}^{\alpha,s,\Delta t}) - \mc{T}^\alpha(x_k^{CH})| + e^{-s \cdot \lambda(t_{k+1}-t_k)}|\mu_{t_k}^{\alpha,s,\Delta t} - \mc{T}^\alpha(\mu_{t_k}^s)| \\
    &\leq L_\alpha 4 R \cdot e^{-s \lambda \Delta t} + e^{-s \lambda \Delta t} \cdot 2 R = 4 R \cdot e^{-s \lambda \Delta t}.
        \end{split}
    \end{equation*}
\end{proof}

As a consequence of the asymptotics above, we can also derive the global convergence of the Consensus Hopping scheme to the invariant measure $\boldsymbol\rho_\infty^\alpha$.

\begin{corollary}\label{cor:HS_convergence}
 Let $x_0 \in \RR^d$ be an arbitrary initial point, and let $\{x_k^{CH}\}_{k \in \NN}$ be the solution to the Consensus Hopping scheme \eqref{eq:consensus-hopping} with initial data $x_0$ and $\tilde{\sigma} = \frac{\delta}{\sqrt{2\lambda}}$. Moreover, for any $k\in\NN$, let $\nu_k^\alpha=\mathcal N(x_k^{CH}, \frac{\delta^2}{2\lambda}I)$. 
 Then, for any $\gamma \in (0, 1)$ and for $\alpha$ sufficiently large, it holds that
 \begin{equation*}
        W_2(\nu_k^\alpha, \boldsymbol\rho_\infty^\alpha) \leq C_\alpha \gamma^k,
 \end{equation*}
 for some $C_\alpha > 0$ depending on $\alpha$.
%  Then, for $\alpha$ large enough, we have that
%  $$W_2(\nu_k^\alpha,\boldsymbol\rho_\infty^\alpha)\longrightarrow0 \quad \text{as } k\to+\infty$$
%  {\color{red} MASSIMO: it would be great to specify in which sense --- quantitatively --- $W_2(\nu_k^\alpha,\boldsymbol\rho_\infty^\alpha)\longrightarrow0 \quad \text{as } k\to+\infty$.}
\end{corollary}
We present two proofs of this result. The first one follows the alignment between the Consensus Hopping scheme and the Consensus Freezing scheme as $s \to \infty$ developed in \cref{thm:FStoHS} together with the convergence of the Consensus Freezing scheme dynamics to the invariant measure $\boldsymbol\rho_\infty^\alpha$ as $k\to+\infty$ established in \cref{thm:fsconv}. The second proof uses \cref{thm:invariant} directly.
\begin{proof}[First proof of \cref{cor:HS_convergence}]
    Let $\gamma \in (0, 1)$ be given. We may first choose $\Delta t, s_0 > 0$ such that $e^{-s_0 \lambda \Delta t} < \gamma$. Using \cref{thm:fsconv}, we may choose $\alpha_0$ such that for all $\alpha \geq \alpha_0$ and all $s \geq s_0$, it holds that
    \begin{equation}\label{eq:fs_conv_corollary}
        W_2(\varrho_{t_k}^{\alpha,s,\Delta t}, \boldsymbol\rho_\infty^\alpha) \leq C_\alpha \gamma^k,
    \end{equation}
    for some $C_\alpha > 0$ depending on $\alpha$, where $\varrho_{t_k}^{\alpha,s,\Delta t}$ is the solution to \eqref{eq:flip_scheme_2} at time $t_k$. 
    By Theorem \ref{thm:FStoHS}, for any $\varepsilon > 0$ there exists $s_\varepsilon \geq s_0$ such that for any $s\geq s_\varepsilon$
    $$W_2(\varrho_{t_k}^{\alpha,s,\Delta t},\nu_k^\alpha)\leq \frac{\varepsilon}{2},$$
    {uniformly with respect to $k$.}
    Using this inequality and \cref{eq:fs_conv_corollary}, we have that for every $s\geq s_\varepsilon$ 
    $$W_2(\nu_k^\alpha,\boldsymbol\rho_\infty^\alpha)\leq W(\nu_k^\alpha,\varrho_{t_k}^{\alpha,s,\Delta t}) + W(\varrho_{t_k}^{\alpha,s,\Delta t},\boldsymbol\rho_\infty^\alpha)\leq \frac{\varepsilon}{2} + C_\alpha \gamma^k.$$
    Taking the limit $\varepsilon \to 0$, we obtain
    $$W_2(\nu_k^\alpha,\boldsymbol\rho_\infty^\alpha)\leq C_\alpha \gamma^k.$$
%     Therefore $W_2(\nu_k^\alpha,\boldsymbol\rho_\infty^\alpha)\leq\varepsilon$ for $k$ large enough, which implies the thesis.
% {\color{red} MASSIMO: Actually here it seems to me that the chain of inqualities above says that: for whichevery $\varepsilon>0$ we fix
% $$
% W_2(\nu_k^\alpha,\boldsymbol\rho_\infty^\alpha)\leq \frac{\varepsilon}{2} + C_\alpha \gamma^k.
% $$
% uniformly with respect to $k$. As $\varepsilon>0$ is arbitrary and it does appear only on the RHS, we conclude that acutally
% $$
% W_2(\nu_k^\alpha,\boldsymbol\rho_\infty^\alpha)\leq  C_\alpha \gamma^k
% $$
% for all $k$ (simply take the limit $\varepsilon \to 0$ on both sides). So, I think that in the statement of Corollary 4.2 we can be more explicit about the rate of convergence and be quantitative.}
\end{proof}
\begin{proof}[Second proof of \cref{cor:HS_convergence}]
    By the definition of $\delta$ and $\lambda$ and \cref{eq:T_alpha,eq:consensus-hopping}, we have that $x_k^{CH} = \mc{T}^\alpha(x_{k-1}^{CH})$. Fixing $R$ such that $x_0 \in B_R(x^*)$, we have by \cref{thm:invariant} that for $\alpha$ large enough, $\mc{T}^\alpha$ is a contraction on $B_{R}(x^*)$ with Lipschitz constant $L_\alpha \to 0$ as $\alpha\to+\infty$ and fixed point $\boldsymbol\mu_\infty^\alpha$, where $\boldsymbol\rho_\infty^\alpha = \mathcal N(\boldsymbol\mu_\infty^\alpha,\frac{\delta^2}{2\lambda}I)$.
\end{proof}

Corollary \ref{cor:HS_convergence} proves the convergence of the Consensus Hopping scheme \eqref{eq:consensus-hopping} to the invariant measure. Yet, for $\alpha>0$ already the first iteration  of \eqref{eq:consensus-hopping} computing the first consensus point could be already quite close to a global minimizer. In the following we clarify in which sense running multiple iterations allows to improve over the first approximation and makes it worth running the algorithm. We start with one more useful variation of the Laplace principle. 

\begin{lemma}\label{lemma:laplace3}
Let Assumption A2 hold and let us assume, without loss of generality, that $\underline f = 0.$ Given $r>0$ let us define $f_r\coloneqq\inf_{x\in B_r(x^*)}f(x).$ Then, for every $r\in(0,R_0]$, $0<q<f_{\infty}-f_r$, and every $v\in\mathbb S^{d-1}$ the following inequality holds
    \begin{equation}
        v^\top(\Xalpha(\rho) - x^*)
        \leq\frac{(q+ f_r)^\nu}{\eta} + \frac{e^{-\alpha q}}{\rho(B_r(x^*))}\mathbb E_\rho\left[\left(v^\top(x- x^*)\right)^+\right],
    \end{equation}
    where $y^+ = \max\{0,y\}$ denotes the ReLU function.
    In particular
    \begin{equation}\label{lap32}
        |\Xalpha(\rho) - x^*|
        \leq\frac{(q+ f_r)^\nu}{\eta} + \frac{e^{-\alpha q}}{\rho(B_r(x^*))}  \sup_{v\in\mathbb S^{d-1}}\mathbb E_\rho\left[\left(v^\top(x- x^*)\right)^+\right].
    \end{equation}
\end{lemma}
\begin{proof}
    The proof is analogous to that of \cite[Proposition 4.5]{B1-fornasier2021global}. Let us assume, without loss of generality, that $x^*=0$, and let us denote with $C_1 = \frac{(q+ f_r)^\nu}{\eta}.$ By definition of $\Xalpha(\rho)$ and $\eta^\alpha_\rho$ we have 
    \begin{equation}\label{vTxalpha1}
    \begin{aligned}
        v^\top\Xalpha(\rho)-C_1 &= \int_{\R^d} (v^\top x-C_1) \eta^\alpha_\rho(\diff x)  = \int_{B_{C_1}(0)} (v^\top x-C_1) \eta^\alpha_\rho(\diff x) + \int_{B_{C_1}(0)^C} (v^\top x-C_1) \eta^\alpha_\rho(\diff x)\\
        &\leq \int_{B_{C_1}(0)^C} (v^\top x)^+ \eta^\alpha_\rho(\diff x) = \int_{B_{C_1}(0)^C} (v^\top x)^+ \frac{e^{-\alpha f(x)}}{\int_{\R^d}e^{-\alpha f(y) } \rho(\diff y)}\rho(\diff x)\\
        &\leq \int_{B_{C_1}(0)^C}  \frac{(v^\top x)^+}{\int_{\R^d}e^{\alpha(f(x)-f(y)) } \rho(\diff y)}\rho(\diff x).
    \end{aligned}
    \end{equation}
    Moreover, for $x\in B_{C_1}(0)^C$, by assumption A2 and the definition of $C_1$, we have
    \begin{equation*}
        \begin{aligned}
          \int_{\R^d}e^{-\alpha f(y) } \rho(\diff y)  &= {\int_{\R^d}e^{\alpha(f(x)-f(y)) } \rho(\diff y)}\geq \int_{B_r(0)}e^{\alpha(f(x)-f_r)} \rho(\diff y) \\
          &\geq \rho(B_r(x^*))e^{\alpha\left( (\eta|x|)^{1/\nu}-f_r\right)}\geq \rho(B_r(x^*)) e^{\alpha q}.
        \end{aligned}
    \end{equation*}
    Using this inequality in \eqref{vTxalpha1} we get 
     \begin{equation}\label{vTxalpha2}
    \begin{aligned}
        v^\top\Xalpha(\rho)-C_1
        &\leq \frac{e^{-\alpha q}}{\rho(B_r(x^*))}\int_{B_{C_1}(0)^C}  (v^\top x)^+\rho(\diff x),
    \end{aligned}
    \end{equation}
which gives the first part of the thesis. 
Inequality \eqref{lap32} follows directly, noticing that for any vector $y\in\R^d$ it holds 
$|y| = \sup_{v\in\mathbb S^{d-1}}v^\top y.$
    
\end{proof}

The next result explains the improvement in the approximation of a global minimizer along the iterations of \eqref{eq:consensus-hopping}.

\begin{theorem}\label{thm:monotonicity}
     Let the same assumptions as Lemma \ref{lemma:laplace3} hold. Let $x_0 \in \RR^d$ be an arbitrary initial point, and let $\{x_k^{CH}\}_{k \in \NN}$ be the solution to the Consensus Hopping scheme \eqref{eq:consensus-hopping} with initial data $x_0$ and $\tilde{\sigma} = \frac{\delta}{\sqrt{2\lambda}}$. If $x_k^{CH}\in B_R(x^*)$, then, for any $r > 0$, we have that
    $$|x_{k+1}^{CH}-x^*|\leq C(\alpha,R,r) + \Gamma(\alpha, R,r)\cdot|x_k^{CH}-x^*|,$$
    where
    \begin{equation*}
        \begin{aligned}
            &C(\alpha, R,r) = \frac{(q+f_r)^\nu}{\eta} + \frac{\tilde\sigma(2\pi\tilde\sigma^2)^{d/2}}{\sqrt{2\pi}|B_r(x^*)|}\exp\left(\frac{R+r}{2\tilde\sigma^2}-\alpha q\right),\\
            & \Gamma(\alpha,R,r) = \frac{\tilde\sigma(2\pi\tilde\sigma^2)^{d/2}}{|B_r(x^*)|} \exp\left(\frac{R+r}{2\tilde\sigma^2}-\alpha q\right).
        \end{aligned}
    \end{equation*}
    In particular, $\Gamma(\alpha,R,r )<1$ if 
    $$\alpha>\frac{1}{q}\left(\frac{r+R}{2\tilde\sigma^2}+\frac{d}{2}\log(2\pi\tilde\sigma^2)-\log(|B_r(x^*)|)\right).$$
\end{theorem}
\begin{remark}
    The quantities $C(\alpha, R, r)$ and $\Gamma(\alpha, R, r)$ can be made arbitrarily small by making $\alpha$ large: For $\Gamma$, we have exponential decay in $\alpha$; for $C$, we may first choose $r$ and $q$ small to make the first summand small and then choose $\alpha$ large to make the second summand small. Thus, eventually, $C(\alpha, R, r) \ll 1$ and $\Gamma(\alpha, R, r) < 1$; in this regime, \cref{thm:monotonicity}, amounts to saying that up to a small error, the quantity $|x_{k}^{CH}-x^*|$ is monotonically decreasing in $k$. This ensures that running the Consensus Hopping scheme \eqref{eq:consensus-hopping} improves the approximation of the global minimizer $x^*$ step by step.
\end{remark}
\begin{proof}
For any $\mu\in\R^d$ and every $\sigma>0$ we have 
\begin{equation}\label{gaussian1}
    \begin{aligned}
        \sup_{v\in\mathbb S^{d-1}}\mathbb E_{\mathcal N\left(\mu, \sigma^2 I_d\right)}\left[(v^\top x)^+ \right] &= \mathbb E_{\mathcal N\left(\mu, \sigma^2 I_d\right)}\left[\left(\frac{1}{|\mu|}\mu^\top x\right)^+ \right] = \mathbb E_{\mathcal N\left(|\mu|, \sigma^2\right)}\left[(y)^+ \right]\\
        &= \sigma \varphi\left(\frac{|\mu|}{\sigma}\right)+ |\mu| \Phi\left(\frac{|\mu|}{\sigma}\right) \leq \frac{\sigma}{\sqrt{2\pi}}+|\mu|,
    \end{aligned}
\end{equation}
where we denoted with $\varphi$ and $\Phi$ the probability density function and the cumulative distribution function of the standard Normal distribution. 
Moreover for any $r>0$, $\mu\in B_R(x^*)$, 
\begin{equation}\label{gaussian2}
    \mathcal N\left(\mu, \sigma^2 I_d\right)(B_r(x^*))\geq \frac{|B_r(x^*)|}{(2\pi\sigma^2)^{d/2}}\exp\left(-\frac{r+R}{2\sigma^2}\right). 
\end{equation}

Applying Lemma \ref{lemma:laplace3} with $\rho = \mathcal N(x_k^{CH}, \tilde\sigma^2 I)$ and inequalities \eqref{gaussian1},\eqref{gaussian2} with $\mu = x_k^{CH}, \sigma = \tilde{\sigma}$, we get the thesis.
\end{proof}

\section{Quantitative mean-field limits and convergence of the particle schemes}\label{sec:particles}
The results of Sections~\ref{sec:delta_CBO}--\ref{sec:HS} concern the laws of the particles. Here we summarize the quantitative mean-field limits established for all three schemes in the Master thesis \cite{haug2026quantitative}, and we show how, following the blueprint of \cite[Theorem 3.8]{B1-fornasier2021global}, they turn our mean-field results into global convergence in probability of the particle schemes, in particular of \cref{alg:CF}.

\subsection{Quantitative mean-field limits}
Each particle system is coupled with $N$ i.i.d.\ copies $\{X^j\}_{j=1}^N$ of its mean-field process sharing the initial data, $X^j_0=\Xf^j_0\overset{\mbox{i.i.d}}{\sim}\rho_0$, and the noise: for \eqref{eq:cbo_delta_finite} and \eqref{eq:flip_scheme_finite}, $X^j$ solves \eqref{eq:cbo_delta_mf} and \eqref{eq:flip_scheme_2}, respectively, driven by $B^j$. For the Consensus Hopping scheme we consider
\begin{equation}\tag{CH$_N$}\label{eq:CH_N}
    \Xf^j_k=\Xalpha(\widehat\nu^{\alpha,N}_{k-1})+\eta^j_k,\qquad X^j_k=\Xalpha(\nu^{\alpha}_{k-1})+\eta^j_k,\qquad k\in\NN,
\end{equation}
with $\widehat\nu^{\alpha,N}_k=\frac1N\sum_{j=1}^N\delta_{\Xf^j_k}$, $\nu^\alpha_k=\law(X^j_k)$, $\nu^\alpha_0=\rho_0$, and $\eta^j_k\overset{\mbox{i.i.d}}{\sim}\Normal(0,\tilde\sigma^2I_d)$, $\tilde\sigma^2=\frac{\delta^2}{2\lambda}$. By induction, $\nu^\alpha_k=\Normal(x^{CH}_{k-1},\tilde\sigma^2I_d)$ and $\Xalpha(\nu^\alpha_k)=x^{CH}_k$ for $k\geq1$, where $\{x_k^{CH}\}$ solves \eqref{eq:consensus-hopping} with $x^{CH}_0=\Xalpha(\rho_0)$. The particle iterates $\Xalpha(\widehat\nu^{\alpha,N}_k)$ are the MPPI, i.e., $(1,N)$-ES, updates, and \eqref{eq:CH_N} is \cref{alg:CF} with $e^{-s\lambda\Delta t}$ replaced by $0$, namely its formal limit $s\to\infty$, mirroring \cref{thm:FStoHS} at the particle level.
\begin{theorem}[{\cite[Theorems 4.3, 4.4, 4.5]{haug2026quantitative}}]\label{thm:MFL}
Let $f$ satisfy \cref{assumptions}-\textbf{A1}, let $\rho_0\in\mathcal P_q(\RR^d)$ for some $q>2$, and let $T>0$, with $T=K\Delta t$ for \eqref{eq:flip_scheme_finite}. Then there exists $C_{\mathrm{PoC}}>0$, independent of $N$, such that, for \eqref{eq:cbo_delta_finite} and \eqref{eq:flip_scheme_finite},
\begin{equation}\label{eq:PoC_unrest}
    \max_{j\in[N]}\EE\Big[\sup_{t\in[0,T]}|\Xf^j_t-X^j_t|^2\Big]\leq C_{\mathrm{PoC}}\,N^{-r},\qquad r\coloneqq\min\Big\{1,\frac{q-2}{4}\Big\},\qquad\forall N\in\NN,
\end{equation}
and the same holds for \eqref{eq:CH_N} with $\sup_{t\in[0,T]}$ replaced by $\max_{0\leq k\leq K}$.
\end{theorem}
Thus, for $q\geq6$, \eqref{eq:PoC_unrest} yields the rate $N^{-1}$, i.e., propagation of chaos in $W_2$ at the rate $N^{-1/2}$, whereas for $q=4$ it only yields $N^{-1/2}$, see \cite[Table 4.1]{haug2026quantitative}.
While we do not include the proofs here, we sketch their main arguments, denoting by $\varrho_\tau$ the mean-field law of the respective scheme, by $\widehat\varrho^N_\tau$ and $\overline\varrho^N_\tau$ the empirical measures of $\{\Xf^j_\tau\}_j$ and $\{X^j_\tau\}_j$, and by $D_\tau\coloneqq\Xalpha(\widehat\varrho^N_\tau)-\Xalpha(\varrho_\tau)$ the consensus error.
\emph{(i) Synchronous coupling.} Since the noise is additive and shared by $\Xf^j$ and $X^j$, it cancels in their difference, which is driven by $D_\tau$ only; e.g., for \eqref{eq:cbo_delta_finite}, $\diff(\Xf^j_t-X^j_t)=-\lambda(\Xf^j_t-X^j_t)\diff t+\lambda D_t\diff t$.
\emph{(ii) Consensus error.} We split $|D_\tau|\leq|\Xalpha(\widehat\varrho^N_\tau)-\Xalpha(\overline\varrho^N_\tau)|+|\Xalpha(\overline\varrho^N_\tau)-\Xalpha(\varrho_\tau)|$. The first term is controlled by the stability estimate {$|\Xalpha(\nu)-\Xalpha(\mu)|\leq C_0W_2(\nu,\mu)$ whenever $\mu$ has bounded second moments {\cite[Corollary 3.3]{gerber_mean-field_2024}}}, which improves \eqref{eq:stability}, together with $W_2^2(\widehat\varrho^N_\tau,\overline\varrho^N_\tau)\leq\frac1N\sum_{j}|\Xf^j_\tau-X^j_\tau|^2$. This holds on the event where {$\overline\varrho^N_\tau$ has bounded second moments}; its complement is handled by H\"older's inequality, $q$-th moment bounds for particles and copies that are uniform in $N$ \cite[Chapter 3]{haug2026quantitative}, and the bound {$|\Xalpha(\nu)|^q\leq b_q\,\EE_\nu[|x|^q]$ for some $b_q>0$ {\cite[Corollary 2.26]{haug2026quantitative}}, cf.~{\cite[Proposition A.4]{gerber_mean-field_2024}}}, which produces the rate $N^{-(q-2)/4}$. The second term is a law of large numbers for the consensus point of i.i.d.\ samples, $\EE|\Xalpha(\overline\varrho^N_\tau)-\Xalpha(\varrho_\tau)|^2\leq CN^{-1}$ {{\cite[Lemma 4.2]{haug2026quantitative}}, cf.~{\cite[Lemma 3.7]{gerber_mean-field_2024}}}. Altogether, $\EE|D_\tau|^2\leq2C_0^2\,\EE|\Xf^j_\tau-X^j_\tau|^2+bN^{-r}$ for some $b>0$.
\emph{(iii) Closing the estimate.} For \eqref{eq:cbo_delta_finite}, Gr\"onwall's inequality yields \eqref{eq:PoC_unrest} with a constant growing exponentially in $T$. For \eqref{eq:flip_scheme_finite}, the consensus point is frozen on $[t_k,t_{k+1})$, and the explicit Ornstein--Uhlenbeck solution gives $\Xf^j_t-X^j_t=e^{-s\lambda(t-t_k)}(\Xf^j_{t_k}-X^j_{t_k})+(1-e^{-s\lambda(t-t_k)})D_{t_k}$, so that $e_k\coloneqq\EE|\Xf^j_{t_k}-X^j_{t_k}|^2$ satisfies $e_{k+1}\leq a\,e_k+bN^{-r}$, $e_0=0$, with $a=e^{-s\lambda\Delta t}+2(1-e^{-s\lambda\Delta t})C_0^2$. For \eqref{eq:CH_N}, the innovations $\eta^j_k$ cancel and $e_k\coloneqq\EE|\Xf^j_k-X^j_k|^2=\EE|D_{k-1}|^2\leq a\,e_{k-1}+bN^{-r}$ with $a=2C_0^2$. Iterating these recursions, and controlling the supremum over $t$, resp.\ $k$, yields $C_{\mathrm{PoC}}\propto\sum_{l=0}^{K-1}(K-l)a^l$ \cite[(4.28), Sections 4.3--4.4]{haug2026quantitative}, which grows exponentially in $K$ unless $a<1$; hence the constants are not uniform in time. Uniform-in-time versions, available for CBO \cite{huang2025uniform,gerber2025uniform,bayraktar2025uniform}, remain open for \eqref{eq:flip_scheme_2} and \eqref{eq:consensus-hopping}, see \cite[Section 4.7]{haug2026quantitative}.

% Thus, for $q\geq6$, \eqref{eq:PoC_unrest} yields the rate $N^{-1}$, i.e., propagation of chaos in $W_2$ at the rate $N^{-1/2}$, whereas for $q=4$ it only yields $N^{-1/2}$, see \cite[Table 4.1]{haug2026quantitative}. While we do not include here the proofs, we sketch their main arguments: first of all we combine synchronous coupling,  then we use the stability estimate $|\Xalpha(\mu)-\Xalpha(\nu)|\leq C_0W_2(\mu,\nu)$ for measures with bounded second moments \cite[Proposition 3.1]{gerber_mean-field_2024}, which improves \eqref{eq:stability}, and, for \eqref{eq:flip_scheme_finite}, the explicit Ornstein--Uhlenbeck solution on each freezing interval. \\
% The constants are not uniform in time: for \eqref{eq:cbo_delta_finite} they stem from a Gr\"onwall argument, while for \eqref{eq:flip_scheme_finite} and \eqref{eq:CH_N} they depend on the number of iterations $K$ through the factor $\sum_{l=0}^{K-1}(K-l)a^l$, with $a=2e^{-2s\lambda\Delta t}+4(1-e^{-s\lambda\Delta t})^2C_0^2$ and $a=2C_0^2$, respectively \cite[(4.28), Section 4.5]{haug2026quantitative}, which grows exponentially in $K$ unless $a<1$. Uniform-in-time versions, available for CBO \cite{huang2025uniform,gerber2025uniform,bayraktar2025uniform}, remain open for \eqref{eq:flip_scheme_2} and \eqref{eq:consensus-hopping}, see \cite[Section 4.7]{haug2026quantitative}.

\begin{remark}\label{rem:MFL_p}
The growth conditions in \cref{assumptions}-\textbf{A1} coincide, up to the choice of the radius $M$, with those of \cite[(A2)--(A3)]{haug2026quantitative} for the exponent $p=2$. More generally, \cite[Theorems 4.3, 4.4, 4.5]{haug2026quantitative} hold for any $p\geq2$ if the quadratic growth conditions in \cref{assumptions}-\textbf{A1} are replaced by growth of order $p$, i.e., $f(x)-\underline f\leq c_u(|x|^p+1)$ for all $x\in\RR^d$ and $c_\ell|x|^p\leq f(x)-\underline f$ for $|x|\geq M$: for $\rho_0\in\mathcal P_q(\RR^d)$ with $q>p$,
\begin{equation*}
\max_{j\in[N]}\EE\Big[\sup_{t\in[0,T]}|\Xf^j_t-X^j_t|^p\Big]\leq C_{\mathrm{PoC}}\,N^{-\min\{\frac p2,\frac{q-p}{2p}\}},
\end{equation*}
i.e., propagation of chaos in $W_p$, at the rate $N^{-1/2}$ if $q\geq p^2+p$. Further generalizations of the growth exponents are discussed in \cite[Section 2.1]{gerber_mean-field_2024}.
\end{remark}

% ===================== END OF NEW SECTION =====================

\section{Numerical Experiments}

The dynamics \eqref{eq:cbo_mf}  and \eqref{eq:CH}  analyzed in the previous sections have been implemented and widely tested against numerous benchmarks and in real-life applications. For instance, \eqref{eq:CH} appears under the name of MPPI in the robotic literature as in \cite{MPPI} and it was widely used afterwards. Similarly, \eqref{eq:cbo_mf} has been tested against numerous benchmarks in e.g., \cite{B1-borghi2021constrained,BORGHI2023113859,pinnau2017consensus,totzeck2020consensus}.
These implementations often come with noise adaptation (e.g., anisotropic, covariance adapted etc.) to cope with the curse of dimensionality with great success. The reader is therefore referred to such more experimental literature for comparisons and benchmarking. For instance, we mention the recent paper on robotics \cite{sun2026CBOrob} where \eqref{eq:cbo_mf} and \eqref{eq:CH} (aka MPPI) are directly compared.\\
Consensus Freezing \eqref{eq:flip_scheme_2} has been introduced first in this present manuscript and it has never been tested in the literature against benchmarks. In order to compare it with the other methods on benchmarks, it would be necessary to equip it with appropriate noise adaptations, which were not analyzed in this work. Hence, we do not provide this systematic comparison here. \\
In this section we present the results of a set of numerical experiments to support the theoretical analysis carried out in the previous sections, and study what happens in practice when the considered methods are implemented in the finite-particle and discrete-time regime. Before presenting the experiments, in the following subsection we describe how the Consensus Freezing scheme is implemented, as the fundamental bridge between (CBO) and (CH). 
\subsection{Implementation of the Consensus Freezing scheme}\label{sec:impCF}
We consider a system of $N$ particles, whose initial positions are sampled from a given distribution $\varrho_0$. In the finite-particles regime, given $s,\lambda,\delta,\Delta t, \alpha$, each particle evolves in time according to the SDE \eqref{eq:flip_scheme_finite}. Given the position of the particles $\{\Xf^j_{t_i}\}_{j=1}^N$ and the consensus point $X_{t_i}^\alpha = \Xalpha(\widehat\varrho^{\alpha,s,\Delta t,N}_{t_i})$, we know that from  Section \ref{sec:gcCF}   the conditional distribution of particle $j$ at time $t_{i+1}$ is given by $\mathcal N\left(\widehat\mu^j_{t_{i+1}}, \widehat \Sigma^j_{t_{i+1}}\right)$ where
\begin{equation}\label{iterative}
    \begin{cases}
        \widehat\mu^j_{t_{i+1}}=(1-e^{-s\lambda\Delta t})X^\alpha_{t_i}+e^{-s\lambda\Delta t}\Xf^j_{t_i}\\
        \widehat\Sigma_{t_{i+1}}^j = (1-e^{-2s\lambda\Delta t})\frac{\delta^2}{2\lambda}I_d.
    \end{cases}
\end{equation}
Given the method parameters, a timestep $\Delta t$ and the initial distribution $\varrho_0$, we generate the sequence of particles $\{\{x^j_k\}_{j=1}^N\}_{k\in\mathbb N}$ as follows.
\begin{algorithm}
\caption{Consensus Freezing}\label{alg:CF}
\begin{algorithmic}[1]
\REQUIRE $N\in\mathbb N$, $\alpha,\lambda,\delta,s,\Delta t>0, \varrho_0\in\mathcal{P}(\mathbb{R}^d)$
\STATE sample $\{x^j_0\}_{j=1}^N$ independently from $\varrho_0.$ 
\FOR{$k = 0, 1,2,\dots$}
    \STATE define  $\widehat \varrho_k^N = \frac{1}{N}\sum_{j=1}^N\delta_{x^j_k}$
    \STATE compute $x^\alpha_k = X^\alpha(\widehat \varrho_k^N)$
    \FOR{$j = 1,\dots,N$}
        \STATE compute $\widehat\mu^j_{k+1}=(1-e^{-s\lambda\Delta t})x^\alpha_k+e^{-s\lambda\Delta t}x^j_k$
        \STATE compute $\widehat\Sigma_{k+1}^j = (1-e^{-2s\lambda\Delta t})\frac{\delta^2}{2\lambda}I_d$
        \STATE sample $x^j_{k+1}$ from $\mathcal N\left(\widehat \mu^j_{k+1},\widehat\Sigma^j_{k+1}\right)$
    \ENDFOR
\ENDFOR
\end{algorithmic}
\end{algorithm}

% \begin{equation}
%     \begin{aligned}
%         \sup_{v\in\mathbb S^{d-1}}\mathbb E_{\nu_k^\alpha}\left[(v^\top x)^+ \right] = \mathbb E_{\nu_k^\alpha}\left[\left(\frac{1}{|\mu|}\mu^\top x\right)^+ \right] = \mathbb E_{\mathcal N\left(|x_k^{CH}|, \sigma^2\right)}\left[(y)^+ \right].
%     \end{aligned}
% \end{equation}
% Using tis inequality with $ $

{
\subsection{Convergence of \cref{alg:CF}}
The decisive feature of the Consensus Freezing scheme is that its implementation is \emph{exact}.
\begin{lemma}\label{lem:exactCF}
Let $\{x^j_k\}$ be generated by \cref{alg:CF} and let $\{\Xf^j\}$ solve \eqref{eq:flip_scheme_finite} with $\Xf^j_0\overset{\mbox{i.i.d}}{\sim}\varrho_0$. Then $(\{x_k^j\}_{j=1}^N)_{k\geq0}$ and $(\{\Xf^j_{t_k}\}_{j=1}^N)_{k\geq0}$ have the same law.
\end{lemma}
\begin{proof}
On $[t_k,t_{k+1})$ the consensus point $X^\alpha_{t_k}=\Xalpha(\widehat\varrho^{\alpha,s,\Delta t,N}_{t_k})$ is frozen and $\mathcal F_{t_k}$-measurable, where $\mathcal F_t\coloneqq\sigma(\Xf^j_0,B^j_\tau:\tau\leq t,\,j\in[N])$, so each particle is an Ornstein--Uhlenbeck process and, with $\omega\coloneqq e^{-s\lambda\Delta t}$,
\begin{equation*}
\Xf^j_{t_{k+1}}=\omega\,\Xf^j_{t_k}+(1-\omega)X^\alpha_{t_k}+\sqrt s\,\delta\int_{t_k}^{t_{k+1}}e^{-s\lambda(t_{k+1}-\tau)}\diff B^j_\tau ,
\end{equation*}
where the stochastic integrals are, for $j\in[N]$, i.i.d.\ $\Normal\big(0,\frac{1-\omega^2}{2s\lambda}I_d\big)$ and independent of $\mathcal F_{t_k}$. Hence $(\{\Xf^j_{t_k}\}_{j})_{k}$ is a Markov chain with initial law $\varrho_0^{\otimes N}$ and transition kernel given by lines 3--8 of \cref{alg:CF}, cf.~\eqref{iterative}.
\end{proof}
\begin{theorem}\label{thm:alg_CF_conv}
Let $f$ satisfy the assumptions of \cref{thm:MFL}, let $\varrho_0\in\mathcal P_q(\RR^d)$ with $q\geq4$, $s,\Delta t>0$, $\gamma\in(e^{-s\lambda\Delta t},1)$, and let $\{x^j_k\}$ be generated by \cref{alg:CF}. Then there is $\alpha_0>0$ such that for all $\alpha\geq\alpha_0$, $K,N\in\NN$ 
\begin{equation}\label{eq:alg_CF_conv_exp}
\mathbb E \Big|\frac1N\sum_{j=1}^Nx^j_K-x^*\Big|^2 \leq 3 \Big(C_{\mathrm{PoC}}(K)\,N^{-\min\{1,\frac{q-2}{4}\}}+S\,N^{-1}+\big(C_\alpha\gamma^K+\varepsilon_\alpha\big)^2\Big),
\end{equation}
and for any $\varepsilon>0$
\begin{equation}\label{eq:alg_CF_conv}
    \mathbb P\bigg(\Big|\frac1N\sum_{j=1}^Nx^j_K-x^*\Big|^2\geq\varepsilon\bigg)\leq\frac3\varepsilon\Big(C_{\mathrm{PoC}}(K)\,N^{-\min\{1,\frac{q-2}{4}\}}+S\,N^{-1}+\big(C_\alpha\gamma^K+\varepsilon_\alpha\big)^2\Big),
\end{equation}
where $C_{\mathrm{PoC}}(K)$ is the constant in \eqref{eq:PoC_unrest} for \eqref{eq:flip_scheme_finite} with $T=K\Delta t$, $S$ and $C_\alpha$ are the constants of \cref{thm:global_bound,thm:fsconv}, and $\varepsilon_\alpha\coloneqq|\boldsymbol\mu^\alpha_\infty-x^*|=W_2(\boldsymbol\rho_\infty^\alpha,\boldsymbol\rho_\infty)\to0$ as $\alpha\to\infty$ by \cref{thm:invariant}.
\end{theorem}
\begin{proof}
By \cref{lem:exactCF} we may replace $x^j_K$ by $\Xf^j_{t_K}$. With the coupled copies $\{X^j\}$ of \eqref{eq:flip_scheme_2}, let $\widehat m\coloneqq\frac1N\sum_j\Xf^j_{t_K}$, $\overline m\coloneqq\frac1N\sum_jX^j_{t_K}$ and $\mu_K\coloneqq\EE[X^1_{t_K}]$. Then
\begin{equation}\label{eq:triangle}
    \EE|\widehat m-x^*|^2\leq3\,\EE|\widehat m-\overline m|^2+3\,\EE|\overline m-\mu_K|^2+3\,|\mu_K-x^*|^2.
\end{equation}
The first (mean-field) term is bounded, by Jensen's inequality and \eqref{eq:PoC_unrest}, by $\frac1N\sum_j\EE|\Xf^j_{t_K}-X^j_{t_K}|^2\leq C_{\mathrm{PoC}}(K)N^{-\min\{1,(q-2)/4\}}$. The second (Monte Carlo) term satisfies, by independence and \cref{thm:global_bound}, $\EE|\overline m-\mu_K|^2\leq\frac1N\EE_{\varrho^{\alpha,s,\Delta t}_{t_K}}[|x|^2]\leq\frac SN$. For the third (optimization) term, since $\boldsymbol\rho_\infty$ has mean $x^*$, Jensen's inequality, \cref{thm:fsconv} and \eqref{eq:Wasserstein_between_gaussians} give $|\mu_K-x^*|\leq W_2(\varrho^{\alpha,s,\Delta t}_{t_K},\boldsymbol\rho_\infty)\leq W_2(\varrho^{\alpha,s,\Delta t}_{t_K},\boldsymbol\rho^\alpha_\infty)+W_2(\boldsymbol\rho^\alpha_\infty,\boldsymbol\rho_\infty)\leq C_\alpha\gamma^K+\varepsilon_\alpha$. Markov's inequality concludes.
\end{proof}
\begin{remark}\label{rem:particles}
(a) Differently from \eqref{eq:convprob}, no time-discretization error enters \eqref{eq:alg_CF_conv}, which holds for \emph{any} $\Delta t,s>0$. Since any $\gamma\in(e^{-s\lambda\Delta t},1)$ is admissible, with $C_\alpha$ uniform in $s\geq s^\star$ (\cref{thm:fsconv}), large $s\lambda\Delta t$ allow small $\gamma$, so that the optimization error is small after few iterations $K$, which limits the growth of $C_{\mathrm{PoC}}(K)$; this is in agreement with \cref{fig:compare_performance_deltat}. As \eqref{eq:convprob}, the bound is meaningful at finite horizon only and encodes the hardness of the problem in the dependence of $C_{\mathrm{PoC}}(K)$ on $\alpha$, cf.~\cite{B1-fornasier2021global}. For $q\geq6$ the rate is $N^{-1}$, as in \eqref{eq:convprob}. 
\\
(b) The same argument applies to the output $x^\alpha_K$ of line 4 of \cref{alg:CF}: in $|x^\alpha_K-x^*|\leq|\Xalpha(\widehat\varrho^{N}_K)-\Xalpha(\varrho^{\alpha,s,\Delta t}_{t_K})|+|\Xalpha(\varrho^{\alpha,s,\Delta t}_{t_K})-\Xalpha(\boldsymbol\rho^\alpha_\infty)|+\varepsilon_\alpha$ the first term is bounded in mean square by $CN^{-\min\{1,(q-2)/4\}}$ \cite[Chapter 4, (4.1)--(4.3)]{haug2026quantitative}, the second by $L_XC_\alpha\gamma^K$ by \eqref{eq:stability} and \cref{thm:global_bound}, since $\Xalpha(\boldsymbol\rho^\alpha_\infty)=\boldsymbol\mu^\alpha_\infty$.\\
(c) For \eqref{eq:CH_N}, \cref{thm:MFL} and \cref{cor:HS_convergence} give a bound of the form \eqref{eq:alg_CF_conv} for $\frac1N\sum_j\Xf^j_K$, with $\gamma^{K-1}$ in place of $\gamma^K$ and any $\gamma\in(0,1)$ \cite[Theorem 4.12]{haug2026quantitative}, and the analogue of (b) for the MPPI output $\Xalpha(\widehat\nu^{\alpha,N}_K)$, as $|\Xalpha(\nu^\alpha_K)-\boldsymbol\mu^\alpha_\infty|=|x^{CH}_K-\boldsymbol\mu^\alpha_\infty|\leq C_\alpha\gamma^K$. For $\delta$-CBO, \cref{thm:MFL} and \cref{thm:globconv_deltaCBO} give the analogue for the time-continuous system \eqref{eq:cbo_delta_finite}, but any implementation, e.g., \eqref{eq:delta-cbo_finite_discrete}, adds a discretization error as in \eqref{eq:convprob}, which deteriorates for large $\Delta t$ (\cref{fig:two_images}); this is precisely what \cref{alg:CF} is meant to avoid.
\end{remark}
}

\subsection{Performance for different values of $\Delta t$}

%$\delta$-CBO vs Consensus Freezing}
We compare the number of iterations required for convergence of $\delta$-CBO with the number of steps required for convergence of Consensus Freezing, for different values of the discretization time-step $\Delta t.$ 
We consider the Ackley function 
\begin{equation}\label{eq:Ackley}
    f(x) = -a \exp\left(-\frac{b}{\sqrt{d}}\|x\|_2\right)
- \exp\left(\frac{1}{d} \sum_{i=1}^d \cos(c x_i)\right)
+ a + e,
\end{equation}
with $a = 20,\ b=0.2,\ c = 2\pi.$ 
The minimizer of $f$ is $x^*=0$. 
For the $\delta$-CBO method, we apply the Euler-Maruyama discretization of the finite-particle dynamics \eqref{eq:delta-cbo_finite_discrete}; the Consensus Freezing we implement as described in Section \ref{sec:impCF}. We apply both methods to the minimization of \eqref{eq:Ackley} with $d=5.$

For both methods, we set $N = 5000,\ \alpha=1.e^{15},\ \lambda = 1,\ \delta = 1.41$ and we consider several values of the discretization time-step $\Delta t\in [10^{-2},10^2].$ 
The initial positions of the particles are drawn independently from the uniform distribution in the box $B=[5,7]^d$. Notice that $x^*\notin B$. Each method terminates when the consensus point is in an $\varepsilon$-neighborhood of the solution $x^*$ for $m$ consecutive iterations, with $\varepsilon = 0.1 $ and $m = 10$.
To avoid giving an advantage to the Consensus Freezing scheme, we choose the value of $\delta$ by parameter tuning for the $\delta$-CBO scheme: that is, with fixed $\Delta t=0.1$ and all other parameters as above, we run the $\delta$-CBO scheme with different values of $\delta\in (0,2]$ and selected the one that requires the smallest number of iterations to terminate. The results of this tuning experiment are reported in \cref{fig:deltaCBO_delta}.
\begin{figure}[ht]
\centering
    \includegraphics[width = 0.75\textwidth]{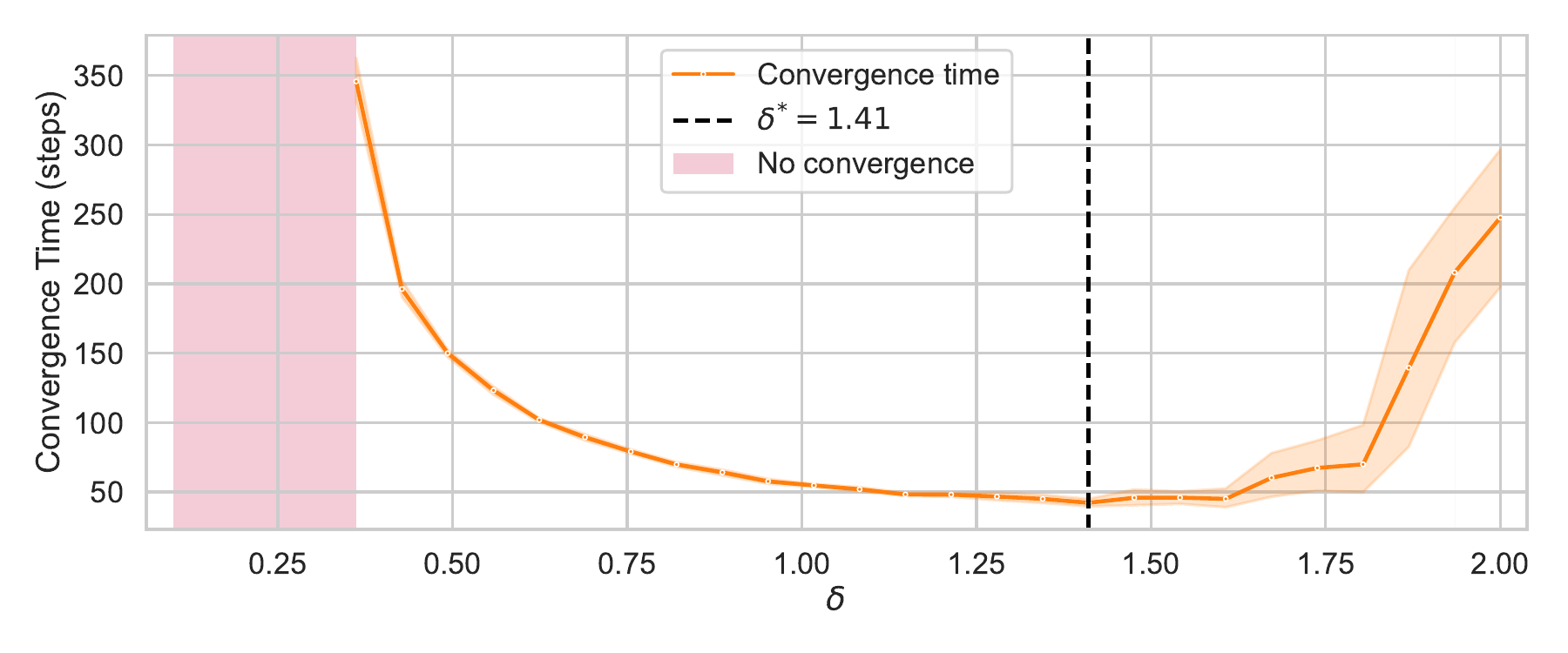}

    \caption{Number of iterations to arrive at termination for $\delta$-CBO with different values of $\delta.$ }
    \label{fig:deltaCBO_delta}
\end{figure}
In \cref{fig:deltaCBOvsCF_steps_large}, for each value of $\Delta t$, we report the number of iterations employed by each method to arrive at termination, averaged over 10 independent runs. We say that a method does not convergence for a given $\Delta t$ if, for one of the 10 independent runs, the termination condition is not satisfied after 500 iterations .

From \cref{fig:deltaCBOvsCF_steps_large} we can see that for small values of the discretization step $\Delta t$, the behavior of the two methods is similar and, initially, increasing $\Delta t$ reduces the number of iterations required by the two methods to arrive at termination. However, as $\Delta t$ approaches 1, the performance of $\delta$-CBO  deteriorates, and the method stops converging for larger values of the stepsize. On the other hand, the Consensus Freezing scheme is not affected by this issue: convergence is retained for all considered values of $\Delta t,$ and the number of iterations to termination decreases as $\Delta t$ increases.
\begin{figure}[ht]
\begin{subfigure}[b]{0.47\textwidth}
    \includegraphics[width = \textwidth]{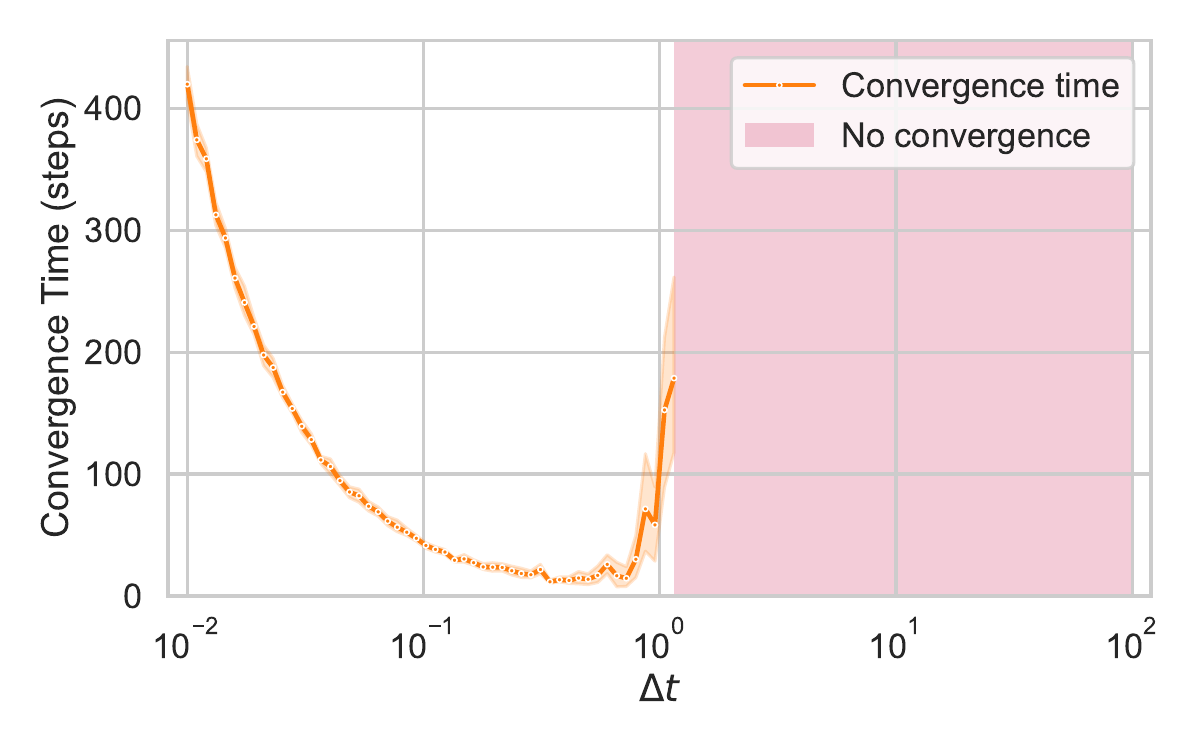}
\end{subfigure}
    \begin{subfigure}{0.47\textwidth}
        \includegraphics[width = \textwidth]{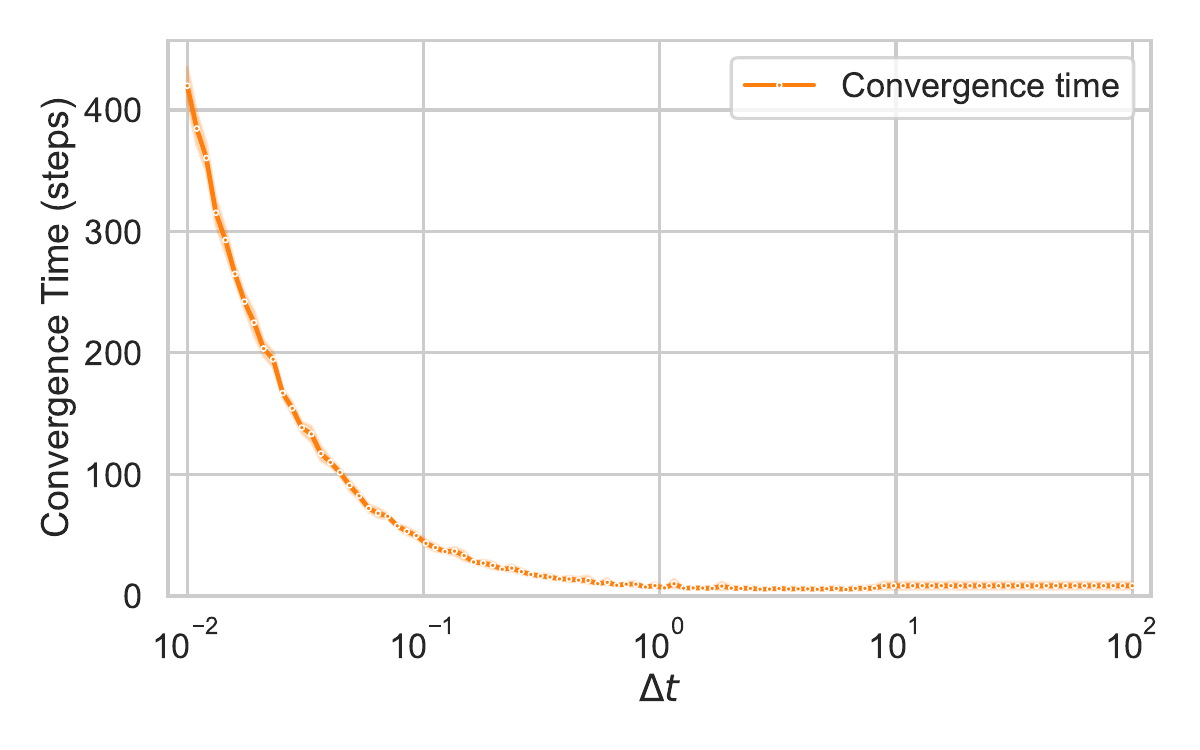}
    \end{subfigure}

    \caption{Number of iterations to arrive at termination for $\delta$-CBO (left) and Consensus Freezing (right), for different values of $\Delta t.$}
    \label{fig:deltaCBOvsCF_steps_large}
\end{figure}
\subsection{Faithfulness of the implementations for different values of $\Delta t$}
We consider a numerical implementation for $\Delta t>0$ faithful if it reproduces correctly the variance of the particles of the time-continuous models.
Therefore we use the empirical variance as an indirect measure of how well, for a chosen time discretization $\Delta t$, the Euler-Maruyama implementation of the $\delta$-CBO dynamics (or the exact implementation of the Consensus Freezing dynamics) approximates the corresponding finite-particle continuous-time dynamics \cref{eq:cbo_delta_finite} (\cref{eq:flip_scheme_finite}, respectively).

To this end, we derive an expression of the unbiased variance estimator to be for the finite-particle continuous-time dynamics, using a variant of the arguments from the proof of \cref{lem:explicit_solution}.
\begin{lemma}
    Let $\rho_0 \in \mathcal{P}_2(\R^d)$ be an initial distribution, $N \in \NN$, and let $\Xf_t$ be a solution to the SDE given in \cref{eq:cbo_delta_finite} or \cref{eq:flip_scheme_finite} for $N$ particles and with $\Xf_0^1, \dots \Xf_0^n$ i.i.d. with $\Xf_0^i\sim\rho_0$ , and let $\underline{\operatorname{Var}_N}$ be the unbiased variance estimator
    \begin{equation*}
        \underline{\operatorname{Var}_N}(X_1, \dots, X_N) = \frac{1}{N-1} \sum_{i=1}^{N} \left(X_i - \frac{1}{N}\sum_{j=1}^{N}X_j\right)^2.
    \end{equation*}
    Then we have that
    % sigma^2 / (2 * lambda) * (1 - exp(-2 * lambda * n_iterations * dt))
    \begin{equation*}
        \EE\big(\underline{\operatorname{Var}_N}(\Xf_t)\big) = \operatorname{Var}(\rho_0) \cdot \omega_t^2 + \frac{\sigma^2}{2\lambda} (1 - \omega_t^2),
    \end{equation*}
    where $\omega_t \coloneqq e^{-\lambda t}$.
\end{lemma}
\begin{proof}
    We prove the lemma for the $\delta$-CBO scheme; the proof for the Consensus Freezing scheme works analogously. First, notice that $\underline{\operatorname{Var}_N}$ is translation-invariant, i.e., for $x_1, \dots, x_N \in \R$ and $z \in \R$, we have that $\underline{\operatorname{Var}_N}(x_1, \dots, x_N) = \underline{\operatorname{Var}_N}(x_1 + z, \dots, x_N + z)$. Thus, also for a random variable $Z$ and random variables $X_1, \dots, X_N$, we have that $\EE(\underline{\operatorname{Var}_N}(X_1, \dots, X_N)) = \EE(\underline{\operatorname{Var}_N}(X_1 + Z, \dots, X_N + Z)).$\\
    Given a solution $\Xf_t$ of \cref{eq:cbo_delta_finite}, we denote by $Z_t$ a solution to the SDE
    \begin{equation*}
        dZ = -\lambda \big(Z - \Xalpha(\widehat\rho^N_t)\big),\quad  Z_0 = 0;
    \end{equation*}
    note that in the setting of the finite-particle dynamics, $\Xalpha(\widehat\rho^N_t)\big)$ is a random process. We see that the processes $Y^j_t \coloneqq \Xf^j_t - Z_t$ are a solution to the SDE
    \begin{equation}\label{eq:experiments_proof_eq}
    \begin{aligned}
        d(Y^j_t) &= -\lambda \big(\Xf_t^j -  \Xalpha(\widehat\rho^N_t) - Z + \Xalpha(\widehat\rho^N_t)\big) + \delta \diff B_t^j& \\
        &= -\lambda \cdot Y^j_t + \delta \diff B_t^j, &(\Xf^j - Z)_0 = \Xf^j_0
    \end{aligned}
    \end{equation}
    Thus, $Y_t$ follows an Ornstein–Uhlenbeck process and hence is distributed (by the same reasoning as in \cref{lem:explicit_solution}) like 
    $\mc{N}(0, \Sigma_t) \ast \Big((s_{k_t,c_\gamma(t)})_\# \rho_0\Big)$, where $\Sigma_t = \frac{\sigma^2}{2\lambda} (1-\omega_t^2)$, $k_t =\omega_t$ and $c_{\gamma}(t)$ is defined as in \cref{lem:explicit_solution}. Thus we have, for each $j$ and $t > 0$ that $\operatorname{Var}(Y_t^j) = \operatorname{Var}(\rho_0) \cdot \omega_t^2 + \frac{\sigma^2}{2\lambda} (1 - \omega_t^2).$ Finally, $Y_0^1, \dots, Y_0^N$ are independent, and we see from \cref{eq:experiments_proof_eq} that they are uncoupled and so also $Y_t^0, \dots, Y_t^N$ are i.i.d. Thus, we have
    \begin{equation*}
        \EE\big(\underline{\operatorname{Var}_N}(\Xf_t)\big) = \EE\big(\underline{\operatorname{Var}_N}(Y_t)\big) = \operatorname{Var}(Y_j^0) = \operatorname{Var}(\rho_0) \cdot \omega_t^2 + \frac{\sigma^2}{2\lambda} (1 - \omega_t^2),
    \end{equation*}
    finishing the proof.
\end{proof}
%In Lemma \ref{lem:explicit_solution} we derived and explicit expression for the solution of the $\delta$-CBO scheme and the Consensus Freezing scheme at a given time $t$ and for given initial distributions.
To show how the numerical methods approximate the time-continuous models for different values of $\Delta t$,  we compare the result of the unbiased variance estimator on the particles after $\bar k=500$ iterations with its expectation given in the previous lemma, at time $t=\bar k\Delta t$.
In the resulting \cref{fig:deltaCBOvsCF_variance_large}, we see that the estimated variance of the empirical distribution for the Consensus Freezing scheme, implemented as in Section \ref{sec:impCF} is close to its expected theoretical value for all considered values of $\Delta t$, indicating that the numerical implementation models well the behavior of the continuous-time finite-particle dynamics \cref{eq:flip_scheme_finite}. On the other hand, from the leftmost plot we can observe that the estimated variance of the particles for the $\delta$-CBO, implemented as an Euler-Maruyama discretization, is close to the theoretically expected value only for small values of $\Delta t$ and that the difference increases rapidly as $\Delta t$ approaches 1. That is, the discretized $\delta$-CBO scheme is a good approximation of the continuous-time dynamics \cref{eq:cbo_delta_finite} only for very small values of the discretization step $\Delta t$.
\begin{figure}[ht]
\begin{subfigure}[b]{0.47\textwidth}
    \includegraphics[width = \textwidth]{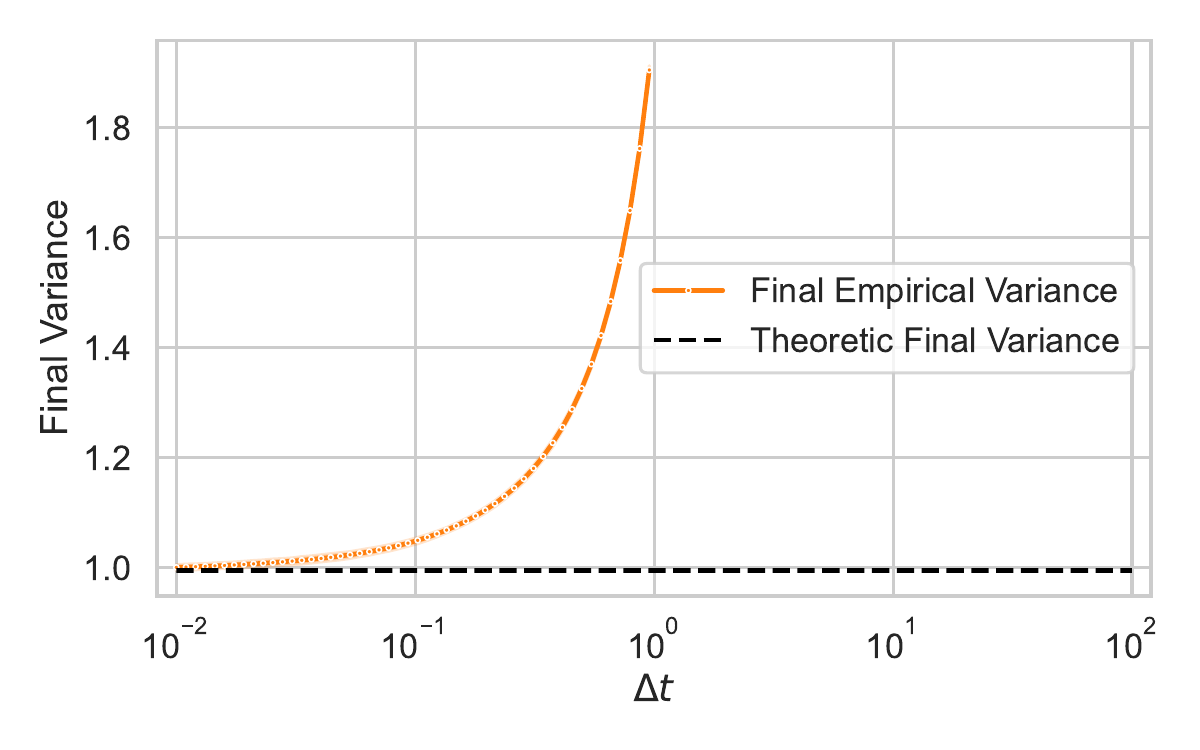}
\end{subfigure}
    \begin{subfigure}{0.47\textwidth}
        \includegraphics[width = \textwidth]{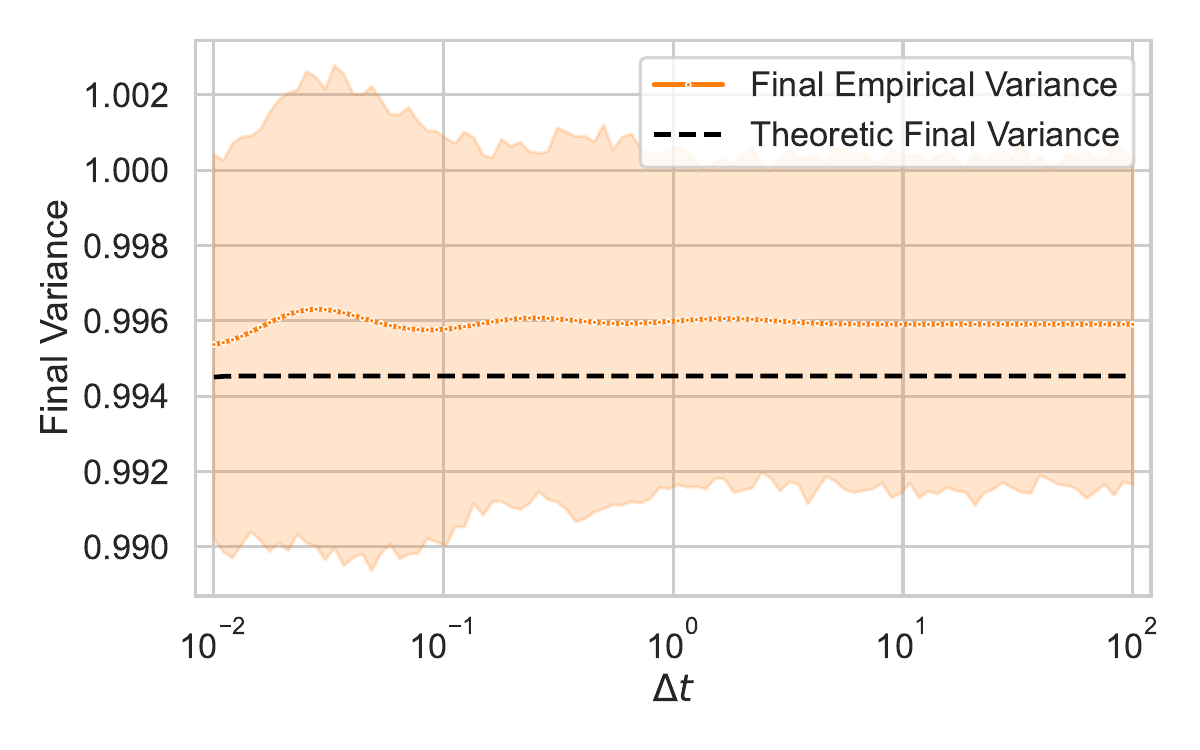    }

    \end{subfigure}

    \caption{Variance of the distribution at termination time for $\delta$-CBO (left) and Consensus Freezing (right), for different values of $\Delta t$, compared with the variance of the solution of the corresponding continuous-time models.}
    \label{fig:deltaCBOvsCF_variance_large}
\end{figure}

\subsection{Convergence of Consensus Freezing to the Consensus Hopping scheme}
In Section \ref{sec:HS} we proved that, in the mean field law, the Consensus Freezing scheme converges to the Consensus Hopping Scheme as $s$ tends to infinity. 
Here, we numerically study the relation between the Consensus Hopping scheme and the Consensus Freezing method in the finite-particles, discrete-time case, implemented as described in Section \ref{sec:impCF}.
The considered problem is the Ackley function \eqref{eq:Ackley} in dimension $d=2$. In both methods we set $N=100$ and $\alpha = 10^{15}.$
For the Consensus Hopping scheme \eqref{eq:consensus-hopping} we set $\sigma = \sqrt{0.5}$ and initial guess $x_0 = (5,5)^\top$. For Consensus Freezing we set $\lambda=1,\ \delta = 1$ and $\Delta t=0.1$ and we consider several values of $ s\in[10^{-1},10^3].$
We run each method for $\bar k =100$ iterations. 
Given we denote with $\widehat\nu_k^\alpha$ the empirical average distribution of the particles for the Consensus Hopping scheme at iteration $k$. For each considered value of $s$ and each iteration $k=1,\dots,\bar k$, we compute the empirical average distribution $\widehat\varrho^{\alpha,s,\Delta t}_k$ and compute the $2$-Wasserstein distance between $\widehat\nu_k^\alpha$ and $\widehat\varrho^{\alpha,s,\Delta t}_k$. In Figure \ref{fig:CFtoHS}, for each value of $s$, we plot the total Wasserstein distance over the $\bar k$ iterations. We observe that the overall distance between the two empirical distribution decreases as $s$ increases, and rapidly approaches 0 for $s$ large enough, reflecting the convergence analysis that we carried out in Section \ref{sec:HS}.
\begin{figure}
\centering
    \includegraphics[width = 0.75\textwidth]{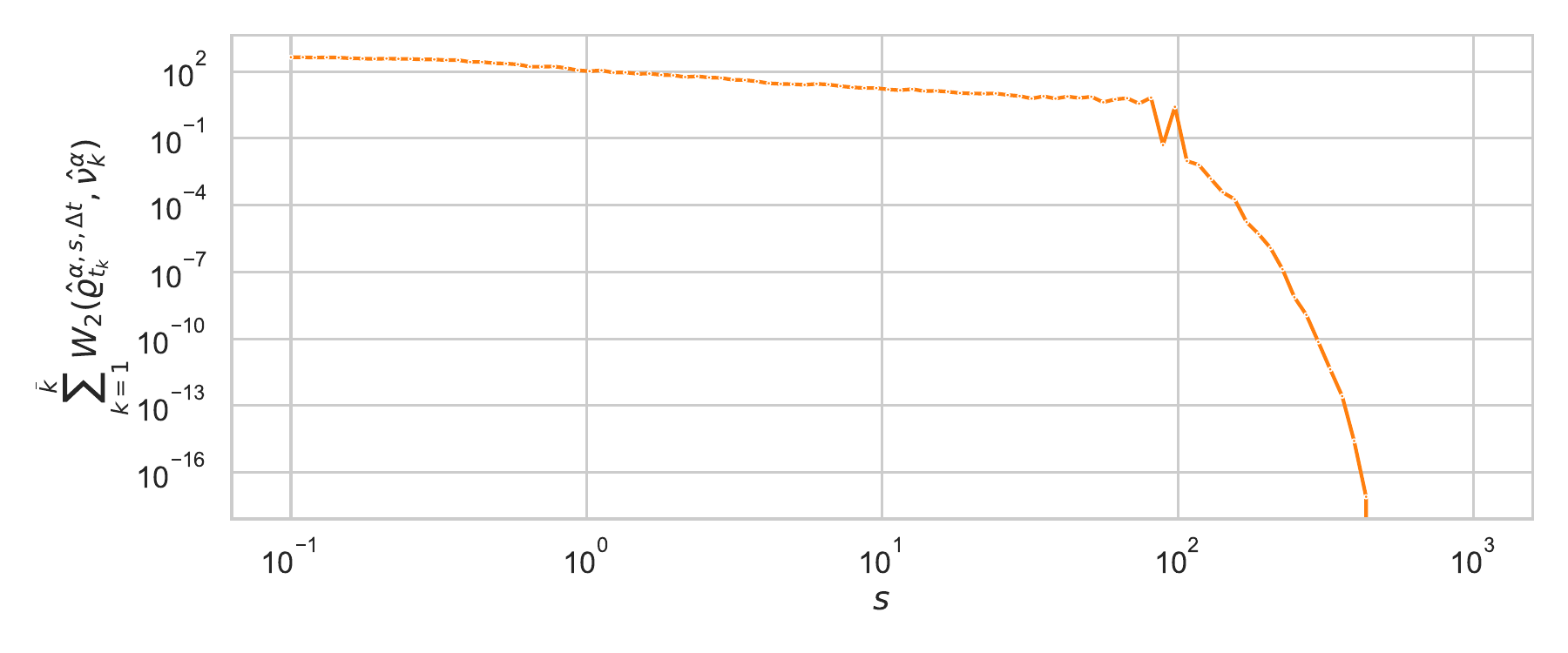}
    \caption{Convergence of Consensus Freezing to Consensus Hopping Scheme}
    \label{fig:CFtoHS}
\end{figure}

\section{Conclusion and outlook}

{We analyzed a class of zero-order global optimization methods, namely CBO, CF, and CH, connected by quantitative asymptotics, for which we provided complete descriptions and guaranteed convergence rates to global minimizers for {a rather large class} of nonsmooth and nonconvex objectives $f$.

These are likely not the final words in global optimization, as we considered here exclusively isotropic noise, which is, {for the hardest nonconvex problems}, inherently affected by the curse of dimensionality. CBO has been formulated with various types of noise: isotropic, anisotropic, and jump diffusion \cite{pinnau2017consensus,B1-carrillo2019consensus,kalise2022consensus}. %, with anisotropic and jump diffusion noise showing superior performance in higher dimensions.
However, without an adaptive covariance matrix for the noise,  multiparticle algorithms are known to struggle with efficiently solving non-separable, ill-conditioned problems. This limitation is well-documented in the global optimization literature, e.g., in the comparison of Particle Swarm Optimization  to Covariance Matrix Adaptation Evolution Strategy (CMA-ES) \cite{HANSEN20115755}, {see also \cite{bonandin2026exploiting,fiedler2026linearly}. In the future we will address this issue by leveraging recent advancements in covariance-modulated optimal transport \cite{burger2023covariancemodulated,B1-carrillo2022sampling}.
We believe that, with such adaptations, the models and theoretical tools developed
in this paper offer a principled basis for tackling continuous nonconvex and nonsmooth
problems with provable global guarantees. Many questions nonetheless remain open, both
theoretical and practical: sharpening the dependence of the bounds on the dimension and
on the problem constants, identifying further structural classes that can be exploited,
and developing efficient implementations that leverage the natural potential of these
methods for large-scale parallelization. We hope this work will encourage the community
to pursue these directions and to test such methods systematically on real-world problems.}

\section{Appendix}

\begin{proof}[Proof of Lemma \ref{lemma:ball_measure}]
    % The proof is basically the same as that of Proposition 4.6 in \url{https://arxiv.org/pdf/2103.15130}. The only difference is that at the beginning 
    % $$T_2 = \frac{\delta^2}{2}\Delta\phi_r(x)$$
    % $$K_2 = \{x\ | \ -\lambda(x-\Xalpha(\rho_t))^\top(x-x^*)(r^2- |x-x^*|^2)^2>\tilde c r^2 \frac{\delta^2}{2} |x-x^*|^2\}.$$
Since $\phi_r(x)\in[0,1]$ and $\phi_r(x)=0$ whenever $ |x-x^*|\geq r$ we have that 
$$\rho^\alpha_t(B_r(x^*)) \geq \int_{\mathbb{R}^d}\phi_r(x) \diff\rho^\alpha_t(x).$$
Let us denote with $T_1(x)$ and $T_2(x)$ the following two quantities
\begin{equation*}
    \begin{aligned}
        &T_1(x) = -\lambda(x-\Xalpha(\rho^\alpha_t))^\top\nabla\phi_r(x)\\
        &T_2(x) = \frac{\delta^2}{2}\Delta\phi_r(x).
    \end{aligned}
\end{equation*}
By \eqref{eq:FP_CBO_delta} we have 
\begin{equation}\label{eq:T1T2}
    \frac{\diff}{\diff t}\int\phi_r(x)\diff\rho^\alpha_t(x) = \int\partial_t\rho^\alpha_t(x)\phi_r(x)\diff x = \int_{\Omega_r} T_1(x)+T_2(x)\diff\rho^\alpha_t(x),
\end{equation}
where the gradient and the Laplacian of $\phi_r$ are given by 
    $$
\nabla \phi_r(v) = -2r^2 \frac{v - v^*}{\left( r^2 - |v - v^*|_2^2 \right)^2} \phi_r(v),
$$
$$
\Delta \phi_r(v) = 2r^2 \left( \frac{2 \left( 2|v - v^*|_2^2 - r^2 \right)|v - v^*|_2^2 - d \left( r^2 - |v - v^*|_2^2 \right)^2}{\left( r^2 - |v - v^*|_2^2 \right)^4} \right) \phi_r(v)
$$
respectively, and $\Omega_r= \{x\in\R^d\ |\  |x-x^*|<r\}$. Introducing the following two subsets 
\begin{equation*}
    \begin{aligned}
        &A_1 = \{x\in\R^d\ | \  |x-x^*|> rc^{1/2}\}\\
        &A_2 = \left\{x\in\R^d\ | \ -\lambda(x-\Xalpha(\rho^\alpha_t))^\top(x-x^*)\left(r^2- |x-x^*|^2\right)^2>\frac{\delta^2}{2}(2c-1)r^2 |x-x^*|^2\right\}
    \end{aligned}
\end{equation*}
the integral in \eqref{eq:T1T2} can be rewritten as
\begin{equation}\begin{aligned}
    \frac{\diff}{\diff t}\int\phi_r(x)\diff\rho^\alpha_t(x) &=
    \int_{\Omega_r\cap A_1^c} T_1(x)+T_2(x)\diff\rho^\alpha_t(x) +
    \int_{\Omega_r\cap A_1\cap A_2^c} T_1(x)+T_2(x)\diff\rho^\alpha_t(x) \\& +
    \int_{\Omega_r\cap A_1\cap A_2} T_1(x)+T_2(x)\diff\rho^\alpha_t(x).  
\end{aligned}\end{equation}
We now bound the three integrals.\\
Let us first consider the case $x\in A_1^c\cap\Omega_r$, which is equivalent to $ |x-x^*|\leq rc^{1/2}.$ In this case we have 
$|x-\Xalpha(\rho^\alpha_t)|\leq |x-x^*|+|\Xalpha(\rho^\alpha_t)-x^*|\leq c^{1/2}r+B$
and $r^2- |x-x^*|^2\geq r^2(1-c).$ Using these inequalities in the definition of $T_1$ we get
\begin{equation}
    \begin{aligned}
        T_1(x) &= \frac{2\lambda r^2\phi_r(x)(x-\Xalpha(\rho^\alpha_t))^\top(x-x^*)}{\left(r^2- |x-x^*|^2\right)^2}\geq-\frac{2\lambda r^2\phi_r(x)|x-\Xalpha(\rho^\alpha_t)| |x-x^*|}{\left(r^2- |x-x^*|^2\right)^2}\\
        &\geq-2\frac{\lambda c^{1/2}(c^{1/2}r+B)}{r(1-c)^2}\phi_r(x) = -p_1\phi_r(x).\\
    \end{aligned}
\end{equation}
On the other hand, by definition of $T_2$ we have
\begin{equation}
    \begin{aligned}
        T_2(x) &= -\frac{\delta^2r^2\phi_r(x)}{\left(r^2- |x-x^*|^2\right)^4}\left(d\left(r^2- |x-x^*|^2\right)^2+2 |x-x^*|^2{\left(r^2-2 |x-x^*|^2\right)^2}\right)\\
        &\geq -\frac{\delta(1+c^2)d}{r^2(1-c)^4}\phi_r(x) = -p_2\phi_r(x).\\
    \end{aligned}
\end{equation}
We now prove that if $x\in \Omega_r\cap A_1\cap A_2^c$ then $T_1(x)+T_2(x)\geq 0.$ This is true if and only if 
\begin{equation}\label{eq:A2_T1T2}
    \left(-\lambda(x-\Xalpha(\rho^\alpha_t))^\top(x-x^*)+\frac{d\delta^2}{2}\right)\left(r^2- |x-x^*|^2\right)^2 \leq \delta^2 |x-x^*|^2\left(2 |x-x^*|^2-r^2\right).
\end{equation}
Since $x\in A_2^c$, $x\in A_1$ and the fact that $(1-c)^2d\leq(2c-1)c$ we have 
\begin{equation*}\begin{aligned}
    \left(-\lambda(x-\Xalpha(\rho^\alpha_t))^\top(x-x^*)+\frac{d\delta^2}{2}\right)\left(r^2- |x-x^*|^2\right)^2 &\leq \frac{\delta^2}{2} |x-x^*|^2(2c-1)r^2 + \frac{\delta^2}{2}(2c-1)cr^4\\
    &\leq \delta^2 |x-x^*|^2\left(2 |x-x^*|^2-r^2\right),
\end{aligned}\end{equation*}
which yields \eqref{eq:A2_T1T2}, and thus $T_1(x)+T_2(x)\geq 0$ for $x\in \Omega_r\cap A_1\cap A_2^c$.
Let us now consider the case $x\in \Omega_r\cap A_1\cap A_2$. By definition of $A_2$ and Cauchy Schwartz inequality we have 
\begin{equation}\begin{aligned}
    T_1(x)&\geq-\frac{2\lambda r^2\phi_r(x)}{(r^2- |x-x^*|^2)^2}|x-\Xalpha(\rho^\alpha_t)| |x-x^*|\geq \frac{4\lambda^2|x-\Xalpha(\rho^\alpha_t)|}{(2c-1)\delta^2}(x-\Xalpha(\rho^\alpha_t))^\top(x-x^*)\\
    & \geq -\frac{4\lambda^2r(cr^2+B)}{(2c-1)\delta^2}\phi_r(x) = -p_3\phi_r(x).
\end{aligned}\end{equation}
By definition of $T_2$ and $\Delta\phi_r(x)$ we have that $T_2(x)\geq 0 $ if 
\begin{equation}\label{eq:A3T2}
    2 |x-x^*|^2(2 |x-x^*|^2-r^2)\geq d(r^2-|x-x|^2)^2.
\end{equation}
Since $x\in\Omega_r\mathcal{A}_1\cap \mathcal{A}_2$ we have 
\begin{equation*}
  2 |x-x^*|^2(2 |x-x^*|^2-r^2)\geq 2r^4c(2c-1)\geq dr^4(1-c)^2\geq d(r^2-|x-x|^2)^2.
\end{equation*}
Thus \eqref{eq:A3T2} holds and $T_2(x)\geq 0$.
Putting everything together, we then get 
\begin{equation}\begin{aligned}
    \frac{\diff}{\diff t}\int\phi_r(x)\diff\rho^\alpha_t(x) &=
    \int_{\Omega_r\cap A_1^c} T_1(x)+T_2(x)\diff\rho^\alpha_t(x) +
    \int_{\Omega_r\cap A_1\cap A_2^c} T_1(x)+T_2(x)\diff\rho^\alpha_t(x) \\& +
    \int_{\Omega_r\cap A_1\cap A_2} T_1(x)+T_2(x)\diff\rho^\alpha_t(x)\\
    &\geq -(p_1+p_2)\int_{\Omega_r\cap A_1^c} \phi_r(x)\diff\rho^\alpha_t(x) 
    -p_3\int_{\Omega_r\cap A_1\cap A_2} \phi_r(x)\diff\rho^\alpha_t(x)\\
    &\geq -p\int\phi_r(x)\diff\rho^\alpha_t(x),
\end{aligned}\end{equation}
 with $p=\max\{p_1+p_2,p_3\}$. The thesis then follows from Gronwall's inequality.
\end{proof}

\begin{lemma}\label{lemma:CFtodelta_N}
    Given $\mu_0\in\R^d,\ s,\lambda,\delta,\Delta t>0$, let $\varrho_t^{\alpha,s,\Delta t}$ be the solution of \eqref{eq:flip_scheme_2} with $\rho_0=\mathcal N\left(\mu_0,\frac{\delta^2}{2\lambda}I_d\right).$
    Then $\varrho^{\alpha,s,\Delta t}$ converges to the solution of \eqref{eq:cbo_delta_mf} as $\Delta t$ goes to 0.
    
\end{lemma}
\begin{proof}
    Let $\{\tau_n\}_{n=0}^\infty\subseteq(0,1)$ be a sequence such that $\tau_n\rightarrow0$ as $n\rightarrow +\infty$.
For every $n$ we define 
\begin{equation*}\begin{aligned}
    f^n:[0&,T]\longrightarrow \mathbb P(\R^d)\\
          &t \longmapsto f^n_t\coloneqq \varrho_t^{\alpha,s,\tau_n},
\end{aligned}\end{equation*}
% $f^n:[0,T]\longrightarrow \mathbb P(\R^d)$, 
with $\varrho_t^{\alpha,s,\tau_n}$ solution of \eqref{eq:flip_scheme_2} with $\Delta t=\tau_n$ and 
$\rho_0 = \mathcal N\left(\mu_0, \frac{\delta^2}{2\lambda}I\right)$, for some $\mu_0\in\mathbb R^d.$ 
From \eqref{eq:mu_t}-\eqref{iterative} we have that for every $n\in\mathbb N$ and every $t\in[0,T]$, $f^n_{t} = \varrho_t^{\alpha,s,\tau_n} = \mathcal N\left(u^n(t), \frac{\delta^2}{2\lambda}I\right)$ with 
\begin{equation*}
    \begin{aligned}
        & u^n(0) = \mu_0\\
        & u^n(t) = \mu_t^{\alpha,s,\tau_n}= X^\alpha(f^n_{t_i}) + e^{-s\lambda(t-t_i)} (u^n(t_i)-X^\alpha(f^n_{t_i})) \quad \text{for } t \in [(i-1)\tau_n, i\tau_n).
    \end{aligned}
\end{equation*}

For every $t \in [(i-1)\tau_n, i\tau_n)$ we then have 
$$
\left| \frac{\diff}{\diff t} u^n(t) \right|  = \left|s\lambda e^{-s\lambda(t-t_i)} (X^\alpha(f^n_{t_i}) - u^n(t_i))\right|\leq s\lambda |X^\alpha(f^n_{t_i}) - u^n(t_i)| \leq M $$
for $M>0$ independent on $n$ and $t$.
This implies that $u^n(t)$ is Lipschitz-continuous with constant $M$ on $[(i-1)\tau_n, i\tau_n)$, which in turn implies that it is Lipschitz-continuous with constant $M$ on $[0,T]$; so $u^n$ are equicontinuous.
Using Arzel\`a-Ascoli, we deduce the existence of a subsequence 
$\{u^{n_k}\}_{k\in\mathbb N}\subseteq \{u^n\}_{n\in\mathbb N}$ such that  $u^{n_k}(t) \to \overline{u}(t)$ as $k\to\infty$ uniformly on $[0,T]$.
% (without relabeling) such that $u^n(t) \to \overline{u}(t)$ as $n\to\infty$ uniformly on $[0,T]$ for some $\overline{u}(t) \in \mathbb R^d$.
Next, we show that also $\frac{\diff}{\diff t}u^{n_k}(t) \to \frac{\diff}{\diff t}\overline{u}(t)$; to prove this claim, it suffices to show that $\frac{\diff}{\diff t}u^{n_k}$ converge uniformly in $t$ to some limit $u'$. First we note that, for $t$ fixed and $i_{n_k}$ such that $t \in [(i_{n_k}-1)\tau_{n_k}, i_n\tau_{n_k})$, $|u^{n_k}(t) - u^{n_k}(t_{i_{n_k}}) | \leq M|t-t_{i_{n_k}}| \leq M\tau_{n_k}$, and thus $u^n(t_{i_n}) \to \overline{u}(t)$ uniformly in $t$. By continuity of $\mathcal T^\alpha$ and uniform (in $t$ and $n$) boundedness of $u^{n_k}(t)$, we deduce that also $x_{i_{n_k}}^* = \mathcal T^\alpha(u^{n_k}(t_{i_{n_k}})) \to \mathcal T^\alpha(\overline{u}(t))$ uniformly in $t$. Thus we see that 
\begin{equation*}
    \frac{\diff}{\diff t}u^{n_k}(t) = s \lambda e^{-s \lambda (t - t_{i_{n_k}})} (x_{i_{n_k}}^\star - u^{n_k}(t_{i_{n_k}})) \to s \lambda (\mathcal T^\alpha(\overline{u}(t)) - \overline{u}(t)) \quad \text{as } k \to \infty, \text{ uniformly in } t,
\end{equation*}
which guarantees that $\frac{\diff}{\diff t}u^{n_k} \to \frac{\diff}{\diff t}\overline{u}$. \\
As we have for any $p \geq 1$, any $x, y \in \R^d$ and any everywhere differentiable function $v\colon [0,T] \to \R^d$ that
\begin{align}
    &W_p\big(\mathcal N(x, \sigma^2 I), \mathcal N(y, \sigma^2 I)\big) = |x - y|, \label{eq:Wasserstein_between_gaussians}\\
    &\frac{\diff}{\diff t} \mathcal N(v(t), \sigma^2 I)(x) = -\nabla \mathcal N(v(t), \sigma^2 I)(x) \cdot \frac{\diff}{\diff t} v(t), \label{eq:derivative_of_gaussian}
\end{align}
we immediately see from~\eqref{eq:Wasserstein_between_gaussians} that $f^n \to \big[t \mapsto \mathcal{N}(\overline{u}(t), \frac{\delta^2}{2\lambda}I)\big] \eqqcolon \overline{f}_t$ uniformly on $[0,T]$ in the $p$-Wasserstein distance. From~\eqref{eq:derivative_of_gaussian}, continuity of $\nabla \mathcal N(x, \sigma^2 I)(y)$ in $(x, y)$ and uniform convergence of $u^{n_k}$ and $\frac{\diff}{\diff t}u^{n_k}$ to $\overline{u}$ and $\frac{\diff}{\diff t}\overline{u}$ respectively, we also deduce that
\begin{equation}\label{eq:derivative_convergence}
    \frac{\diff}{\diff t} f^{n_k}_{t} \to \frac{\diff}{\diff t} \overline{f}_t \quad \text{uniformly in } t \in [0,T].
\end{equation}
% Equivalently, 
% \begin{equation*}
%     \varrho^{\alpha,s,\tau_n}_t\to \overline{\varrho}^{\alpha,s}_t,\quad
%         \frac{\diff}{\diff t} \varrho^{\alpha,s,\tau_n}_t\to \frac{\diff}{\diff t} \overline{\varrho}^{\alpha,s}_t \quad\text{as }n\to\infty,
% \end{equation*}
% where $\overline{\varrho}^{\alpha,s}_t= \mathcal{N}(\overline{u}(t), \frac{\delta^2}{2\lambda}I)$ is the density of $\overline{f}_t$.

We now prove that $\overline{f}_t$ satisfies the Fokker-Planck equation for the $\delta$-CBO scheme.
For every $k$ and every $t\in [(i-1)\tau_{n_k}, i\tau_{n_k})$, $\overline{f}_t$ satisfies the Fokker-Planck equation for \eqref{eq:flip_scheme_2}. That is, 
\begin{equation}\label{FP1}
    \frac{\diff}{\diff t}\int_{\mathbb{R}^d}\Phi(x)\overline{f}_t(x)\diff x = -s\lambda \int_{\mathbb{R}^d}f^{n_k}_t(x)\left(x-\Xalpha(f^{n_k}_{(i-1)\tau_{n_k}})\right)^\top \nabla \Phi(x) \diff x + \frac{\delta^2}{2}\int_{\mathbb{R}^d}f^{n_k}_t(x)\Delta\Phi(x)\diff x.
\end{equation}
Using the fact that $f^{n_k}_t\rightarrow\overline{f}_t$ we have
\begin{equation}\label{I1}
   \lim_{k\rightarrow+\infty}\frac{\delta^2}{2}\int_{\mathbb{R}^d}f^{n_k}_t(x)\Delta\Phi(x)\diff x = \frac{\delta^2}{2}\int_{\mathbb{R}^d}\lim_{k\rightarrow+\infty}f^{n_k}_t(x)\Delta\Phi(x)\diff x  = \frac{\delta^2}{2}\int_{\mathbb{R}^d}\overline{f}_t(x)\Delta\Phi(x)\diff x. 
\end{equation}

Using again the convergence of $f^{n_k}_t$ to $\overline{f}_t$ and the continuity of $\Xalpha$, we get
\begin{equation*}\begin{aligned}
    &\lim_{k\rightarrow+\infty}\left|\Xalpha(f^{n_k}_{(i-1)\tau_{n_k}})f^{n_k}_t(x)-\Xalpha(\overline{f}_t)\overline{f}_t(x)\right|\\
    &\leq \lim_{k\rightarrow+\infty} \left|\Xalpha(f^{n_k}_{(i-1)\tau_{n_k}})\right|\cdot|f^{n_k}_t(x)-\overline{f}_t(x)|+\lim_{k\rightarrow+\infty}\left|\Xalpha(f^{n_k}_{(i-1)\tau_{n_k}})-\Xalpha(\overline{f}_t)\right|\\
    &\leq\lim_{k\rightarrow+\infty}\left|\Xalpha(f^{n_k}_{(i-1)\tau_{n_k}})-\Xalpha(\overline{f}_t)\right|\\
    &\leq \lim_{k\rightarrow+\infty}\left|\Xalpha(f^{n_k}_{(i-1)\tau_{n_k}})-\Xalpha(\overline{f}_{(i-1)\tau_{n_k}})\right| + \lim_{k\rightarrow+\infty}\left|\Xalpha(\overline{f}_{(i-1)\tau_{n_k}}) - \Xalpha(\overline{f}_t)\right|= 0,
\end{aligned}\end{equation*}
where we used in the second inequality the fact that $\Xalpha$ is bounded on the compact set $\cup_{k \in \NN} \mathrm{Im}(f^{n_k})$.
We thus get further
\begin{equation}\label{I2}
    \lim_{k\rightarrow+\infty} s\lambda \int_{\mathbb{R}^d}f^{n_k}_t(x)\left(x-\Xalpha(f^{n_k}_{(i-1)\tau_{n_k}})\right)^\top \nabla \Phi(x) \diff x = \int_{\mathbb{R}^d}\overline{f}_t(x)(x-\Xalpha(\overline{f}_t))^\top \nabla \Phi(x) \diff x.
\end{equation}
By \cref{eq:derivative_convergence} we have that 
$$
    \lim_{k\rightarrow+\infty}\frac{\diff}{\diff t}f^{n_k}_t(x) = \frac{\diff}{\diff t}\overline{f}_t(x).
$$

Moreover 
$$\left|\frac{\diff}{\diff t} f^{n_k}_t(x)\Phi(x)\right|\leq |\Phi(x)|\cdot|\nabla f^{n_k}_t(x)|\left |\frac{\diff}{\diff t} u^{n_k}(t)\right|\leq \frac{2\lambda M}{\delta^2}|x-u^{n_k}(t)|\cdot|\Phi(x)|.$$
Since $\Phi(x)$ has compact support, this implies that $\left|\frac{\diff}{\diff t} f^n_t(x)\Phi(x)\right|$ is bounded from above by an integrable function and therefore, by the dominated convergence theorem,

\begin{equation}\label{I3}
\lim_{k\rightarrow+\infty}\frac{\diff}{\diff t}\int_{\mathbb{R}^d}\Phi(x)f^{n_k}_t(x)\diff x = \frac{\diff}{\diff t}\int_{\mathbb{R}^d}\Phi(x)\overline{f}_t(x)\diff x.
\end{equation}
Taking the limit as $k$ that goes to infinity on both sides of \eqref{FP1}, by \eqref{I1}, \eqref{I2}, \eqref{I3} we have
$$
\frac{\diff}{\diff t}\int_{\mathbb{R}^d}\Phi(x)\overline{f}_t(x)\diff x = -s\lambda \int_{\mathbb{R}^d}\overline{f}_t(x)(x-\Xalpha\left(\overline{f}_t)\right)^\top \nabla \Phi(x) \diff x + \frac{\delta^2}{2}\int_{\mathbb{R}^d}\overline{f}_t(x)\Delta\Phi(x)\diff x,$$
which is the Fokker-Planck equation for $\delta$-CBO.
\end{proof}

\section{AI Methodology}
Following the template available at \url{https://simons.berkeley.edu/media/31229/download}, Table \ref{AI_tab} and the paragraph below describe how AI was involved in the research that led to this work and the preparation of this manuscript.
\begin{table}[H]
\begin{center}
\renewcommand{\arraystretch}{1.3}
\begin{tabular}{@{}l@{\qquad}c@{}}
\toprule
\textbf{Activity} & \textbf{AI used} \\
\midrule
Asking the research question & NO \\
Coming up with the approach & NO \\
Proof development & NO \\
Writing and exposition & NO \\
Checking for bugs & YES \\
Other supporting tasks & YES \\
\bottomrule
\end{tabular}
\end{center}
\caption{AI Methodology}\label{AI_tab}
\end{table}
% 1. Asking the research question: proposing a model, a definition, an assumption or a conjecture.
% 2. Coming up with the approach: the organizing plan of the proof, such as a reduction, a relaxation or a sequence of intermediate claims.
% 3. Proof development: the important technical ideas, those one would highlight in a 50-minute talk on the work.
% 4. Writing and exposition: substantive explanation and presentation, such as a proof overview or an example.
% 5. Checking for bugs: finding or correcting consequential errors, such as a circular argument or a missing hypothesis.
% 6. Other supporting tasks: copy-editing, figures, numerical experiments, minor lemmas, literature search.
% For the first four categories, answer Yes if a human making a comparable contribution would have been offered co-authorship; for the last two, any substantial contribution counts.

%\begin{sloppypar}
{\bf Details}
Claude Opus 5.5 was used exclusively for polishing the paper and literature search.
%\end{sloppypar}
\printbibliography
\bigskip
\begin{center}
  \FundingLogos
  
  \vspace{0.5em}
  \begin{tcolorbox}\centering\small
    % CHOOSE the sentence that applies to your grant/programme and replace placeholders.
    % 1) For Horizon Europe / ERC (recommended wording):
    Funded by the European Union. Views and opinions expressed are however those of the author(s) only and do not necessarily reflect those of the European Union or the European Research Council Executive Agency. Neither the European Union nor the granting authority can be held responsible for them. This project has received funding from the European Research Council (ERC) under the European Union’s Horizon Europe research and innovation programme (grant agreement No. 101198055, project acronym NEITALG).
    
    % 2) (Alternatively, for Horizon 2020 ERC grants use the H2020 wording:)
    % This project has received funding from the European Research Council (ERC) under the European Union's Horizon 2020 research and innovation programme (grant agreement No. <GRANT NUMBER>, project acronym <ACRONYM>).
  \end{tcolorbox}
\end{center}
\end{document}